\documentclass{article}
\usepackage{amsmath}
\usepackage{amsthm}

\usepackage{textcomp, arcs}
\usepackage{times}
\usepackage{xcolor}
\usepackage{enumitem}
\usepackage{xparse,etoolbox}
\usepackage{caption}
\usepackage{chngcntr}
\counterwithin{figure}{section}
\counterwithin{table}{section}
\usepackage{tikz}
\usepackage{graphicx}
\usepackage[mathscr]{eucal}
\usepackage{scrextend}
\usepackage{wasysym}
\usepackage{etoolbox}
\makeatletter

\usepackage[all]{xy}

\usepackage{etoolbox}
\makeatletter
\patchcmd{\@makechapterhead}{50\p@}{\chapheadtopskip}{}{}

\patchcmd{\@makeschapterhead}{50\p@}{\chapheadtopskip}{}{}
\makeatother
\newlength{\chapheadtopskip}
\usepackage[parfill]{parskip} 
\usepackage{amssymb}
\usepackage{imakeidx}
\usepackage{enumitem}
\usepackage[mathlines]{lineno}
\DeclareMathAlphabet{\mathpzc}{OT1}{pzc}{m}{it}
\makeatletter
\newcommand{\mylabel}[2]{#2\def\@currentlabel{#2}\label{#1}}
\makeatother
\usepackage{fancyhdr}
\usepackage{enumitem}
\usepackage{xpatch}
\newtheorem{theorem}{Theorem}[section]
\newtheorem{lemma}[theorem]{Lemma}
\newtheorem{obs}[theorem]{Observation}
\newtheorem{defn}[theorem]{Definition}
\newtheorem{prop}[theorem]{Proposition}
\newtheorem{cor}[theorem]{Corollary}

\newtheorem{Claim}[theorem]{Claim}

\newtheorem{subclaim}{Subclaim}[theorem]
\newcounter{claimlevel}[theorem]

\ExplSyntaxOn
\NewDocumentEnvironment{claim}{O{=}}
 {
  \str_case:nn { #1 }
   {
    {=}  { }
    {+}  { \stepcounter{claimlevel} }
    {-}  { \addtocounter{claimlevel}{-1} }
   }
  \begin{ Claim \int_to_Roman:n { \value{claimlevel} } }
 }
 {
  \end{ Claim \int_to_Roman:n { \value{claimlevel} } }
 }
\ExplSyntaxOff

\newenvironment{claimproof}[1]{\par\noindent\underline{Proof:}\space#1}{\hspace{1mm}$\blacksquare$}

\usepackage[left=2.54cm,top=2.54cm,right=2.54cm,bottom=2.54cm]{geometry}
\usepackage[nodisplayskipstretch]{setspace}
\newlength\FHoffset
\fancyheadoffset{\FHoffset}
\newlength\FHright
\makeatletter
\usepackage{framed}

 \newtheoremstyle{TheoremNum}
        {\topsep}{\topsep}              
        {\itshape}                      
        {}                              
        {\bfseries}                     
        {.}                             
        { }                             
        {\thmname{#1}\thmnote{ \bfseries #3}}
    \theoremstyle{TheoremNum}
    \newtheorem{thmn}{Theorem}

\newtheoremstyle{PropNum}
        {\topsep}{\topsep}              
        {\itshape}                      
        {}                              
        {\bfseries}                     
        {.}                             
        { }                             
        {\thmname{#1}\thmnote{ \bfseries #3}}
    \theoremstyle{PropNum}

\newtheoremstyle{LemmaNum}
        {\topsep}{\topsep}              
        {\itshape}                      
        {}                              
        {\bfseries}                     
        {.}                             
        { }                             
        {\thmname{#1}\thmnote{ \bfseries #3}}
    \theoremstyle{LemmaNum}
    
\makeatletter
\renewcommand\subitem{\@idxitem\nobreak\hspace*{20\p@}}
\renewcommand\subsubitem{\@idxitem\nobreak\hspace*{20\p@}}
\makeatother

\date{}
\usepackage{titlesec}
\titleformat{\subsection}{\normalfont\large\bfseries}{Subsection \thesubsection}{1em}{}
\titleformat{\section}{\normalfont\Large\bfseries}{Section \thesection}{1em}{}
\makeatother
\title{Some Extensions of Thomassen's Theorem to Longer Paths}
\author{Joshua Nevin}
\begin{document}
\maketitle

\begin{center}\textbf{Abstract}\end{center} Let $G$ be a planar embedding with list-assignment $L$ and outer cycle $C$, and let $P$ be a path of length at most four on $C$, where each vertex of $G\setminus C$ has a list of size at least five and each vertex of $C\setminus P$ has a list of size at least three. In this paper, we prove some results about partial $L$-colorings $\phi$ of $C$ with the property that any extension of $\phi$ to an $L$-coloring of $\textnormal{dom}(\phi)\cup V(P)$ extends to $L$-color all of $G$. We use these results in a later sequence of papers to prove some results about list-colorings of high-representativity embeddings on surfaces. 

\section{Introduction and Motivation}\label{IntroMotivSec}

All graphs in this paper have a finite number of vertices. Given a graph $G$, a \emph{list-assignment} for $G$ is a family of sets $\{L(v): v\in V(G)\}$, where each $L(v)$ is a finite subset of $\mathbb{N}$. The elements of $L(v)$ are called \emph{colors}. A function $\phi:V(G)\rightarrow\bigcup_{v\in V(G)}L(v)$ is called an \emph{$L$-coloring of} $G$ if $\phi(v)\in L(v)$ for each $v\in V(G)$, and $\phi(x)\neq\phi(y)$ for any adjacent vertices $x,y$. Given an $S\subseteq V(G)$ and a function $\phi: S\rightarrow\bigcup_{v\in S}L(v)$, we call $\phi$ an \emph{ $L$-coloring of $S$} if $\phi$ is an $L$-coloring of the induced graph $G[S]$. A \emph{partial} $L$-coloring of $G$ is an $L$-coloring of an induced subgraph of $G$. Likewise, given an $S\subseteq V(G)$, a \emph{partial $L$-coloring} of $S$ is a function $\phi:S'\rightarrow\bigcup_{v\in S'}L(v)$, where $S'\subseteq S$ and $\phi$ is an $L$-coloring of $S'$. Given an integer $k\geq 1$, $G$ is called \emph{$k$-choosable} if it is $L$-colorable for every list-assignment $L$ for $G$ such that $|L(v)|\geq k$ for all $v\in V(G)$. The motivation for the work of this paper is as follows. In a subsequent series of papers, we prove that, given a graph $G$ embedded on a surface of genus $g$, $G$ can be $L$-colored, where $L$ is a list-assignment for $G$ in which every vertex has a 5-list except for a collection of pairwise far-apart components, each precolored with an ordinary 2-coloring, as long as the face-width of $G$ is at leat $2^{\Omega(g)}$ and the precolored components are of distance at least $2^{\Omega(g)}$ apart. This provides an affirmative answer to a generalized version of a conjecture of Thomassen and also generalizes the results of  \cite{DistPrecVertChoosePap} of Dvo\v{r}\'ak, Lidick\'y, Mohar, and Postle about distant precolored vertices. It can also be viewed as a ``high-representativity" analogue to the result of \cite{HyperbolicFamilyColorSurfPap} about distant precolored vertices on embeddings of high edge-width and as a generalization of Theorem \ref{thomassen5ChooseThm} in which a) arbitrarily many faces, rather than just one face, are permitted to have 3-lists  and b) the embedding $G$ is not planar but rather locally planar in the sense that it has high representativity. We prove the above result by a minimal counterexample argument in which we color and delete a subgraph of a minimal counterexample $G$, where this subgraph consists mostly of a path between two of the faces of $G$ which have 3-lists. The trickiest part of this is dealing with the region near the boundaries, i.e near these faces with 3-lists. This requires some results about planar graphs, which we prove in this paper. 

Given a graph $G$ with list-assignment $L$ and a partial $L$-coloring $\phi$ of $G$, we frequently analyze the situation where we delete the vertices of $\textnormal{dom}(\phi)$ and remove the colors of the deleted vertices from the lists of their remaining neighbors. We thus make the following definition. 
\begin{defn}\label{ListDefnSInNotV}\emph{Let $G$ be a graph, with list-assignment $L$. Let $\phi$ be a partial $L$-coloring of $G$. We define a list-assignment $L_{\phi}$ for $G\setminus\textnormal{dom}(\phi)$, where $L{\phi}(v):=L(v)\setminus\{\phi(w): w\in N(v)\cap\textnormal{dom}(\phi)\}$ for each $v\in V(G)\setminus \textnormal{dom}(\phi)$.} \end{defn}

 Our main result for this paper is Theorem \ref{MainHolepunchPaperResulThm} below. 

\begin{theorem}\label{MainHolepunchPaperResulThm} Let $G$ be a planar embedding with outer cycle $C$, $L$ be a list-assignment for $V(G)$, and $P:=p_0q_0zq_1p_1$ be a subpath of $C$, where $q_0, q_1$ have no common neighbor in $C\setminus P$, and
\begin{enumerate}[label=\arabic*)] 
\itemsep-0.1em
\item Each of $L(p_0)$ and $L(p_1)$ is nonempty and one of them has a list of size at least three; AND
\item For each $v\in V(C\setminus P)$, $|L(v)|\geq 3$, and, for each $v\in\{q_0, z, q_1\}\cup V(G\setminus C)$, $|L(v)|\geq 5$.
\end{enumerate}
Then there is a partial $L$-coloring $\phi$ of $V(C)\setminus\{q_0, q_1\}$, where $p_0, p_1, z\in\textnormal{dom}(\phi)$, such that each of $q_0, q_1$ has an $L_{\phi}$-list of size at least three, and any extension of $\phi$ to an $L$-coloring of $\textnormal{dom}(\phi)\cup\{q_0, q_1\}$ extends to $L$-color all of $G$. 
 \end{theorem}

In the setting above, it is not necessarily true that there is an $L$-coloring $\phi$ of $\{p_0, z, p_1\}$ such that any extension of $\phi$ to an $L$-coloring of $V(P)$ extends to $L$-color $G$, but Theorem \ref{MainHolepunchPaperResulThm} tells us that something almost as good is true: There is a partial $L$-coloring $\phi$ of $C\setminus\{q_0, q_1\}$ containing $\{p_0, z, p_1\}$ in its domain where, for each $k=0,1$, $\phi$ uses at most two colors of $L(q_k)$ among the neighbors of $q_k$. To prove Theorem \ref{MainHolepunchPaperResulThm}, we need some intermediate facts about paths of lengths 2 and 3 in facial cycles of planar graphs. In particular, we need Theorems \ref{SumTo4For2PathColorEnds} and \ref{CombinedT1T4ThreePathFactListThm}. We prove Theorem \ref{SumTo4For2PathColorEnds} in Section \ref{ExtCol2PathAugColSec} and prove Theorem \ref{CombinedT1T4ThreePathFactListThm} over the course of Sections \ref{PrFPart1T1T4}-\ref{CorColSecRes}. We then prove Theorem \ref{MainHolepunchPaperResulThm} in Section \ref{MainResHolepunchSec}. In Section \ref{ConseqFinSec}, the final section of this paper, we prove two results about 5-paths that are analogues of Theorem \ref{MainHolepunchPaperResulThm}. These are Theorems \ref{ModifiedRes5ChordCaseDegen} and \ref{RainbowNonEqualEndpointColorThm}. Roughly speaking, in a later paper, we use Theorems \ref{MainHolepunchPaperResulThm}, \ref{ModifiedRes5ChordCaseDegen}, and \ref{RainbowNonEqualEndpointColorThm} as black boxes to show the following: If we have a graph $G$ embedded on a surface $\Sigma$, a family $\mathcal{C}$ of pairwise far-apart null-homotopic facial cycles and a list-assignment $L$ for $V(G)$ where all the vertices of $\bigcup\mathcal{C}$ have lists of size at least three and all the other vertices have lists of size at least five, then, under certain conditions, for any $\mathcal{C}'\subseteq\mathcal{C}$ with $|\mathcal{C}'|\geq 2$, we can color and delete a connected subgraph $H$ of $G$ with an $L$-coloring $\phi$ of $V(H)$ such that $H$ has nonempty intersection with all the elements of $\mathcal{C}'$ and is far away from all the elements of $\mathcal{C}\setminus\mathcal{C}'$, and almost all the vertices of $D_1(H)$ have $L_{\phi}$-lists of size at least three. We then use this to generalize the result of \cite{DistPrecVertChoosePap} to distant 2-colored components on embeddings. We now introduce our main object of study for this paper. 

\begin{defn} \emph{A \emph{rainbow} is a tuple $(G, C, P, L)$, where $G$ is a planar graph with outer cycle $C$, $P$ is a path on $C$ with $|E(P)|\geq 1$, and $L$ is a list-assignment for $V(G)$ such that each endpoint of $P$ has a nonempty list and furthermore, $|L(v)|\geq 3$ for each $v\in V(C\setminus P)$ and $|L(v)|\geq 5$ for each $v\in V(G\setminus C)$. Letting $p, p^*$ be the endpoints of $P$, we say the rainbow is \emph{end-linked} if $|L(p)|+|L(p^*)|\geq 4$.}
 \end{defn}

Given the statement of Theorem \ref{MainHolepunchPaperResulThm}, it is also natural to introduce the following terminology.

\begin{defn}\label{GeneralAugCrownNotForLink}  \emph{Given a graph $G$ with list-assignment $L$, we define the following. }
\begin{enumerate}[label=\emph{\arabic*)}]
\itemsep-0.1em
\item \emph{For any subgraph $H$ of $G$ and  partial $L$-coloring $\phi$ of $G$, we say that $\phi$ is \emph{$(H,G)$-sufficient} if any extension of $\phi$ to an $L$-coloring of $\textnormal{dom}(\phi)\cup V(H)$ extends to $L$-color all of $G$.}
\item\emph{Given a path $P\subseteq C$ of length at least two, where $P$ has endpoints $p, p'$ and terminal edges $pq, p'q'$, we let $\textnormal{Crown}_{L}(P, G)$ be the set of $(P,G)$-sufficient partial $L$-colorings $\phi$ of $V(C)\setminus\{q, q'\}$ such that}
\begin{enumerate}[label=\emph{\alph*)}]
\itemsep-0.1em
\item\emph{$p, p'\in\textnormal{dom}(\phi)$ and, if $|E(P)|>3$, then at least one vertex of $\mathring{P}\setminus\{q, q'\}$ lies in $\textnormal{dom}(\phi)$}; AND
\item\emph{For each $x\in V(P)\setminus\textnormal{dom}(\phi)$, $|L_{\phi}(x)|\geq 3$.} 
\end{enumerate}
\end{enumerate}
 \end{defn}

We usually drop the subscript $L$ or the coordinate $G$ from the notation above (or both) if these are clear from the context. With the definitions above in hand, we have the following compact restatement of Theorem \ref{MainHolepunchPaperResulThm}.

\begin{thmn}[\ref{MainHolepunchPaperResulThm}] Let $(G, C, P, L)$ be a rainbow, where $P$ has length four, the endpoints of $\mathring{P}$ have no common neighbor in $C\setminus P$, each internal vertex of $P$ has an $L$-list of size at least five, and at least one endpoint of $P$ has a list of size at least three. Then $\textnormal{Crown}_{L}(P, G)\neq\varnothing$. \end{thmn} 

We now provide some background and state some results that we use later. In \cite{AllPlanar5ThomPap}, Thomassen showed that all planar graphs are 5-choosable. Actually, Thomassen proved something stronger. 

\begin{theorem}\label{thomassen5ChooseThm}
Let $G$ be a planar graph with facial cycle $C$ and $xy\in E(C)$. Let $L$ be a list assignment for $G$, where each vertex of $G\setminus C$ has a list of size at least five, each vertex of $C\setminus\{x,y\}$ has a list of size at least three, and $xy$ is $L$-colorable. Then $G$ is $L$-colorable.
\end{theorem}

Theorem \ref{thomassen5ChooseThm} has the following useful corollary, which we use frequently. 

\begin{cor}\label{CycleLen4CorToThom} Let $G$ be a planar graph with outer cycle $C$ and let $L$ be a list-assignment for $G$ where each vertex of $G\setminus C$ has a list of size at least five.  If $|V(C)|\leq 4$, then any $L$-coloring of $V(C)$ extends to an $L$-coloring of $G$. \end{cor}

Because of Corollary \ref{CycleLen4CorToThom}, planar embeddings which have no separating cycles of length 3 or 4 play a special role in our analysis, so we give them a name.

\begin{defn} \emph{Given a planar graph $G$, we say that $G$ is \emph{short-inseparable} if, for any cycle $F\subseteq G$ with $|V(F)|\leq 4$, either $V(\textnormal{Int}_G(F))=V(F)$ or $V(\textnormal{Ext}_G(F))=V(F)$.} \end{defn}

We also rely on the following very result from Postle and Thomas, which is an analogue of Theorem \ref{thomassen5ChooseThm} where the precolored edge has been replaced by two lists of size two. 

\begin{theorem}\label{Two2ListTheorem} Let $G$ be a planar graph with outer face $F$, and let $v,w\in V(F)$. Let $L$ be a list-assignment for $V(G)$ where $|L(v)|\geq 2$, $|L(w)|\geq 2$, and furthermore, for each $x\in V(F)\setminus\{v,w\}$, $|L(x)|\geq 3$, and, for each $x\in V(G\setminus F)$, $|L(x)|\geq 5$. Then $G$ is $L$-colorable. \end{theorem}

Lastly, in this paper, we rely on the following simple but useful result, which is a consequence of a characterization in \cite{lKthForGoBoHm6} of obstructions to extending a precoloring of a cycle of length at most six in a planar graph. 

\begin{theorem}\label{BohmePaper5CycleCorList} Let $G$ be a short-inseparable planar embedding with facial cycle $C$, where $5\leq |V(C)|\leq 6$ . Let $G'=G\setminus C$ and $L$ be a list-assignment for $G$, where $|L(v)|\geq 5$ for all $v\in V(G')$. Let $\phi$ be an $L$-coloring of $V(C)$ which does not extend to $L$-color $G$. If $|V(C)|=5$, then $G'$ consists of a lone vertex $v$ adjacent to all five vertices of $C$, where $L_{\phi}(v)=\varnothing$. On the other hand, if $|V(C)|=6$, then $G'$ consists of one of the following.
\begin{enumerate}[label=\roman*)]
\itemsep-0.1em
\item A lone vertex $v$ adjacent to at least five vertices of $C$, where $L_{\phi}(v)=\varnothing$; OR
\item An edge where, for each $v\in V(G')$, $L_{\phi}(v)$ is the same 1-list and $G[N(v)\cap V(C)]$ is a length-three path; OR
\item A triangle where, for each $v\in V(G')$, $L_{\phi}(v)$ is the same 2-list and $G[N(v)\cap V(C)]$ is a length-two path.
\end{enumerate}
\end{theorem}

\section{Extending Colorings of 2-Paths}\label{2PathBWheelCaseSec}

In this section, we take note of some facts that we need for the proof of Theorem \ref{SumTo4For2PathColorEnds} and elsewhere. 

\begin{defn}\emph{A \emph{broken wheel} is a graph $G$ with a vertex $p$ such that $G-p$ is a path $q_1\ldots q_n$ with $n\geq 2$, where $N(p)=\{q_1, \ldots, q_n\}$. The subpath $q_1pq_n$ of $G$ is called the \emph{principal path} of $G$. On the other hand, a \emph{wheel} is a graph $G$ with a vertex $p$ (called the \emph{central vertex}) such that $G-p$ is a cycle and $p$ is adjacent to all the vertices of $G-p$.} \end{defn}

Note that, if $|V(G)|\leq 4$, then the above definition does not uniquely specify the principal path, although in practice, whenever we deal with broken wheels, we specify the principal path beforehand so that there is no ambiguity. Furthermore, since we frequently deal with colorings of paths, we introduce the following natural notation. 

\begin{defn} \emph{Let $G$ be a graph with list-assignment $L$. Let $P:=p_1\ldots p_n$ be path. Given a partial $L$-coloring $\phi$ of $G$ with $V(P)\subseteq\textnormal{dom}(\phi)$, we denote $\phi|_P$ by  $(\phi(p_1), \phi(p_2),\ldots, \phi(p_n))$. If $n=3$, then, for each $(c, c')\in L(p_1)\times L(p_3)$, we let $\Lambda_{G,L}^P(c, \bullet, c')$ be the set of $d\in L(p_2)$ such that there is an $L$-coloring of $G$ using $c,d,c'$ on $p_1, p_2, p_3$ respectively. Likewise, given a pair $(c, c')$ in $L(p_1)\times L(p_2)$ or $L(p_2)\times L(p_3)$ respectively, we define the sets $\Lambda_{G,L}^P(c, c', \bullet)$  and $\Lambda_{G,L}^P(\bullet, c, c')$  analogously.}  \end{defn}

The notation above always requires us to specify an ordering of the vertices of the 2-path. Sometimes we make this explicit by writing $\Lambda^{p_1p_2p_3}_{G,L}(\cdot, \cdot, \cdot)$. Whenever any of $P, G, L$ are clear from the context, we drop the respective super- or subscripts from the notation above. 

\begin{defn}\label{GUniversalDefinition} \emph{Let $G$ be a graph with list-assignment $L$ and $P:= p_0qp_1$ be a 2-path in $G$. Given an $a\in L(p_0)$, we say that $a$ is \emph{$(P, G)$-universal} if, for each $b\in L(q)\setminus\{a\}$ and $c\in L(p_1)\setminus\{b\}$, the $L$-coloring $(a,b,c)$ of $p_0qp_1$ extends to $L$-color $G$. We say that $a$ is \emph{almost $(P, G)$-universal}, if, for each $b\in L(q)\setminus\{a\}$, the $L$-coloring $(a,b)$ of $p_0p_1$ extends to a family of at least $|L(p_1)|-1$ different $L$-colorings of $G$, each using a different color on $p_1$. Note that, if $p_0p_1\in E(G)$ and $L(p_0)\subseteq L(p_1)$, then no color of $L(p_0)$ is $(P, G)$-universal.  }
 \end{defn}

\begin{theorem}\label{BWheelMainRevListThm2} Let $G$ be a broken wheel with principal path $P=pp'p''$ and let $L$ be a list-assignment for $V(G)$ in which each vertex of $V(G)\setminus\{p, p'\}$ has a list of size at least three. Let $G-p'=pu_1\ldots u_tp''$ for some $t\geq 0$. 
\begin{enumerate}[label=\arabic*)]
\item\label{BWheel1Lb} Let $\phi_0, \phi_1$ be a pair of distinct $L$-colorings of $pp'$. For each $i=0,1$, let $S_i:=\Lambda_G(\phi_i(p), \phi_i(p'), \bullet)$, and suppose that $|S_0|=|S_1|=1$. Then the following hold. 
\begin{enumerate}[label=\Alph*)]
\itemsep-0.1em
\item\label{BWheel1A} If $\phi_0(p)=\phi_1(p)$ and $S_0=S_1$, then $|E(G-p')|$ is even; AND
\item\label{BWheel1B} If $\phi_0(p)=\phi_1(p)$ and $S_0\neq S_1$, then $|E(G-p')|$ is odd and, for each $i=0,1$, $S_i=\{\phi_{1-i}(p')\}$; AND
\item\label{BWheel1C} If $\phi_0(p)\neq\phi_1(p)$ and $S_0=S_1$, then $|E(G-p')|$ is odd and $(\phi_0(p), \phi_0(p'))=(\phi_1(p'), \phi_1(p))$
\end{enumerate}
\item\label{BWheel2Lb} Let $q\in\{p, p'\}$ and $\mathcal{F}$ be a family of $L$-colorings of $pp'$, where $|\mathcal{F}|\geq 3$ and $\mathcal{F}$ is constant on $q$. Then there is a $\phi\in\mathcal{F}$ such that $|\Lambda_G(\phi(p), \phi(p'), \bullet)|\geq 2$.
\item\label{BWheel3Lb} Suppose $|V(G)|\geq 4$, let $x$ be the unique vertex of distance two from $p$ on the path $G-p'$, and let $a\in L(p)$ with $L(u_1)\setminus\{a\}\not\subseteq L(x)$. Then $a$ is almost $G$-universal and furthermore, if $|V(G)|>4$, then $a$ is $G$-universal. 
\end{enumerate} 
 \end{theorem}

\begin{proof}  Let $z$ be the unique common neighbor of $p, p'$. Both 1) and 2) are straightforward to check, so now we prove 3). Suppose first that $|V(G)|>4$, but $a$ is not $G$-universal. Note that $x=u_2$ and $pp''\not\in E(G)$, so there is an $L$-coloring $\phi$ of $V(P)$ which does not extend to $L$-color $G$, where $\phi(p)=a$. Let $b=\phi(p')$. Now, $b\in L(u_1)\cap L(u_2)$ and $a\in L(u_1)$, and furthermore, $|L(u_1)|=|L(u_2)|=3$. By Theorem \ref{thomassen5ChooseThm}, the $L$-coloring $(\phi(p'), \phi(p''))$ of $p'p''$ extends to an $L$-coloring $\psi$ of $V(G)\setminus\{p, u_1\}$. As $L(u_1)\setminus\{a\}\not\subseteq L(u_2)$, the lone color of $L(u_1)\setminus\{a, b\}$ is distinct from $\psi(u_2)$. Thus, $\phi\cup\psi$ extends to an $L$-coloring of $G$, a contradiction, so $a$ is $G$-universal. Now suppose $|V(G)|=4$. Thus, $x=p''$. Suppose $a$ is not almost $G$-universal. As $|L(p'')|\geq 3$, there is a $b\in L(p')\setminus\{a\}$ with $|\Lambda_G(a, b, \bullet)|<2$, so $a\in L(u_1)$ and $b\in L(u_1)\cap L(p'')$. Since $L(u_1)\setminus\{a\}\not\subseteq L(p'')$, there is a $c\in L(u_1)\setminus\{a\}$ with $c\not\in L(p'')$, so $c\neq b$. Coloring $u_1$ with $c$, each color of $L(p'')\setminus\{b\}$ is left over for $p''$, so $|\Lambda_G(a, b, \bullet)|\geq 2$, a contradiction.  \end{proof}

Given a planar embedding $G$, a facial cycle $C$ of $G$, an integer $k\geq 1$ and a $k$-chord $Q$ of $C$, there is a natural way to talk about one or the other ``side" of $Q$ in $G$, which is made precise below. 

\begin{defn}\label{ContractNatCQPartChordDefn} \emph{Let $G$ be a planar graph with outer cycle $C$, let $k\geq 1$ be an integer, and let $Q$ be a $k$-chord of $C$. Let $C_0$ and $C_1$ be the two cycles of $C\cup Q$ which contain $Q$. The unique \emph{natural $Q$-partition} of $G$ is the pair $\{G_0, G_1\}$ of subgraphs of $G$ such that $G_i=\textnormal{Int}_G(C_i)$ for each $i=0,1$. In particular, $G=G_0\cup G_1$ and $G_0\cap G_1=Q$.}
\end{defn}

We also introduce the following pieces of notation, the first of which is standard. 

\begin{defn}\label{StandQD} \emph{Given a path $Q$, we let $\mathring{Q}$ denote the path obtained from $Q$ by deleting its endpoints. In particular, if $|E(Q)|\leq 2$, then $\mathring{Q}=\varnothing$. For any $x,y\in V(Q)$, we let $xQy$ denote the subpath of $Q$ with endpoints $x, y$.} \end{defn}

In proving the main results of this paper, we make frequent use of the following observation, which is straightforward to check using Theorems \ref{thomassen5ChooseThm} and \ref{Two2ListTheorem}. 

\begin{lemma}\label{PartialPathColoringExtCL0} Let $(G, C, P, L)$ be a rainbow and $\phi$ be a partial $L$-coloring of $V(P)$ which does not extend to $L$-color $G$, where $\textnormal{dom}(\phi)$ contains the endpoints of $P$. Then either there is a chord of $C$ with one endpoint in $\textnormal{dom}(\phi)\cap V(\mathring{P})$ and the other endpoint in $C\setminus P$, or there is a $v\in V(G\setminus C)\cup (V(\mathring{P})\setminus\textnormal{dom}(\phi))$ with $|L_{\phi}(v)|\leq 2$.
\end{lemma}

We also rely on the following useful result, which follows from Lemma 1 and Theorem 3 of \cite{ExpManyColChooseThomPap}. 

\begin{theorem}\label{EitherBWheelOrAtMostOneColThm} Let $(G, C, P, L)$ be a rainbow, where $P=p_0qp_1$. Suppose there is more than one $L$-coloring of $V(P)$ which does not extend to $L$-color $G$. Then there is a path $Q$ with $V(Q)\subseteq V(C-q)$, where $Q$ has endpoints $p_0, p_1$ and every vertex of $Q$ is adjacent to $q$. In particular, if $G$ is short-inseparable and every chord of $C$ has $q$ as an endpoint, then $G$ is a broken wheel with principal path $P$. 
\end{theorem}

Combining Theorems \ref{BWheelMainRevListThm2} and \ref{EitherBWheelOrAtMostOneColThm}, we obtain the following consequence. 

\begin{prop}\label{CorMainEitherBWheelAtM1ColCor}  Let $(G, C, P, L)$ be a rainbow, where $P:=p_0qp_1$, every chord of $C$ is incident to $q$, and $G$ is short-inseparable.  Suppose further that either $|V(C)|>3$ or $G=C$. Let $\mathcal{F}$ be the set of $L$-colorings of $V(P)$ which do not extend to $L$-color $G$ and let $i\in\{0,1\}$. Then:
\begin{enumerate}[label=\arabic*)]
\itemsep-0.1em
\item\label{PropCor1} If $|L(p_i)|\geq 2$ and no color of $L(p_i)$ is $(P, G)$-universal then $G$ is a broken wheel with principal path $P$, where either $|V(G)|\leq 4$ or, letting $p_ixy$ be the 2-path of $C-q$ with $p_i$ as an endpoint, $L(p_i)\subseteq L(x)\cap L(y)$; AND 
\item\label{PropCor2}  Given $\psi, \psi'\in\mathcal{F}$ using the same color on $p_i$ and different colors on $p_{1-i}$, we have $\{\psi(q), \psi'(q)\}=\{\psi(p_{1-i}), \psi'(p_{1-i})\}$ and $\mathcal{F}=\{\psi, \psi'\}$, and $G$ is a broken wheel with principal path $P$, where $G-q$ has even length; AND
\item\label{PropCor3}  If $\phi$ is an $L$-coloring of $\{p_0, p_1\}$ and $S\subseteq L_{\phi}(q)$ with $|S|\geq 2$ and $S\cap\Lambda_G^P(\phi(p_0), \bullet, \phi(p_1))=\varnothing$, then,
\begin{enumerate}[label=\roman*)]
\itemsep-0.1em
\item $|S|=2$ and $G$ is a broken wheel with principal path $P$, where $G-q$ is a path of odd length; AND
\item  For any $\psi\in\mathcal{F}$,  either $\psi(p_0)=\psi(p_1)=s$ for some $s\in S$, or $\psi, \phi$ restrict to the same coloring of $\{p_0, p_1\}$.
\end{enumerate}
\item\label{PropCor4} If $\mathcal{F}\neq\varnothing$ and $C$ is induced, $G$ is not a broken wheel, then $G$ is a wheel and its outer cycle has odd length. 
\end{enumerate} \end{prop}

\begin{proof} \ref{PropCor1} and \ref{PropCor2} follow from Theorems \ref{BWheelMainRevListThm2} and \ref{EitherBWheelOrAtMostOneColThm}.  For \ref{PropCor3}, it follows from Theorem \ref{EitherBWheelOrAtMostOneColThm} that $G$ is a broken wheel with principal path $P$, and the rest follows from a simple parity argument. For \ref{PropCor4}, since $G$ is not a broken wheel, it follows from Lemma \ref{PartialPathColoringExtCL0} that there is a $w\in V(G\setminus C)$ adjacent to all three vertices of $P$, and then it follows from Theorem \ref{EitherBWheelOrAtMostOneColThm} applied to $G-q$ that $G-q$ is a broken wheel with principal path $p_0wp_1$, where $G\setminus\{q, w\}$ is a path of even length, so $w$ is adjacent to all the vertices of $C$ and $C$ has odd length.  \end{proof}

\section{Extending Colorings of the Endpoints of 2-Paths}\label{ExtCol2PathAugColSec}

Throughout this paper, we repeatedly construct a smaller rainbow from a given rainbow as specified below: 

\begin{defn} \emph{Let $\mathcal{G}=(G,C,P,L)$ be a rainbow and let $Q$ be either a path or a cycle, where $Q$ intersects with $C\setminus\mathring{P}$ on either one or two vertices and furthermore, if $|V(Q)\cap V(C\setminus\mathring{P})|=2$, then $Q$ is a path intersecting with $C\setminus\mathring{P}$ precisely on the endpoints of $Q$. We then define the following:}
\begin{enumerate}[label=\emph{\arabic*)}]
\itemsep-0.1em
\item\emph{If $V(Q)\cap V(C\setminus\mathring{P}) =\{v, v'\}$ for some $v\neq v'$, then $Q$ is path and  we let $G^Q$ be the subgraph of $G$ bounded by outer cycle $C^Q:=v(C\setminus\mathring{P})v'+Q$, and $\mathcal{G}^Q$ denote the rainbow $(G^Q, C^Q, Q, L)$.}
\item\emph{If $V(Q)\cap V(C\setminus\mathring{P})=\{v\}$ for some vertex $v$, then we set $C^Q:=Q$ and we set $G^Q:=Q$ if $Q$ is a path and $G^Q:=\textnormal{Int}_G(Q)$ if $Q$ is a cycle.}
\end{enumerate} \end{defn}

To prove the result which makes up the Section \ref{ExtCol2PathAugColSec}, we introduce one more piece of notation. 

\begin{defn}\label{EndNotationColor} \emph{Let $G$ be a graph with list-assignment $L$ and $P\subseteq G$ be a path with endpoints $p, p'$. We define $\textnormal{End}_L(P, G)$ to be the set of $L$-colorings $\phi$ of $\{p, p'\}$ such that $\phi$ is $(P,G)$-sufficient.} \end{defn}

We usually drop the subscript $L$ if it is clear from the context.

\begin{theorem}\label{SumTo4For2PathColorEnds} Let $(G, C, P, L)$ be an end-linked rainbow, where $P:=p_0qp_1$. Then  $\textnormal{End}(P,G)\neq\varnothing$. Furthermore, either $|\textnormal{End}(P, G)|\geq 2$ or there is an even-length path $Q$ with $V(Q)\subseteq V(C)$, where $Q$ has endpoints $p_0, p_1$ and each vertex of $Q$ is adjacent to $q$.   \end{theorem}

\begin{proof} Suppose not, and let $G$ be a vertex-minimal counterexample to the theorem. Let $C, P, L$ be as in the statement of the theorem. If $V(C)=V(P)$, then, by Corollary \ref{CycleLen4CorToThom}, any $L$-coloring of $V(P)$ extends to an $L$-coloring of $G$, and thus $|\textnormal{End}(P,G)|\geq 2$, contradicting our assumption that $G$ is a counterexample. Thus, $|V(C)|>3$. It is also straightforward to check using Corollary \ref{CycleLen4CorToThom}, Theorem \ref{thomassen5ChooseThm}, togther with the minimality of $|V(G)|$, that $G$ is short-inseparable and every chord of $C$ is incident to $q$. As $|\textnormal{End}(P, G)|<2$, it follows Theorem \ref{EitherBWheelOrAtMostOneColThm} that $G$ is a broken wheel with principal path $P$. We may suppose, by removing colors from the lists of some of the vertices if necessary, that each of $p_0, p_1$ has a list of size at most three and each vertex of $C\setminus P$ has a list of size precisely three. If $\textnormal{End}(P,G)\neq\varnothing$, then, since $G$ is a counterexample, we have $|\textnormal{End}(P, G)|=1$, and $G-q$ is a path of odd length, contradicting 2) of Proposition \ref{CorMainEitherBWheelAtM1ColCor}. Thus, $\textnormal{End}(P, G)=\varnothing$. We note now that $|V(G)|>4$, or else, since $|L(p_0)|+|L(p_1)|\geq 4$, we have $\textnormal{End}(P, G)\neq\varnothing$. Now suppose without loss of generality that $|L(p_0)|\geq 2$. Let $p_0uv$ be the length-two subpath of $G-q$ with $p_0$ as an endpoint, and $P':=vqp_1$. Note that $G^{P'}=G\setminus\{p_0, u\}$.

\begin{claim}\label{NoElUses1} No element of $\textnormal{End}(P', G^{P'})$ uses a color of $L(p_0)\cap L(v)$ on $v$. \end{claim}

\begin{claimproof} Suppose not, and let $a\in L(p_0)\cap L(v)$ and $\phi\in\textnormal{End}(P', G^{P'})$ with $\phi(v)=a$. Let $\psi$ be the $L$-coloring of of $\{p_0, p_1\}$ using $\phi(v), \phi(p_1)$ on $p_0, p_1$ respectively.  For any extension of $\psi$ to an $L$-coloring $\psi'$ of $V(P)$, we have $a\in L_{\psi'}(v)$, and there is another color left for $u$, so $\psi$ extends to $L$-color $G$. Thus, $\psi\in\textnormal{End}(P,G)$, which is false. \end{claimproof} 

Now, Claim \ref{NoElUses1}, together with the minimality of $G$, imply that $|L(p_1)|+|L(p_0)\cap L(v)|<4$. It follows that $(4-|L(p_0)|)+|L(p_0)\cap L(v_0)|<4$, so $L(p_0)\not\subseteq L(v)$. As $\textnormal{End}(P,G)=\varnothing$, no color of $L(p_0)$ is $(P, G)$-universal, contradicting \ref{PropCor1} of Proposition \ref{CorMainEitherBWheelAtM1ColCor}.  \end{proof}

\section{Extending Colorings of 3-Paths}\label{LinkColoring3PathExCycleSec}

To prove Theorem \ref{MainHolepunchPaperResulThm}, we need several intermediate results about 3-paths. The results are combined into a single enumerated list, which is Theorem \ref{CombinedT1T4ThreePathFactListThm}. We prove Theorem  \ref{CombinedT1T4ThreePathFactListThm} by a sequence of minimal counterexample arguments. In the proofs of Theorems  \ref{CombinedT1T4ThreePathFactListThm} and \ref{MainHolepunchPaperResulThm}, we frequently produce a smaller counterexample from a minimal counterexample by contracting a path of length two and making use of the following observation. 

\begin{obs}\label{MinCounterReUseObs} Let $G$ be a short-inseparable planar graph with outer cycle $C$ and list-assignment $L$. Let $u\in V(G)$ and $x_0x_1x_2x_3\subseteq C$ be a 3-path, each vertex of which is adjacent to $u$, where $x_1, x_2$ have lists of size at least three. Let $\psi$ be an $L$-coloring of $G\setminus\{x_1, x_2\}$ with $\psi(x_0)\neq\psi(x_3)$. Then either $\psi(x_0)\not\in L(x_2)$ or $\psi$ extends to $L$-color $G$. \end{obs}

To state Theorem \ref{CombinedT1T4ThreePathFactListThm}, we need two more definitions.  

\begin{defn} \emph{Let $G$ be a planar embedding with outer cycle $C$ and $P:=p_0q_0q_1p_1$ be a 3-path in $C$. For each $i\in\{0,1\}$, let $x_i$ be the unique vertex of $V(C\setminus\mathring{P})\cap N(q_i)$ farthest from $p_i$ on the path $C\setminus\mathring{P}$.}
\begin{enumerate}[label=\emph{\arabic*)}]
\itemsep-0.1em
\item\emph{We define a $(P,G)$-\emph{obstruction} to be path $Q^*$ with endpoints $x_0, x_1$, where each entry is a path whose vertices lie in $V(C\setminus\mathring{P})$ either $x_0=x_1=Q^*$ or, if $x_0\neq x_1$, then the cycle $Q^*+x_0q_0q_1x_1$ has length at least five (i.e $|E(Q^*)|>1$) and there is a $w^*\in V(G\setminus C)$ adjacent to all the vertices of this cycle. We say that $Q^*$ of \emph{triangle-type} if $x^0=x^1=x^*=Q^*$ for some $x^*\in V(C\setminus P)$.}
\item\emph{For each terminal edge $e=p_kq_k$ of $P$, we say that $e$ is \emph{even- tilted} (resp. \emph{odd-tilted}) in $(P, G)$ if there is an even (resp. odd)-length path $Q$ whose vertices lie in $V(C\setminus\mathring{P})$, where $Q$ has endpoints $p_k, x_k$ and each vertex of $Q$ is adjacent to $q_k$. Given a $(P,G)$-obstruction $Q^*$, we say $Q^*$ is \emph{fully even} if it has even length and each terminal edge of $P$ is even-tilted in $(P,G)$.}
\end{enumerate}
\emph{For the purpose of our minimal counterexample arguments, the following observation is useful.}
 \end{defn}

\begin{obs}\label{MinCountChordSepCyDObs} Let $G$ be a planar embedding with outer cycle $C$ and $P\subseteq C$ be a 3-path. Then,
\begin{enumerate}[label=\arabic*)]
\itemsep-0.1em
\item for any cycle $D\subseteq G$ with $|E(D)|\leq 4$ and any path $Q^*$ in $G$, $Q^*$ is a $(P,G)$-obstruction if and only if it is a $(P,\textnormal{Ext}(D))$-obstruction. Likewise, for each terminal edge $e$ of $P$, $e$ is even-tilted (resp. odd-tilted) in $(P, G)$ if and only if it is even-tilted (resp. odd tilted) in $(P, \textnormal{Ext}(D))$; AND
\item For any chord $xy$ of $C$ with $x,y\not\in V(\mathring{P})$ and any path $Q^*$ in $G$, the following hold: Letting $G=G^0\cup G^1$ be the natural $xy$-partition of $G$, where $P\subseteq G^1$, $Q^*$ is a $(P,G)$-obstruction if and only if it is a $(P, G^1)$-obstruction. Likewise, for each terminal edge $e$ of $P$, $e$ is even-tilted (resp. odd-tilted) in $(P, G)$ if and only if it is even-tilted (resp. odd tilted) in $(P, G^1)$.
\end{enumerate}
\end{obs}

The last definition we need in order to state Theorem \ref{CombinedT1T4ThreePathFactListThm} is the following: 

\begin{defn}\label{BaseColoringDefn} \emph{Let $\mathcal{G}:=(G, C, P, L)$ be a rainbow, where $P:=p_0q_0q_1p_1$ is a 3-path. A $\mathcal{G}$-base-coloring is an $L$-coloring $\phi$ of $\{p_0, p_1\}$ such that, letting $\mathcal{F}$ be the set of extensions of $\psi$ to $L$-colorings of $V(P)$ which do not extend to $L$-color $G$, we have $|\mathcal{F}|\leq 2$, and in fact, one of the following holds.}
\begin{enumerate}
\itemsep-0.1em
\item [\mylabel{T4PartA}{\textnormal{B1)}}]  $|\mathcal{F}|\leq 1$; OR
\item [\mylabel{T4PartB}{\textnormal{B2)}}]  \emph{There is a triangle-type $(P,G)$-obstruction and a $j\in\{0,1\}$ such that $p_jq_j$ is even-tilted in $\mathcal{G}$, where every element of $\mathcal{F}$ uses the same color on $q_{1-j}$}; OR;
\item [\mylabel{T4PartC}{\textnormal{B3)}}] \emph{There is a triangle-type $(P,G)$-obstruction and $\{\psi'(q_0), \psi'(q_1)\}$ is constant as $\psi'$ runs over $\mathcal{F}$. Furthermore, at least one terminal edge of $P$ is odd-tilted in $\mathcal{G}$.}
\end{enumerate}
\end{defn}

We prove Theorem \ref{CombinedT1T4ThreePathFactListThm}, stated below, over the course of Sections \ref{PrFPart1T1T4}-\ref{CorColSecRes}. 

\begin{theorem}\label{CombinedT1T4ThreePathFactListThm} Let $\mathcal{G}:=(G, C, P, L)$ be a rainbow, where $P:=p_0q_0q_1p_1$. For each $i\in\{0,1\}$, let $x_i$ be the unique vertex of $V(C\setminus\mathring{P})\cap N(q_i)$ which is farthest from $p_i$ on the path $C\setminus\mathring{P}$. Then all of the following hold. 
\begin{enumerate}
\itemsep-0.1em
\item [\mylabel{LabCrownNonEmpt}{\textnormal{T1)}}] If $\mathcal{G}$ is end-linked, then there is a $(P, G)$-sufficient $L$-coloring of $\{p_0, x_0, x_1, p_1\}$.
\item [\mylabel{LabCrownNonEmpt2}{\textnormal{T2)}}] 
\begin{enumerate}[label=\Alph*)]
\itemsep-0.1em
\item\label{T2PartA} If $|L(p_1)|\geq 3$, then either there is a fully even $(P, G)$-obstruction or there are two $(P,G)$-sufficient $L$-colorings of $\{p_0, x_0, x_1, p_1\}$ using different colors on $p_1$; AND
\item\label{T2PartB} If $x_0=p_0$ and $x_1=p_1$, then either there is a $(P,G)$-obstruction or at most one $L$-coloring of $\{p_0, p_1\}$ is not $(P, G)$-sufficient. 
\end{enumerate}
\item [\mylabel{LabCrownNonEmpt3}{\textnormal{T3)}}] Suppose $|L(p_1)|\geq 3$. Then there is a $(P, G)$-sufficient $L$-coloring of $\{p_0, x_1, p_1\}$, and, furthermore, there is either a $(P,G)$-sufficient $L$-coloring of $\{p_0, p_1\}$ or a $(P,G)$-obstruction of even length, where $x_1\in V(C\setminus P)$. 
\item [\mylabel{LabCrownNonEmpt4}{\textnormal{T4)}}] Suppose at least one endpoint of $P$ has a list of size at least three. There is a $\mathcal{G}$-base-coloring of $\{p_0, p_1\}$. 
\end{enumerate}
 \end{theorem}

We prove \ref{LabCrownNonEmpt}-\ref{LabCrownNonEmpt4} with four separate minimal counterexample arguments. To avoid repeatedly re-introducing the same paths in $G$, we define the following, which we retain throughout the proof of Theorem \ref{CombinedT1T4ThreePathFactListThm}. 

\begin{defn} \emph{We let $P_0:=p_0q_0x_0$ and $P_1:=x_1q_1p_1$ (possibly, one or both of $P_0, P_1$ is an edge). Likewise, we let  $M:=x_0q_0q_1x_1$ (possibly $M$ is a triangle) and we let $R_0:=p_0q_0q_1x_1$ and $R_1:=x_0q_0q_1p_1$ (possibly one, but not both, of $R_0, R_1$ is a triangle).}\end{defn} 

\section{The Proof of \ref{LabCrownNonEmpt}}\label{PrFPart1T1T4}

\begin{proof} Suppose toward a contradiction that $\mathcal{G}$ does not satisfy \ref{LabCrownNonEmpt}, where $G$ is vertex-minimal with respect to this property. By removing colors from the lists of some vertices if necessary, we suppose for convenience that each vertex of $C\setminus P$ has a list of size precisely three, and $|L(p_1)|+|L(p_4)|=4$. It is immediate from Corollary \ref{CycleLen4CorToThom}, together with the minimality of $G$, that $G$ is short-inseparable. Likewise, it is immediate from Theorem \ref{thomassen5ChooseThm}, together with the minimality of $G$, that any chord of $C$ has an endpoint in $\mathring{P}$. 

\begin{Claim}\label{NoChorBothEndPClaimFirst3Path} $C$ is induced. \end{Claim}

\begin{claimproof} If $G$ contains either of the edges  $p_0q_1, p_1q_0$, then Theorem \ref{SumTo4For2PathColorEnds} implies there is a $(P, G)$-sufficient $L$-coloring of $\{p_0, x_0, x_1, p_1\}$. Likewise, if $|V(C)|\leq 4$, then Corollary \ref{CycleLen4CorToThom} implies that any $L$-coloring of $\{p_0, x_0, x_1, p_1\}$ is $(P, G)$-sufficient. Thus, $|V(C)|>4$ and every chord of $C$ has precisely one endpoint in $\{q_0, q_1\}$ and the other endpoint in $C\setminus P$. Now suppose $C$ is not induced, and suppose without loss of generality that there is a chord of $C$ incident to $q_0$. Thus, $|V(G^{R_1})|<|V(G)|$ and $x_0\in V(C\setminus P)$. As $|L(x_0)|=3$, it follows from the minimality of $G$ that there is a family of $|L(p_1)|$ different $(R_1, G^{R_1})$-sufficient $L$-colorings of $\{x_0, x_1, p_1\}$, each using a different color on $x_0$. Likewise, it follows from it follows from Theorem \ref{SumTo4For2PathColorEnds} that there is a family of $|L(p_0)|$ different elements of $\textnormal{End}(P_0, G^{P_0})$, each using a different color on $x_0$, so there is a $(P_0, G^{P_0})$-sufficient $L$-coloring $\phi$ of $\{p_0, x_0\}$ and a $(R_1, G^{R_1})$-sufficient $L$-coloring $\psi$ of $\{x_0, x_1, p_1\}$ with $\phi(x_0)=\psi(x_0)$, and the union is a $(P, G)$-sufficient $L$-coloring of $\{p_0, x_0, x_1, p_1\}$, contradicting our assumption on $G$. \end{claimproof}

\begin{Claim}\label{First3PathNeighborShareClMCL} $N(p_0)\cap N(q_1)=\{q_0\}$, and likewise, $N(p_1)\cap N(q_0)=\{q_1\}$. \end{Claim}

\begin{claimproof} Suppose not. As $C$ is induced, we suppose without loss of generality that there is a $w\in V(G\setminus C)$ adjacent to each of $q_0, p_1$. Now, $\{q_0, p_1\}\subseteq N(q_1)\subseteq \{q_0, w, p_1\}$. Let $P_*=p_0q_0wp_1$. Note that $G^{P_*}=G-q_1$. By minimality, there is a $(P_*, G^{P_*})$-sufficient partial $L$-coloring $\psi$ of $V(C^{P_*})\setminus\{q_0, w\}$ with $p_0, p_1\in\textnormal{dom}(\psi)$, where $|L_{\psi}(w)|\geq 3$. Let $a=\psi(p_0)$ and $b=\psi(p_1)$. Since $G$ is a counterexample, there is an $L$-coloring $\phi$ of $V(P)$ using $a,b$ on $p_0, p_1$ respectively, where $\phi$ does not extend to $L$-color $G$. As $C$ is induced, $\phi\cup\psi$ is a proper $L$-coloring of its domain. Note that $|L_{\phi\cup\psi}(w)|\geq 1$, so $\phi\cup\psi$ extends to $L$-color $G$, a contradiction.  \end{claimproof}

Now let $\phi$ be an arbitrary $L$-coloring of $\{p_0, p_1\}$. As above, $\phi$ extends to an $L$-coloring $\psi$ of $V(P)$, where $\psi$ does not extend to $L$-color $G$. Since $C$ is induced, it follows from Lemma \ref{PartialPathColoringExtCL0} that there exists a $w\in V(G\setminus C)$ with at least three neighbors in $C$, contradicting Claim \ref{First3PathNeighborShareClMCL}. This proves \ref{LabCrownNonEmpt} of Theorem \ref{CombinedT1T4ThreePathFactListThm}. \end{proof}

\section{The proof of \ref{LabCrownNonEmpt2}}

\begin{proof} Suppose \ref{LabCrownNonEmpt2} does not hold, and let $(G, C, P, L)$ be a counterexample to \ref{LabCrownNonEmpt2}, where $G$ is vertex-minimal with respect to this property. We may suppose, by removing colors from the lists of some of the vertices if necessary, that each vertex of $C\setminus P$ has a list of size precisely three, although, due to the statement of \ref{T2PartB}, we cannot remove colors from the lists of $V(P)$, even from the endpoints of $P$. However, we may suppose that $|L(p_1)|\geq 3$. To see this, note that, if $|L(p_1)|<3$, then \ref{T2PartB} is violated, and if we add colors to $L(p_1)$, then \ref{T2PartB} is still violated, so we suppose that $|L(p_1)|\geq 3$. Recalling Observation \ref{MinCountChordSepCyDObs}, it follows from the minimality of $|V(G)|$, together with Corollary \ref{CycleLen4CorToThom} and Theorem \ref{thomassen5ChooseThm}, that $G$ is short-inseparable and every chord of $C$ has an endpoint in $\mathring{P}$. 

\begin{claim}\label{ChordsAcrossForImpCrown} $p_0q_1\not\in E(G)$ and $p_1q_0\not\in E(G)$. Furthermore, $x_0=p_0$.  \end{claim}

\begin{claimproof} Suppose the claim does not hold. Thus, \ref{T2PartA} is violated. Suppose first that $E(G)$ contains one of $p_0q_1, p_1q_0$. Now, there is an $i\in\{0,1\}$ with $x_i=x_{1-i}=p_{1-i}$, and, by Theorem \ref{SumTo4For2PathColorEnds} applied to $G-q_{1-i}$, there are two elements of $\textnormal{End}(p_iq_ip_{1-i}, G-q_{1-i})$ differing precisely on $p_1$. Each such element is also a $(P,G)$-sufficient $L$-coloring of $\{p_0, p_1\}$, which is false. Thus,  $p_0q_1, p_1q_0\not\in E(G)$.  Now suppose $x_0\neq p_0$. As shown above, $x_0\in V(C\setminus P)$. 

\vspace*{-8mm}
\begin{addmargin}[2em]{0em} 
\begin{subclaim}\label{H0BWheelEvenLen} $G^{P_0}$ is a broken wheel, and $|V(G^{P_0})|$ is even. \end{subclaim}

\begin{claimproof} Suppose not. By Theorem \ref{SumTo4For2PathColorEnds} that there exist $\phi, \phi'\in\textnormal{End}(P_0, G^{P_0})$ where $\phi(x)\neq\phi'(x)$. Let $b=\phi(x)$ and $b'=\phi'(x)$. Since $b\neq b'$ and $|L(p_1)|\geq 3$, it follows from \ref{LabCrownNonEmpt} that there exist $(R_1, G^{R_1})$-sufficient $L$-colorings $\psi, \psi^*$ of $\{x_0, x_1, p_1\}$, where $\psi(x_1)\neq\psi^*(x_1)$ and $\psi(x), \psi^*(x_0)\in\{b, b'\}$. Possibly $\psi(x_0)=\psi^*(x_0)$, but, in any case, taking appropriate unions, we produce two $(P, G)$-sufficient $L$-colorings of $\{p_0, x_0, x_1, p_1\}$ using different colors on $p_1$, which is false. \end{claimproof} \end{addmargin}

Applying Theorem \ref{SumTo4For2PathColorEnds}, we fix a $\phi\in\textnormal{End}(P_0, G^{P_0})$. Let $b=\phi(x_0)$. It follows from Subclaim \ref{H0BWheelEvenLen} that there is no fully even $(R_1, G_{R_1})$-obstruction, or else there is a fully even $(P, G)$-obstruction. By minimality, there are two distinct $(R_1, G^{R_1})$-sufficient $L$-colorings $\psi, \psi'$ of $\{x_0, x_1, p_1\}$ with $\psi(x_0)=\psi'(x_0)=b$, where $\psi, \psi'$ differ on $p_1$. Each of $\phi\cup\psi$ and $\phi\cup\psi'$ is $(P, G)$-sufficient, which is false. \end{claimproof}

\begin{claim} $C$ is induced. \end{claim}

\begin{claimproof} Suppose not. Each chord of $C$ is incident to $q_1$, so $x_1\in V(C\setminus P)$ and $x_0\neq x_1$. Thus, again, \ref{T2PartA} is violated. 

\vspace*{-8mm}
\begin{addmargin}[2em]{0em} 
\begin{subclaim}\label{H1BWheelEvenLenSubCL2} $G^{P_1}$ is a broken wheel, where $|V(G^{P_1})|$ is even..\end{subclaim}

\begin{claimproof} Suppose not. By \ref{LabCrownNonEmpt}, there is a $(R_0, G^{R_0})$-sufficient $L$-coloring $\phi$ of $\{p_0, x_0\}$. By Theorem \ref{SumTo4For2PathColorEnds}, there are two $(P_1, G^{P_1})$-sufficient $L$-colorings of $\{x_1, p_1\}$, each using $\phi(x_1)$ on $x_1$. Taking appropriate unions, we produce two $(P, G)$-sufficient $L$-colorings of $\{p_0, x_1, p_1\}$, contradicting our assumption on $G$. \end{claimproof}\end{addmargin}

It follows from Subclaim \ref{H1BWheelEvenLenSubCL2} that there is no fully even $(R_0, G^{R_0})$-obstruction. Thus, by minimality, there exist two $(R_0, G^{R_0})$-sufficient $L$-colorings $\phi, \phi'$ of $\{p_0, x_1\}$, where $\phi(x_1)\neq\phi'(x_1)$. Let $b=\phi(x_1)$ and $b'=\phi'(x_1)$. As $|L(p_1)|\geq 3$, it follows from Theorem \ref{SumTo4For2PathColorEnds} that there exist $\psi, \psi^*\in\textnormal{End}(P_1, G^{P_1})$ with $\psi(x_1), \psi^*(x_1)\in\{b, b'\}$ and $\psi(p_1)\neq\psi^*(p_1)$. Possibly $\psi(x_1)=\psi^*(x_1)$, but, in any case, taking appropriate unions, we produce two $(P, G)$-sufficient $L$-colorings of $\{p_0, x_0, x_1, p_1\}$ which differ on $p_1$, which is false, as $G$ is a counterexample.  \end{claimproof}

Since $C$ is induced, $x_0=p_0$ and $x_1=p_1$. As $G$ is a counterexample, there are at least two $L$-colorings of $\{p_0, p_1\}$ which are not $(P,G)$-sufficient, or else both a) and b) are satisfied. By Lemma \ref{PartialPathColoringExtCL0}, there is a $w\in V(G\setminus C)$ with at least three neighbors in $P$, or else any $L$-coloring of $\{p_0, p_1\}$ is $(P, G)$-sufficient, contradicting our assumption.

\begin{claim}\label{WAdjAtMoOneP0P1TwoSideEndChord} $w$ is adjacent to at most one of $p_0, p_1$ \end{claim}

\begin{claimproof} Suppose $w$ is adjacent to both of $p_0, p_1$. Let $G=G'\cup G''$ be the $p_0wp_1$-partition of $G$, where $C\setminus\mathring{P}\subseteq G'$ and $\mathring{P}\subseteq G''$. Note that $V(G'')=V(P)\cup\{w\}$. For any $L$-coloring $\phi$ of $V(P)$, $\phi$ leaves at least one color for $w$, i.e $|L_{\phi}(w)|\geq 1$, so $\textnormal{End}(p_0wp_1, G')\subseteq\textnormal{End}(P, G)$. Since $|\textnormal{End}(P, G)|\leq 1$, it follows from Theorem \ref{EitherBWheelOrAtMostOneColThm} that $G'$ is a broken wheel with principal path $p_0wp_1$, so there is a $(P, G)$-obstruction. Thus, \ref{T2PartA} is violated, the $(P,G)$-obstruction is not fully even, and $|V(G')|$ is odd. As $\textnormal{End}(p_0wp_1, G')\subseteq\textnormal{End}(P, G)$, it follows from Theorem \ref{SumTo4For2PathColorEnds} that there are two elements of $\textnormal{End}(P, G)$ using different colors on $p_1$, contradicting the fact that \ref{T2PartA} is violated.  \end{claimproof}

Now, there exist two distinct $L$-colorings $\pi^0, \pi^1$ of $V(P)$ which do not extend to $L$-color $G$, where $\pi^0$ and $\pi^1$ do not restrict to the same $L$-coloring of $\{p_0, p_1\}$. Possibly \ref{T2PartB}, but not \ref{T2PartA}, is violated, in which case $\pi^0(p_0)\neq\pi^1(p_0)$. In any case, for each $k=0,1$, we let $a^k:=\pi^k(p_i)$ and $b_k:=\pi^k(p_{1-i})$. By Claim \ref{WAdjAtMoOneP0P1TwoSideEndChord}, there is an $i\in\{0,1\}$ such that $N(w)\cap V(P)=\{p_i, q_0, q_1\}$. We now define the following:
\begin{enumerate}[label=\arabic*)]
\itemsep-0.1em
\item Let $v$ be the vertex of $N(w)\cap V(C\setminus\mathring{P})$ which is farthest from $p_i$ on the path $C\setminus\mathring{P}$. By Claim \ref{WAdjAtMoOneP0P1TwoSideEndChord}, $v\neq p_{1-i}$. Let $C\setminus\mathring{P}$ be denoted by $p_iu_1\cdots u_tp_{1-i}$ for some $t\geq 1$. 
\item  Let $P^-:=p_iwv$ and $P^+:=vwq_{1-i}p_{1-i}$. Finally, for each $k=0,1$, let $S^k$ be the set of $s\in L(v)$ such that there is a $(P^+, G^{P_+})$-sufficient $L$-colorings of of $\{v, p_{1-i}\}$ using $s, b^k$ on $v, p_{1-i}$ respectively. 
\end{enumerate}

This is illustrated in Figure \ref{DropHKPartGPDagger}, where $P^+$ is indicated in bold. Note that $G-q_i=G^{P_-}\cup G^{P_+}$. By \ref{LabCrownNonEmpt}, each of $S^0, S^1$ is nonempty. 

\begin{center}\begin{tikzpicture}
\node[shape=circle,draw=black] [label={[xshift=-1.2cm, yshift=-0.7cm]\textcolor{red}{$\{a^0, a^1\}$}}] (p0) at (-2,0) {$p_i$};
\node[shape=circle,draw=white] (p0+) at (-0.5,0) {$\cdots$};
\node[shape=circle,draw=black] (q0) at (-2,2) {$q_i$};
\node[shape=circle,draw=black] (v) at (1,0) {$v$};
\node[shape=circle,draw=white] (p1+) at (3,0) {$\cdots$};
\node[shape=circle,draw=black] (q1) at (5,2) {\small $q_{1-i}$};
\node[shape=circle,draw=black] [label={[xshift=1.2cm, yshift=-0.7cm]\textcolor{red}{$\{b^0, b^1\}$}}]  (p1) at (5,0) {\small $p_{1-i}$};
\node[shape=circle,draw=white] (H) at (-0.6,0.35) {$G^{P_-}$};
\node[shape=circle,draw=black] (w) at (-0.5,1.1) {$w$};
\node[shape=circle,draw=white] (w') at (3,1) {\small $G^{P_+}$};
\node[shape=circle,draw=white] (wL) at (4,2.9) {};

 \draw[-] (p0) to (p0+) to (v) to (p1+) to (p1) to (q1) to (q0) to (p0);
 \draw[-] (w) to (p0);
 \draw[-] (w) to (q0);
 \draw[-] (w) to (q1);
 \draw[-] (w) to (v);
\draw[-, line width=1.8pt] (v) to (w) to (q1) to (p1);
\end{tikzpicture}\captionof{figure}{}\label{DropHKPartGPDagger}\end{center}

Since $w$ only has four neighbors in $V(P)\cup\{w\}$, we immediately have the following for each $k=0,1$:
\begin{equation}\label{EqDagg1M}\tag{$\dagger$}\textnormal{If $v=p_i$ then $a^k\not\in S^k$ and, if $v\neq p_i$, then $\Lambda_{G^{P_-}}(a^k, c, \bullet)\cap S_k=\varnothing$ for each $c\in L_{\pi^k}(w)$} \end{equation}

\begin{claim}\label{HNotEdgeAlmostKUni} If $G^{P_-}$ is not an edge and $a^0\neq a^1$, then either there is a $k\in\{0,1\}$ such that $a^k$ is an almost $(P_-, G^{P_-})$-universal  color of $L(p_i)$ or $G^{P_-}$ is a broken wheel, where $3\leq |V(G^{P_-})|\leq 4$ and $\{a^0, a^1\}\subseteq\bigcap (L(u): u\in V(G^{P_-})\setminus\{w\})$. \end{claim}

\begin{claimproof} Suppose $G^{P_-}$ is not an edge and $a^0\neq a^1$, but the claim does not hold. By Theorem \ref{EitherBWheelOrAtMostOneColThm}, $G^{P_-}$ is a broken wheel with principal path $p_iwv$ and $\{a^0, a^1\}\subseteq L(u_1)$, so $G^{P_-}$ is a not a triangle. By \ref{BWheel3Lb} of Theorem \ref{BWheelMainRevListThm2}, $|V(G^{P_-})|\neq 4$. Thus, $|V(G^{P_-})|>4$ and $p_iu_1u_2u_3\subseteq G^{P_-}\setminus\{w\}$, and, again by \ref{BWheel3Lb} of Theorem \ref{BWheelMainRevListThm2}, $\{a^0, a^1\}\subseteq L(u_1)\cap L(u_2)$. Let $G^*$ be a graph obtained from $G$ by deleting the vertices $u_1, u_2$ and replacing them with the edge $p_iu_3$, so that $G^*$ has outer cycle $p_iu_3(C\setminus\mathring{P})p_{1-i}q_{1-i}q_i$, and the outer cycle of $G^*$ is still induced. Furthermore, there is no $(P, G^*)$-obstruction and each of $\pi^0, \pi^1$ is still a proper $L$-coloring of its domain in $G$. Thus, by minimality, there is a $k\in\{0,1\}$ such that $\pi^k$ extends to an $L$-coloring of $G^*$. As $\{a_0, a_1\}\subseteq L(u_1)\cap L(u_2)$, this contradicts Observation \ref{MinCounterReUseObs}. \end{claimproof}

We show now that $G^{P_+}$ is a wheel, i.e we have the structure in Figure \ref{DropConditionCounterExFig2}.

\begin{claim}\label{KWheelCenVerHww'} $G^{P_+}$ is a wheel with a central vertex adjacent to all the vertices of its outer face, where $|V(G^{P_+})|$ is even. \end{claim}

\begin{claimproof} Suppose not. Thus, there is no fully even $(P^+, G^{P_+})$-obstruction. We first show that $G^{P_-}$ is not an edge. Suppose $G^{P_-}$ is an edge. Thus, $G^{P_+}=G-q_i$. Since $G$ is short-inseparable, $G-q_i$ is not a wheel, so there is no $(P^+, G-q_i)$-obstruction. By the minimality of $|V(G)|$, there is at most one $L$-coloring of $\{p_0, p_1\}$ which is not $(P^+, G-q_i)$-sufficient, so at least one of $\pi^0, \pi^1$ extends to $L$-color $G$, a contradiction. Thus, $G^{P_-}$ is not an edge, so $|L(v)|=3$. There is no fully even $(P^+, G^{P_+})$-obstruction, so, by the minimality of $|V(G)|$, each of $S^0, S^1$ has size at least two and furthermore, if $b^0\neq b^1$, then $S^0\cup S^1=L(v)$. 

\vspace*{-8mm}
\begin{addmargin}[2em]{0em} 
\begin{subclaim} $b^0\neq b^1$. \end{subclaim}

\begin{claimproof} Suppose $b^0=b^1=b$ for some $b$. Thus, $a^0\neq a^1$. Furthermore, $S^0=S^1=S$ for some $S\subseteq L(v)$. Since $|S|\geq 2$, it follows from (\ref{EqDagg1M}) that there is an $r\in L(v)$ such that $\Lambda_{G^{P_-}}(a^k, c, \bullet)=r$ for each $k=0,1$ and $c\in L_{\pi^k}(w)$. By Theorem \ref{EitherBWheelOrAtMostOneColThm}, $G^{P_-}$ is a broken wheel with principal path $p_iwv$. As $a^0\neq a^1$ and each $L_{\pi^k}(w)$ is a set of size at least two which does not contain $a^k$, we contradict \ref{BWheel1Lb} \ref{BWheel1C} of Theorem \ref{BWheelMainRevListThm2}. \end{claimproof}\end{addmargin}

Since $b^0\neq b^1$, we have $S^0\cup S^1=L(v)$, and thus, by (\ref{EqDagg1M}), there exist distinct $r^0, r^1\in L(v)$, where $S^0=L(v)\setminus\{r^0\}$ and $S^1=L(v)\setminus\{r^1\}$, and furthermore, for each $k=0,1$ and $c\in L_{\pi^k}(w)$, we have $\Lambda_{G^{P_-}}(a^k, c, \bullet)=r^k$. In particular, $a^0\neq a^1$ and neither of $a_0, a_1$ is an almost $G^{P_-}$-universal color of $L(p_i)$. By minimality, there is a $(P^+, G^{P_+})$-obstruction, or else, since $a^0\neq a^1$ and $vp_{1-i}\not\in E(G)$, there is a $k\in\{0,1\}$ such that $S^k=L(v)$. Thus, $G^{P_+}$ is a wheel with central vertex $w^+$ adjacent to all the vertices of its outer face, and, by assumption, $|V(G^{P_+})|$ is odd. For each $k=0,1$, $L_{\pi}(a^k)$ is a list of size at least two which does not contain $a^k$, so the above implies that $G^{P_-}$ is not a triangle. By Claim \ref{HNotEdgeAlmostKUni}, $G^{P_-}$ is a broken wheel with principal path $p_iwv$, where $|V(G^{P_-})|=4$ and $\{a^0, a^1\}\subseteq L(v)$. It follows that, for each $k\in\{0,1\}$ and $c\in L$, we have $a^k\in\Lambda_{G^{P_-}}(a^k, c, \bullet)$, so $r^0=a^0$ and $r^1=a^1$. Now, let $H$ be the broken wheel $G^{P_+}\setminus\{w, q_{1-i}\}=G^{vw^+p_{1-i}}$ and, for each $k=0,1$, let $S_*^k:=L(w^+)\setminus\{a^k, b^k, \pi^k(q_{1-i})\}$. 

\vspace*{-8mm}
\begin{addmargin}[2em]{0em} 
\begin{subclaim}\label{Fork01SStark} For each $k=0,1$, $S^*_k\cap\Lambda_{H}(a^k, \bullet, b^k)=\varnothing$. \end{subclaim}

\begin{claimproof} Suppose there is a $k\in\{0,1\}$ for which this does not hold. Since $G^{P_-}$ is not a triange, $\pi^k$ extends to an $L$-coloring $\psi$ of $V(P\cup G^{P_+})$ using $a^k$ on $v$. As $p_i, v$ are using the same color, ther is a color left for $w$, and, in particular, as $|V(G^{P_-})|=4$, $\psi$ extends to $L$-color $G$, which is false. \end{claimproof}\end{addmargin}

Possibly there is a $k\in\{0,1\}$ such that $a^k=b^k$, but, in any case, since $H$ is not a triangle, it now follows from \ref{PropCor3} i) of Proposition \ref{CorMainEitherBWheelAtM1ColCor} that $|S^*_0|=|S^*_1|=2$. In particular, $a^0\neq b^0$ and $a^1\neq b^1$. Since $a^0\neq a^1$, we contradict \ref{PropCor3} ii) of Proposition \ref{CorMainEitherBWheelAtM1ColCor}. This proves Claim \ref{KWheelCenVerHww'}. \end{claimproof}

Let $w^+$ be  the central vertex of $G^{P_+}$ and $H$ be the broken wheel $G^{P_+}\setminus\{w, q_{1-i}\}=G^{vw^+p_{1-i}}$. Note that $|V(H)|$ is even. In particular, $vp_{1-i}\not\in E(G)$. Furthermore, $\textnormal{End}(vw^+p_{1-i}, H)\subseteq\textnormal{End}(P^+, G^{P_+})$. Since $G$ is $K_{2,3}$-free, we have $v\in V(C\setminus P)$, so $p_iwv$ is a 2-path.

\begin{center}\begin{tikzpicture}
\node[shape=circle,draw=black] (p0) at (-2,0) {$p_i$};
\node[shape=circle,draw=white] (p0+) at (-0.5,0) {$\cdots$};
\node[shape=circle,draw=black] (q0) at (-2,2) {$q_i$};
\node[shape=circle,draw=black] (v) at (1,0) {$v$};
\node[shape=circle,draw=white] (p1+) at (3,0) {$\cdots$};
\node[shape=circle,draw=black] (q1) at (5,2) {\small $q_{1-i}$};
\node[shape=circle,draw=black] (p1) at (5,0) {\small $p_{1-i}$};
\node[shape=circle,draw=white] (H) at (-0.6,0.35) {$G^{P_-}$};
\node[shape=circle,draw=white] (H+) at (3,0.35) {$H$};
\node[shape=circle,draw=black] (w) at (-0.5,1.1) {$w$};
\node[shape=circle,draw=black] (w') at (3,1.1) {\small $w^+$};
\node[shape=circle,draw=white] (wL) at (4,2.9) {};

 \draw[-] (p0) to (p0+) to (v) to (p1+) to (p1) to (q1) to (q0) to (p0);
 \draw[-] (w) to (p0);
 \draw[-] (w) to (q0);
 \draw[-] (w) to (q1);
 \draw[-] (q1) to (w') to (w) to (v) to (w') to (p1);
\draw[-, line width=1.8pt] (v) to (w) to (q1) to (p1);
\end{tikzpicture}\captionof{figure}{}\label{DropConditionCounterExFig2}\end{center}

\begin{claim}  $b^0\neq b^1$. \end{claim}

\begin{claimproof} Suppose $b^0=b^1=b$ for some color $b$. Thus, $a^0\neq a^1$.

\vspace*{-8mm}
\begin{addmargin}[2em]{0em} 
\begin{subclaim}\label{NeitherA0A1AlmostUniv} Neither $a^0$ nor $a^1$ is an almost $G^{P_-}$-universal color of $L(p_i)$. Furthermore, $G^{P_-}$ is a  broken wheel, where $|V(G^{P_-})|$ is even. \end{subclaim}

\begin{claimproof} Suppose there is an $a^k\in L(p_i)$ which is almost $G^{P_-}$-universal. By \ref{BWheel2Lb} of Theorem \ref{BWheelMainRevListThm2}, there is a $c\in L_{\pi^k}(w^+)$ with $|\Lambda_{H}(\bullet, c, b)|\geq 2$. Since $|L_{\pi^0}(w)|\geq 2$, there is a $d\in L_{\pi^0}(w)\setminus\{c\}$. By our assumption on $a^k$, we have $\Lambda_{G^{P_-}}(a^k, d, \bullet)|\geq 2$. Since $|L(v)|=3$, it follows that $\pi^k$ extends to $L$-color $G$, which is false. Now suppose either $G^{P_-}$ is not a broken wheel or $|V(G^{P_-})|$ is odd. For each $k=0,1$, let $T^k$ be the set of $c\in L(v)$ such that there is a $(p_iwv, G^{P_-})$ sufficient $L$-colorings of $\{p_i, v\}$ using $a^k$ on $p_i$ and $c$ on $v$. Thus, $T^k\cap S^k=\varnothing$. As $|L(v)|=3$ and $a^0\neq a^1$, it follows from Theorem \ref{SumTo4For2PathColorEnds} that $|T_0|\geq 2$ and $|T_1|\geq 2$, and $T_0\cup T_1=L(v)$. Thus, $S^0$ and $S^1$ are disjoint singletons, which is false, as $b^0=b^1$. \end{claimproof}\end{addmargin}

By Subclaim \ref{NeitherA0A1AlmostUniv}, together with Claim \ref{HNotEdgeAlmostKUni}, we get that $|V(G^{P_-})|=4$, so $v=u_2$, and have $\{a^0, a^1\}\subseteq L(u_1)\cap L(v)$. Now choose an arbitrary $k\in\{0,1\}$. As $H-w^+$ has even length, it follows from \ref{PropCor3} i) of Proposition \ref{CorMainEitherBWheelAtM1ColCor} that $\Lambda_{H}(a^k, \bullet, a)\cap L_{\pi^k}(w^+)\neq\varnothing$, so $\pi^k$ extends to an $L$-coloring $\psi$ of $V(P\cup H)$ with $\psi(v)=a^k$. In particular, since the same color is used on $p_i, v$, it follows that $\psi$ extends to $L$-color $G$, contradicting our assumption on $\pi^k$. \end{claimproof}

\begin{claim}\label{K'WheelOn5Nbg} Neither $b^0$ nor $b^1$ is a $(vw^+p_{1-i},H)$-universal color of $L(p_{1-i})$, and $\{b^0, b^1\}\subseteq L(u_t)$. Furthermore, $|V(G^{P_+})|=6$, and $v=u_{t-1}$

\end{claim}

\begin{claimproof} Let $k\in\{0,1\}$ and suppose $b^k$ is  $(vw'p_{1-i}, H)$-universal color of $L(p_{1-i})$. As $vp_{1-i}\not\in E(G)$, we can extend $\pi^k$  to an $L$-coloring $\psi$ of $V(P\cup G^{P_-})$ using Theorem \ref{thomassen5ChooseThm} applied to $G^{P_-}$. As $|N(w^+)\cap \textnormal{dom}(\psi)|=4$, it follows from our assumption on $b^k$ that $\pi^k$ extends to $L$-color $G$, which is false, so there is no such $b^k$, and $\{b^0, b^1\}\subseteq L(u_t)$. Now suppose $|V(G^{P_+})|\neq 6$. As $|V(H)|$ is even, we have $|V(G^{P_+})|\geq 8$. Since $b^0\neq b^1$, it follows from \ref{PropCor1} of Proposition \ref{CorMainEitherBWheelAtM1ColCor} that $\{b^0, b^1\}\subseteq L(u_t)\cap L(u_{t-1})$. Let $G'$ be a graph obtained from $G$ by deleting the vertices $u_{t-1}, u_t$ and replacing them with the edge $u_{t-2}p_1$, so that $G'$ has outer cycle $C':=(C\setminus\{x,y\})+u_{t-2}p_1$. Each of $\pi^0, \pi^1$ is still a proper $L$-coloring of its domain in $G'$. The outer cycle of $G'$ is still induced, and there are no $(P, G')$-obstructions, so, by minimality, there is a $k\in\{0,1\}$ such that $\pi^k$ extends to $L$-coloring $\psi$ of $G'$, contradicting Observation \ref{MinCounterReUseObs}. \end{claimproof}

\begin{claim}\label{SubCLUnionComi=0Case} $G^{P_-}$ is a broken wheel, where $|V(G^{P_-})|$ is even, and, for each $k\in\{0,1\}$, we have $b^k\not\in L_{\pi^k}(w)\cup L(v)$. \end{claim}

\begin{claimproof} Supppose first that $G^{P_-}$ is either not a broken wheel, or that $|V(G^{P_-})|$ is odd. Choose an arbitrary $k\in\{0,1\}$. Let $T\subseteq L(v)$ be the set of $c\in L(v)$ such that there is a $(p_iwv, G^{P_-})$-sufficient $L$-coloring of $\{p_i, v\}$ using $a^k, c$ on $p_i, v$ respectively. By Theorem \ref{SumTo4For2PathColorEnds}, $|T|\geq 2$. Thus, $\pi^k$ does not extend to an $L$-coloring of $V(P\cup H)$ which uses a color of $T$ on $v$. Since $|V(G^{P_+})|=6$, the diagram in Figure \ref{For6SplK} shows the lower bounds on the size of $L_{\pi^k}$-lists of the vertices of $G^{P_+}\setminus\{q_{1-i}, p_{1-i}\}$. As $|T|\geq 2$, the graph in Figure \ref{For6SplK} is indeed $L_{\pi^k}$-colorable using a color of $T$ on $v$, a contradiction. Thus, $G^{P_-}$ is indeed a broken wheel, where $|V(G^{P_-})|$ is even.

\begin{center}\begin{tikzpicture}

\node[shape=circle,draw=black] [label={[xshift=-0.5cm, yshift=-0.6cm]\textcolor{red}{$T$}}]  (v) at (1,0) {$v$};

\node[shape=circle,draw=black] [label={[xshift=0.0cm, yshift=0.0cm]\textcolor{red}{$\geq 2$}}] (w) at (0,1.1) {$w$};

\node[shape=circle,draw=black] [label={[xshift=0.7cm, yshift=-0.6cm]\textcolor{red}{$\geq 2$}}] (ut) at (3, 0) {\small $u_t$};

\node[shape=circle,draw=black] [label={[xshift=0.0cm, yshift=0.0cm]\textcolor{red}{$\geq 3$}}] (w') at (3,1.4) {\small $w^+$};

 \draw[-] (w) to (v);
 \draw[-] (w) to (w') to (v) to (ut);
 \draw[-] (w') to (ut);
\end{tikzpicture}\captionof{figure}{}\label{For6SplK}\end{center}

Now suppose there is a $k\in\{0,1\}$ with $b^k\in L_{\pi^k}(w)\cup L(v)$, say $k=0$. Suppose first that $b^0\in L_{\pi^0}(w)$. As $\Lambda_{G^{P_-}}(a^0, b^0, \bullet)\neq\varnothing$, we can extend $\pi^0$ to an $L$-coloring $\psi$ of $V(P\cup G^{P_-})$ using $b^0$ on $w$. Since two neighbors of $w^+$ are using the same color, $\psi$ extends to $L$-color $G$, contradicting our assumption on $\pi^0$. Thus, $b^0\not\in L_{\pi^0}(w)$, so $b^0\in L(v)$. It follows that $\Lambda_{G^{P_-}}(a^0, \bullet, b^0)\cap L_{\pi^0}(w)=\varnothing$, or else, since $w^+$ has only five neighbors, $\pi^0$ extends to $L$-color $G$. As $|V(G^{P_-})|$ is even, we contradict \ref{PropCor3} i) of Proposition \ref{CorMainEitherBWheelAtM1ColCor}.  This proves Claim \ref{SubCLUnionComi=0Case}. \end{claimproof}

As $\{b^0, b^1\}\subseteq L(u_t)$, it follows from Claim \ref{SubCLUnionComi=0Case} that $|L(v)\cap L(u_t)|\leq 1$. Now choose an arbitrary $k\in\{0,1\}$. For each $c\in L_{\pi^k}(w)$, we have $\Lambda_{G^{P_-}}(a^k, c, \bullet)\subseteq L(v)\cap L(u_t)$, or else $\pi^k$ extends to $L$-color $G$. Thus, there is an $r\in L(v)$ such that $\Lambda_{G^{P_-}}(a^k, c, \bullet)=\{r\}$ for each $c\in L_{\pi^k}(w)$. In particular, $r\not\in L_{\pi^k}(w)$. As $|L_{\pi^k}(w)|\geq 2$, we extend $\pi^k$ to an $L$-coloring of $G^{P_+}$ using $r$ on $v$, which is possible as $|V(G^{P_+})|=6$. There is at least one color of $L_{\pi^k}(w)$ not used by $w^+$, so $\pi^k$ extends to $L$-color $G$, which is false. This proves \ref{LabCrownNonEmpt2} of Theorem \ref{CombinedT1T4ThreePathFactListThm}. \end{proof}

\section{The Proof of \ref{LabCrownNonEmpt3}}

\begin{proof}  Suppose \ref{LabCrownNonEmpt3} of Theorem \ref{CombinedT1T4ThreePathFactListThm} does not hold not and let $G$ be a vertex-minimal counterexample to the theorem. By removing colors from the lists of some vertices if necessary, we also suppose for convenience that each vertex of $G\setminus C$ has a list of size precisely five, each vertex of $C\setminus P$ has a list of size precisely three, and that $|L(p_0)|=1$ and $|L(p_1)|=3$. Recalling Observation \ref{MinCountChordSepCyDObs}, it follows from the minimality of $|V(G)|$, together with Corollary \ref{CycleLen4CorToThom} and Theorem \ref{thomassen5ChooseThm}, that $G$ is short-inseparable and every chord of $C$ has an endpoint in $\mathring{P}$. Since $G$ is a counterexample, Corollary \ref{CycleLen4CorToThom} also implies that $|V(C)|>4$.  Let $a$ be the lone color of $L(p_0)$ and let $L(p_1)=\{b^0, b^1, b^2\}$. As $G$ is a counterexample and $p_0p_1\not\in E(G)$, it follows that, for each $k=0,1,2$ (where possibly $b^k=a$), there is an $L$-coloring $\sigma^k$ of $V(P)$, where $\sigma^k$ uses $a,b^k$ on $p_0, p_1$ respectively and does not extend to $L$-color $G$. 

\begin{claim}\label{EvChorCQ0EndPx1P1} Every chord of $C$ has $q_0$ as an endpoint and the other endpoint in $C\setminus P$. In particular, $x_1=p_1$. \end{claim}

\begin{claimproof} If $E(G)$ contains either $p_0q_1$ or $p_1q_0$, then Theorem \ref{SumTo4For2PathColorEnds} implies that $\textnormal{End}(P, G)\neq\varnothing$, so both parts of \ref{LabCrownNonEmpt3} are satisfied, contradicting our assumption on $G$. Thus, $p_0q_1, p_1q_0\not\in E(G)$. Now suppose Claim \ref{EvChorCQ0EndPx1P1} does not hold. Thus, $x_1\in V(C\setminus P)$. As no chord of the outer cycle of $G^{R_0}$ is incident to $q_1$, it follows from the minimality of $G$ that there is a $(R_0, G^{R_0})$ sufficient $L$-coloring $\tau$ of $\{p_0, x_1\}$.  As $|L(p_1)|=3$, it follows from Theorem \ref{SumTo4For2PathColorEnds} applied to $G^{P_1}$ that $\tau$ extends to a $(P,G)$-sufficient $L$-coloring of $\{p_0, x_1, p_1\}$, so, since $G$ is a counterexample, $\textnormal{End}(P,G)=\varnothing$ and, since $x_1\in V(C\setminus P)$, there is no $(P,G)$-obstruction of even length. In particular, $x_0\neq x_1$. 

\vspace*{-8mm}
\begin{addmargin}[2em]{0em}
\begin{subclaim}\label{NoColSuFo411} No color of $L(p_1)$ is almost-$(P_1, G^{P_1})$-universal. \end{subclaim}

\begin{claimproof} Let $k\in\{0,1,2\}$ and suppose $b^k$ is $(P_1, G^{P_1})$-universal. Let $L'$ be a list-assignment for $G^{P_0}\cup H$ where $L'(x_1)=\{\sigma^k(q_1)\}\cup\Lambda_{G^{P_1}}(\bullet, \sigma^k, b^k)$ and otherwise $L'=L$. Note that $|L'(x_1)|=3$ and the $L'$-coloring $\pi:=(a, \sigma^k(q_0), \sigma^k(q_1))$ of $p_0q_0q_1$ does not extend to $L$-color $G^{R_0}$. On the other hand, $\pi$ does extend to $L'$-coloring $V(G^{P_0})\cup\{q_1\}$, so there is an $L'$-coloring of $x_0q_0q_1$ which does not extend to $L$-color $G^M$. By \ref{PropCor4} of Proposition \ref{CorMainEitherBWheelAtM1ColCor} applied to $G^M$, there is a $(P,G)$-obstruction of even length, which is false. \end{claimproof}\end{addmargin}

It follows from Subclaim \ref{NoColSuFo411} and Theorem \ref{EitherBWheelOrAtMostOneColThm} that $G^{P_1}$ is a broken wheel.

\vspace*{-8mm}
\begin{addmargin}[2em]{0em}
\begin{subclaim} $G^{P_1}$ is a triangle. \end{subclaim}

\begin{claimproof} Suppose not. Let $yzp_1$ be the unique 2-path of $C\setminus\mathring{P}$ with $p_1$ as an endpoint. Possibly $y=x_1$, but, in any case, $x_0\neq y$ and $q_0y\not\in E(G)$. Let $G^{\dagger}:=G\setminus\{z, p_1\}$ and $P^{\dagger}:=p_0q_0q_1y$. As there is no $(P,G)$-obstruction of even length, there is no $(P^{\dagger}, G^{\dagger})$-obstruction of even length. By minimality, there is a $(P^{\dagger}, G^{\dagger})$-sufficient $L$-coloring $\pi$ of $\{p_0, y\}$. As no color of $L(p_1)$ is almost-$(P_1, G^{P_1})$-universal, we have $L(p_1)=L(z)$, and, by \ref{BWheel3Lb} of Theorem \ref{BWheelMainRevListThm2}, $\pi(y)\in L(p_1)$, so $\pi$ extends to an $L$-coloring $\pi'$ of $\{p_0, y, p_1\}$ using the same color on $y, p_1$. But then, the restriction of $\pi'$ to $\{p_0, p_1\}$ lies in $\textnormal{End}(P,G)$, which is false, as $\textnormal{End}(P,G)=\varnothing$. \end{claimproof}\end{addmargin}

Since $G^{P_1}$ is a triangle and no color of $L(p_1)$ is almost-$(P_1, G^{P_1})$-universal, we have $L(p_1)=L(x_1)$. Recall that $\tau$ is an $(R_0, G^{R_0})$-sufficient $L$-coloring of $\{p_0, x_1\}$. Let $c:=\tau(x_1)$. As $|L(p_1)|=3$, we suppose for the sake of definiteness that $b^0, b^1\neq c$. As $q_0x_1\not\in E(G)$, we have $\sigma^0(q_1)=\sigma^1(q_1)=c$, or else one of $\sigma^0, \sigma^1$ extends to $L$-color $G$. For each $k=0,1,2$, let $T^k:=\Lambda_{G^{P_0}}(a, \sigma^k(q_0), \bullet)$. Then, for each $k=0,1$ we immediately have the following.
\vspace*{-2mm}
\begin{equation}\label{sigmakQoCb1-kObs}\tag{$\ast$} \textnormal{The $L$-coloring $(\sigma^k(q_0), c, b_{1-k})$ of $q_0q_1x_1$ does not extend to $L$-color $G^M$ using a color of $T^k$ on $x_1$} \end{equation}

\vspace*{-12mm}
\begin{addmargin}[2em]{0em}
\begin{subclaim} $d(x_0, x_1)>1$. \end{subclaim}

\begin{claimproof} Suppose not. Since $x_0\neq x_1$, the outer cycle of $G^M$ is an induced 4-cyce and $V(G^M)=\{x_0, q_0, q_1, x_1\}$. It follows from (\ref{sigmakQoCb1-kObs}) that, for each $k=0,1$, we have $T^k=\{b_{1-k}\}$. Thus, $G^{P_0}$ is a broken wheel and, by \ref{BWheel1Lb} \ref{BWheel1B} of Theorem \ref{BWheelMainRevListThm2}, we have $\sigma^k(q_0)=b_k$ for each $k=0,1$. Now consider $\sigma^2$. As $\sigma^2$ also does not extend to $L$-color $G$, it follows that $T^2$ and $L(x_1)\setminus\{b^2, \sigma^2(q_1)\}$ are the same singleton. Since $|T^2|=1$, it follows from \ref{BWheel2Lb} of Theorem \ref{BWheelMainRevListThm2} that there is a $k\in\{0,1\}$ with $\sigma^2(q_0)=\sigma^k(q_0)=b^k$, so $T^2=\{b_{1-k}\}$ and $L(x_1)\setminus\{b^2, \sigma^2(q_1)\}=\{b_{1-k}\}$. But then, $\sigma^2(q_0)=\sigma^2(q_1)=b^k$, which is false. \end{claimproof}\end{addmargin}

Possibly there is a $k\in\{0,1\}$ with $T^k=\{b^{1-k}\}$, but, in any case, as $x_0x_1\not\in E(G)$, there are two $L$-colorings of $\{x_0, q_0, q_1, x_1\}$ which use different colors on $x_1$ and do not extend to $L$-color $G^M$. By \ref{LabCrownNonEmpt2} \ref{T2PartB}, $G^M$ is a wheel. Let $w$ be the central vertex of $G^M$. As there is no $(P,G)$-obstruction of even length, $G^M\setminus\{q_0, q_1, w\}$ is an odd-length path.

\vspace*{-8mm}
\begin{addmargin}[2em]{0em}
\begin{subclaim} $|T^0|=|T^1|=1$ and $T^0\neq T^1$. \end{subclaim}

\begin{claimproof} Suppose there is a $k\in\{0,1\}$ with $|T_k|\geq 2$. Let $L'$ be a list-assignment for $V(G^M)$ with $L'(x_0)=T^k\cup\{\sigma^k(q_0)\}$, where otherwise $L'=L$. Now, $|L'(x_0)|\geq 3$ and, by (\ref{sigmakQoCb1-kObs}), the $L'$-coloring $(\sigma^k(q_0), c, b_{1-k})$ of $q_0q_1x_1$ does not extend to an $L'$-coloring of $G^M$. Since $|V(G^{x_0wx_1})|$ is odd, the outer cycle of $G^M$ has even length and we contradict \ref{PropCor4} of Proposition \ref{CorMainEitherBWheelAtM1ColCor}. Thus, $|T^0|=|T^1|=1$. Now suppose $T^0=T^1=\{d\}$ for some $d\in L(x_0)$. As $w$ only has four neighbors in $x_0q_0q_1x_1$, it follows that, for each $k\in\{0,1\}$, there is an $L$-coloring of $x_0wx_1$ which does not extend to $L$-color $G^{x_0wx_1}$, where this $L$-coloring uses $d, b_{1-k}$ on $x_0, x_1$ respectively. But then, by \ref{PropCor2} of \ref{CorMainEitherBWheelAtM1ColCor}, $|V(G^{x_0wx_1})|$ is even, which is false. \end{claimproof}\end{addmargin}

Since $T^0\neq T^1$ and $|T^0|=|T^1|=1$, $G^{P_0}$ is a broken wheel, and it follows from \ref{BWheel1Lb} of Theorem \ref{BWheelMainRevListThm2} that $T^0\cup T^1=\{\sigma^0(q_0), \sigma^1(q_1)\}=S$ for some set $S$ of size two. Let $x^*$ be the unique neighbor of $x_0$ on the path $G^{x_0wx_1}-w$. We jave $S\subseteq L(x^*)$ by (\ref{sigmakQoCb1-kObs}). As $|L(x^*)|=3$ and $|L(w)\setminus (S\cup\{c\})|\geq 2$, there is a $d\in L(w)\setminus (S\cup\{c\})$ with $d\not\in L(x^*)$. At least one of $b^0, b^1$ is distinct from $d$, so, by coloring $w$ with $d$, we contradict (\ref{sigmakQoCb1-kObs}).  This proves Claim \ref{EvChorCQ0EndPx1P1}. \end{claimproof}

Since $x_1=p_1$, we have $\textnormal{End}(P, G)=\varnothing$, and $G^{P_1}$ is an edge. If $x_0=p_0$, the, by \ref{LabCrownNonEmpt}, $\textnormal{End}(P, G)\neq\varnothing$, so $C$ is not induced. By Claim \ref{EvChorCQ0EndPx1P1}, $x_0\in V(C\setminus P)$. 

\begin{claim}\label{G''w''WheelCycle} $G^M$ is a wheel with a central vertex adjacent to all the vertices of its outer cycle, which has odd length. \end{claim}

\begin{claimproof} Suppose not. For each $k\in\{0,1,2\}$, let $S^k$ be the set of $s\in L(x_0)$ such that there is an $(R_1, G^M)$-sufficient $L$-coloring of $\{x_0, p_1\}$ using $s, b^k$ on $x_0, p_1$ respectively. As the outer cycle of $G^M$ is induced, it follows from \ref{LabCrownNonEmpt2} that, for each $k\in\{0,1,2\}$, we have $|S^k|\geq 2$. Since $\Lambda_{G^{P_0}}(a, \sigma^k(q_0), \bullet)\neq\varnothing$, it follows that $|S^k|=2$ and $\Lambda_{G^{P_0}}(a, \sigma^k(q_0), \bullet)$ is the lone color of $L(x_0)\setminus S^k$. Thus, there exist distinct $k, k'\in\{0,1,2\}$ such that $\sigma^k(x_0)=\sigma^{k'}(q_0)$ and $S^k=S^{k'}$, or else we contradict \ref{BWheel2Lb} of Theorem \ref{BWheelMainRevListThm2} applied to $G^{P_0}$. In particular, there is an $r\in L(x_0)\setminus (S^k\cup S^{k'})$. Yet \ref{LabCrownNonEmpt2}, togther with our assumption on $G^M$, also implies that that there is a $(R_1, G^M)$-sufficient $L$-coloring of $\{x_0, p_1\}$ using $r$ on $x_0$ and a color of $\{b^k, b^{k'}\}$ on $p_1$, contradicting our choice of $r$. \end{claimproof}

Let  $w^*$ be the central vertex of $G^M$. Note that $G^M\setminus\{q_0, q_1\}=G^{x_0w^*p_1}$, which is a broken wheel an even number of vertices. Let $zyp_1$ be the unique subpath of $C\setminus\mathring{P}$ of length two with endpoint $p_1$. Since $\textnormal{End}(P, G)=\varnothing$, it is straightforward to check that no color of $L(p_1)$ is $(x_0w^*p_1, G^{x_0w^*p_1})$-universal. 

\begin{claim} $|V(G^M)|=6$. In particular,  $x_0=z$. \end{claim}

\begin{claimproof} Suppose not. Thus, $|V(G^M)|\geq 8$. Let $z'$ be the unique neighbor of $z$ on the path $C-y$. By \ref{PropCor1} of Proposition \ref{CorMainEitherBWheelAtM1ColCor}, $L(p_1)=L(y)=L(z)$. Let $G^*$ be a graph obtained from $G$ by deleting the vertices $z,y$ and replacing them with the edge $z'p_1$, so that $G^*$ has outer cycle $(C\setminus\{z,y\})+z'p_1$. Every chord of the outer face of $G^*$ is incident to $q_0$ and each $\sigma^k$ is still a proper $L$-coloring of its domain in $G^*$, so, by the minimality of $|V(G)|$, there is a $k\in\{0,1,2\}$ such that $\sigma^k$ extends to an $L$-coloring $\psi$ of $G^*$. Since $b\in L(z)$, this contradicts Observation \ref{MinCounterReUseObs}. \end{claimproof}

\begin{claim}\label{ForEachdLamG'} For each $k=0,1,2$, we have $\Lambda_{G^{P_0}}(a, \sigma^k(q_0), \bullet)\subseteq L(y)\setminus\{b^k\}$, and furthermore, $\sigma^k(q_0)\neq b^k$. \end{claim}

\begin{claimproof}  If $\Lambda_{G^{P_0}}(a, \sigma^k(q_0), \bullet)\not\subseteq L(y)\setminus\{b^k\}$, then $\sigma^k$ extends to $L$-color $G$, which is false. Now suppose $\sigma^k(q_0)=b^k$. As two neighbors of $w^*$ are using the same color, we can again extend $\sigma^k$ to an $L$-coloring of $G$. \end{claimproof}

As $L(y)=L(p_1)$, Claim \ref{ForEachdLamG'} implies that $|L(x_0)\cap L(p_1)|\geq 2$.  By Theorem \ref{SumTo4For2PathColorEnds}, there is a $(P_0, G^{P_0})$-sufficient $L$-coloring $\phi$ of $\{p_0, x_0\}$, where $\phi(p_0)=a$. Let $c:=\phi(x_1)$. If $c\in L(p_1)$, then $c=b^k$ for some $k\in\{0,1,2\}$, and then it follows from our choice of $\phi$ that $\sigma^k\cup\phi$ is not a proper $L$-coloring of its domain, so $\sigma^k(q_0)=b^k$, contradicting Claim \ref{ForEachdLamG'}. Thus, $c\not\in L(p_1)$, so $c\not\in L(y)$. By Claim \ref{ForEachdLamG'}, we have $b^k\not\in\Lambda_{G^{P_0}}(a, \sigma^k(q_0), \bullet)$ for each $b^k$. By our choice of $\phi$, it follows that each $\sigma^k$ uses $c$ on $q_0$. Let $S:=L(u)\cap L(p_1)$. As $|S|\geq 2$ and $c\not\in L(p_1)$, it follows from Theorem \ref{thomassen5ChooseThm} that $\Lambda_{G^{P_0}}(a, c, \bullet)\cap S\neq\varnothing$. Letting $k\in\{0,1,2\}$ with $b^k\in\Lambda_{G^{P_0}}(a, c, \bullet)\cap S$, we extend $\sigma^k$ to an $L$-coloring of $G$ by using the same color on $x_0, p_1$, contradicting our assumption on $\sigma^k$. This proves \ref{LabCrownNonEmpt3} of Theorem \ref{CombinedT1T4ThreePathFactListThm}. \end{proof}

\section{The Proof of \ref{LabCrownNonEmpt4}}\label{CorColSecRes}

We first note that \ref{LabCrownNonEmpt4} is false if we drop the requirement that $p_1$ has a list of size at least three, and only require that $(G, C, P, L)$ is end-linked, as the following counterexample illustrates. In Figure \ref{DropConditionCounterExFig3},  for any $L$-coloring $\psi$ of $\{p_0, p_1\}$, there are two extensions of $\psi$ to $L$-colorings of $V(P)$ which do not extend to $L$-color $G$.

\begin{center}
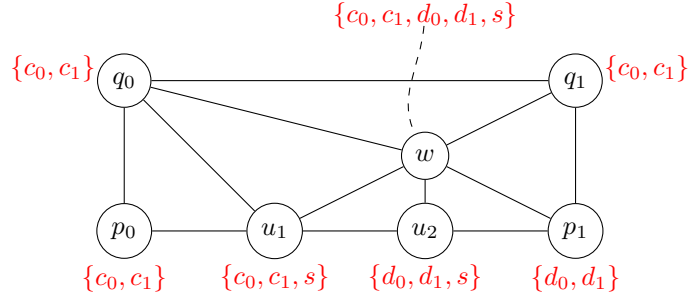
\begin{tikzpicture}
\node[shape=circle,draw=black] [label={[xshift=-0.0cm, yshift=-1.3cm]\textcolor{red}{$\{c_0, c_1\}$}}] (p1) at (0,0) {$p_0$};
\node[shape=circle,draw=black] [label={[xshift=-0.95cm, yshift=-0.5cm]\textcolor{red}{$\{c_0, c_1\}$}}] (p2) at (0,2) {$q_0$};
\node[shape=circle,draw=black] [label={[xshift=-0.0cm, yshift=-1.3cm]\textcolor{red}{$\{c_0, c_1, s\}$}}] (u1) at (2,0) {$u_1$};
\node[shape=circle,draw=black] [label={[xshift=0.0cm, yshift=-1.3cm]\textcolor{red}{$\{d_0, d_1\}$}}] (p4) at (6,0) {$p_1$};
\node[shape=circle,draw=black] [label={[xshift=0.95cm, yshift=-0.5cm]\textcolor{red}{$\{c_0, c_1\}$}}] (p3) at (6,2) {$q_1$};
\node[shape=circle,draw=black] [label={[xshift=0.00cm, yshift=-1.3cm]\textcolor{red}{$\{d_0, d_1, s\}$}}] (u2) at (4,0) {$u_2$};
\node[shape=circle,draw=black] (w) at (4,1) {$w$};
\node[shape=circle,draw=white] [label={[xshift=0.00cm, yshift=-0.50cm]\textcolor{red}{$\{c_0, c_1, d_0, d_1, s\}$}}] (wL) at (4,2.9) {};

 \draw[-] (p1) to (u1) to (u2) to (p4) to (p3) to (p2) to (p1);
 \draw[-] (p2) to (u1);
 \draw[-] (u1) to (w);
 \draw[-] (u2) to (w);
 \draw[-] (p4) to (w);
 \draw[-] (p3) to (w);
 \draw[-] (p2) to (w);
\draw[dashed, -] (wL) to [out=-95, in=115] (w);
\end{tikzpicture}\captionof{figure}{\ref{LabCrownNonEmpt4} of Theorem \ref{CombinedT1T4ThreePathFactListThm} is false if we let $|L(p_0)|=|L(p_1)|=2$}\label{DropConditionCounterExFig3}\end{center}

We now prove \ref{LabCrownNonEmpt4}.

\begin{proof} Suppose \ref{LabCrownNonEmpt4} does not hold and let $G$ be a vertex-minimal counterexample to \ref{LabCrownNonEmpt4}. We suppose without loss of generality that $|L(p_1)|\geq 3$, as Definition \ref{BaseColoringDefn} is symmetric in the endpoints of $P$. By removing colors from some lists if necessary, we suppose further that each vertex of $C\setminus P$ has a list of size precisely three and each vertex of $G\setminus C$ has a list of size precisely five, and furthermore, $|L(p_0)|=1$ and $|L(p_1)|=3$. Applying Theorem \ref{thomassen5ChooseThm} and Corollary \ref{CycleLen4CorToThom}, the minimality of $|V(G)|$ immediately implies that $G$ is short-inseparable and every chord of $C$ is incident to one of $\{q_0, q_1\}$. Since $G$ is a counterexample, we have $\textnormal{End}(P,G)=\varnothing$, or else there is an $L$-coloring of $\{p_0, p_1\}$ satisfying \ref{T4PartA} of Definition \ref{BaseColoringDefn}. In particular, by Corollary \ref{CycleLen4CorToThom}, we have $V(C)\neq V(P)$, so $q_0q_1\not\in E(G)$. It follows now that every chord of $C$ has one endpoint in $\mathring{P}$ and the other endpoint in $C\setminus P$, or else there is a $q\in V(\mathring{P})$ adjacent to each of $p_0, p_1$, and if that holds, then it follows from Theorem \ref{SumTo4For2PathColorEnds} that $\textnormal{End}(P,G)\neq\varnothing$, which is false. We now let $a$ be the lone color of $L(p_0)$. For each $b\in L(p_1)$ we let $\mathcal{F}_b$ denote the set of $L$-colorings of $V(P)$ which use $a,b$ on $p_0, p_1$ respectively and do not extend to $L$-coloring of $G$. Possibly $a\in L(p_1)$ but $p_0p_1\not\in E(G)$ in any case. For each $\pi\in\mathcal{F}_b$ we let $S_{\pi}:=\{\pi(q_0), \pi(q_1)\}$.

\begin{claim}\label{AtMostTwoForEachBLp1} $x_1\in V(C\setminus P)$ and there is a $(P,G)$-obstruction of even length. Furthermore, for each $j\in\{0,1\}$ if $G$ has a triangle-type $(P,G)$-obstruction and $G^{P_j}$ is a broken wheel with an even number of vertices, then, for each $b\in L(p_1)$, there are at most two elements of $\mathcal{F}_b$ using the same color on $q_{1-j}$.
\end{claim}

\begin{claimproof} As $\textnormal{End}(P,G)=\varnothing$, it follows from \ref{LabCrownNonEmpt3} that $x_1\in V(C\setminus P)$ and there is a $(P,G)$-obstruction of even length. Now suppose $x_0=x_1$ and let $j\in\{0,1\}$, where $G^{P_j}$ is a broken wheel, where $G^{P_j}-q_j$ has even length. Let $b\in L(p_1)$ and suppose there are three elements of $\mathcal{F}_b$ using the same color $r$ on $q_{1-j}$. Thus, there is a $T\subseteq L(q_{1-j})$ with $|T|\geq 3$, such that, for each $s\in T$, there is an element of $\mathcal{F}_b$ using $s, r$ on the respective vertices $q_j, q_{1-j}$. Now, there is an $L$-coloring $\psi$ of $G^{P_{1-j}}$ using $r$ on $q_{1-j}$, where $\psi(p_{1-j})=a$ if $j=1$ and $\psi(p_{1-j})=b$ if $j=0$. Possibly $r\in\{a,b\}$, but, in any case, since  $|V(G^{P_j})|$ is even, we have $p_jx_j\not\in E(G)$, so $\psi$ extends to an $L$-coloring $\psi^*$ of $V(P\cup G^{P_{1-j}})$ using $a,b$ on the respective vertices $p_0, p_1$. Since $|T\setminus\{\psi^*(x_{1-j})\}|\geq 2$ and $G^{P_j}-q_j$ is a path of even length, it follows from \ref{PropCor3} of Proposition \ref{CorMainEitherBWheelAtM1ColCor} that at least one element of $\mathcal{F}_b$ extends to $L$-color $G$, which is false. \end{claimproof}

Claim \ref{AtMostTwoForEachBLp1} implies that, if at least one of \ref{T4PartA}-\ref{T4PartC} of Definition \ref{BaseColoringDefn} holds, then there is a $b\in L(p_1)$ with $|\mathcal{F}_b|\leq 2$. As $G$ is a counterexample to \ref{LabCrownNonEmpt4}, it follows that $G$ violates all of \ref{T4PartA}-\ref{T4PartC}. Since $|V(C)|>4$, we let $C\setminus\mathring{P}:=p_0u_1\cdots u_tp_1$ for some $t\geq 1$.

\begin{claim}\label{X1EndChordCMinPCL3}  No color of $L(p_1)$ is $(P_1, G^{P_1})$-universal. \end{claim}

\begin{claimproof}  Suppose there is such a $b\in L(p_1)$. 

\vspace*{-8mm}
\begin{addmargin}[2em]{0em}
\begin{subclaim} $x^0=x^1$ and, in particular, $G^{P_0}$ is a broken wheel with principal path $P_0$. \end{subclaim}

\begin{claimproof} Suppose not. Since \ref{T4PartA} is violated, there exist distinct $\pi, \pi'\in\mathcal{F}_b$. As $b$ is $(P_1, G^{P_1})$-universal, each of $(a, \pi(q_0), \pi(q_1))$ and $(a, \pi'(q_0), \pi'(q_1))$ is an $L$-coloring of $P-p_1$ which does not extend to $L$-color $G^{P_0}\cup H$. By Theorem \ref{EitherBWheelOrAtMostOneColThm}, $G^{P_0}\cup H$ is a broken wheel with principal path $P-p_1$, so $x_0=x_1$ and the subclaim holds. \end{claimproof}\end{addmargin}

Let $x_0=x_1=x^*$ for some $x^*\in V(C\setminus P)$. Since $b$ is $(P_1, G^{P_1})$-universal, it follows that, for each $\pi\in\mathcal{F}_b$, we have $\Lambda_{G^{P_0}}(a, \pi(q_0), \bullet)=\{\pi(q_1)\}$. In particular, for any distinct $\pi, \pi'\in\mathcal{F}_b$, we have $\pi(q_0)\neq\pi'(q_0)$. 

\vspace*{-8mm}
\begin{addmargin}[2em]{0em}
\begin{subclaim}\label{DistinctSigmaSigma'Avoid} There exist $\sigma, \sigma'\in\mathcal{F}_b$ with $S_{\sigma}\neq S_{\sigma'}$. \end{subclaim}

\begin{claimproof} Suppose not. Since \ref{T4PartC} is violated, $|V(G^{P_0})|$ is even. Again, since $|\mathcal{F}_b|>1$, there exist distinct $\pi, \pi'\in\mathcal{F}_b$, so $\pi(q_0)=\pi'(q_1)$ and $\pi'(q_0)=\pi(q_1)$. But then, by \ref{BWheel1Lb} \ref{BWheel1B} of Theorem \ref{BWheelMainRevListThm2}, $|V(G^{P_0})|$ is odd. \end{claimproof}\end{addmargin}

Let $\sigma, \sigma'$ be as in Subclaim \ref{DistinctSigmaSigma'Avoid}.  Now, if $\sigma(q_1)\neq\sigma'(q_1)$, then, again, it follows from \ref{BWheel1Lb} \ref{BWheel1B} of Theorem \ref{BWheelMainRevListThm2} applied to $G^{P_0}$ that $S_{\sigma}=S_{\sigma'}$, which is false. Thus, $\sigma(q_1)=\pi'(q_1)=r$ for some $r\in L(q_1)$ and it follows from \ref{BWheel1Lb} \ref{BWheel1A} of Theorem \ref{BWheelMainRevListThm2} that $G^{P_0}-q_0$ is a path of even length. Since  \ref{T4PartB} is violated, there is a $\tau\in\mathcal{F}_b$ with $\tau(q_1)\neq r$. Again, since $b$ is $(P_1, G^{P_1})$-universal, we have $\Lambda_{G^{P_0}}(a, \tau(q_0), \bullet)=\{\tau(q_1)\}$, so $\tau$ restricts to a different $L$-coloring of $p_0q_0$ than either of $\sigma, \sigma'$, contradicting \ref{BWheel2Lb} of Theorem \ref{BWheelMainRevListThm2}. \end{claimproof}

\begin{Claim}\label{IfNotTriThenIntersecBdCL} $G^{P_1}$ is a broken wheel with principal path $P_1$, where $L(p_1)=L(u_t)$ and $G^{P_1}-q_1$ has length at most two. Furthermore, if $G^{P_1}$ has length precisely two, then $|L(u_{t-1})\cap L(p_1)|\geq 2$. \end{Claim}

\begin{claimproof} As $|L(p_1)|>1$ and no color of $L(p_1)$ is $(P_1, G^{P_1})$-universal, it follows from Theorem \ref{EitherBWheelOrAtMostOneColThm} that $G^{P_1}$ is a broken wheel with principal path $P_1$. Furthermore, $L(p_1)=L(u_t)$. Suppose $|V(G^{P_1})|>4$. By \ref{PropCor1} of Proposition \ref{CorMainEitherBWheelAtM1ColCor},  $L(p_1)=L(u_{t-1})$. Let $G^{\dagger}:=G\setminus\{p_1, u_t\}$ and $P^{\dagger}:=p_0q_0q_1u_{t-1}$. Let $C^{\dagger}$ be the outer face of $G^{\dagger}$. By the minimality of $|V(G)|$, there is an $L$-coloring $\phi$ of $\{p_0, u_{t-1}\}$ satisfying one of \ref{T4PartA}-\ref{T4PartC} applied to the rainbow $(G^{\dagger}, C^{\dagger}, P^{\dagger}, L)$. Let $b=\phi(u_{t-1})$. Then $b\in L(p_1)$ and, for any $\pi\in\mathcal{F}_b$, the union $\pi\cup\phi$ is a proper $L$-coloring of its domain which extends to $L$-color $V(P\cup G^{P_1})$, since $q_0u_{t-1}\not\in E(G)$ and the same color is used on $u_{t-1}, p_1$. As $x_1\not\in\{u_{t-1}, u_t, p_1\}$, there is a triangle-type $(P,G)$-obstruction if and only if there is a triangle-type $(P^{\dagger}, G^{\dagger})$-obstruction. Furthermore, $G^{\dagger}$ is still short-inseparable, and $G^{P_1}$ is a broken wheel with an even number of vertices if and only if $G^{P_1}\setminus\{u_t, p_1\}$ is a broken wheel with an even number of vertices, so, by minimality, there is a $\pi\in\mathcal{F}_b$ such that the $L$-coloring $(a, \pi(q_0), \pi(q_1), b)$ of $P^{\dagger}$ extends to $L$-color $G^{\dagger}$, and $\pi\cup\psi$ extends to $L$-color $G$, a contradiction. \end{claimproof}

Over the course of Claims \ref{IfX0X1SameObsK1NotDel}-\ref{X0X1DistinctNoPGObsT4}, we now rule out the possibility that $x_0=x_1$. 

\begin{claim}\label{IfX0X1SameObsK1NotDel} If $x_0=x_1$,  then both of the following hold. 
\begin{enumerate}[label=\arabic*)]
\itemsep-0.1em
\item $G^{P_0}$ is a broken wheel with an even number of vertices; AND
\item $G^{P_1}$ is broken wheel with four vertices.
\end{enumerate} \end{claim}

\begin{claimproof} Suppose $x_0=x_1$. We first prove 1). Suppose 1) does not hold. Since $|L(p_0)|=1$, it follows from Theorem \ref{SumTo4For2PathColorEnds} that there is an $S\subseteq L(u_t)$ with $|S|=2$, where any $L$-coloring of $\{p_0, x_0\}$ using a color of $S$ on $x_0$ is $(P_0, G^{P_0})$-sufficient. If $G^{P_1}$ is a triangle, then, for each $b\in  L(p_1)\setminus S$ and $\pi\in\mathcal{F}_b$, we have $S_{\pi}=S$ and \ref{T4PartC} is satisfied, which is false. As $L(p_1)\setminus S\neq\varnothing$, it follows that $G^{P_1}$ is not a triangle. By Claim \ref{IfNotTriThenIntersecBdCL}, $G^{P_1}$ is a broken wheel with four vertices, i.e $x_0=x_1=u_{t-1}$, and furthermore, $S\cap L(p_1)\neq\varnothing$. Let $s\in S\cap L(p_1)$ and consider $\mathcal{F}_s$. Since $G^{P_1}-q_1$ has even length and \ref{T4PartB} is violated, there is a $\pi\in\mathcal{F}_s$ with $\pi(q_0)\neq s$. Since $\pi(q_1)\neq s$ as well, $\pi$ extends to an $L$-coloring of $V(P\cup G^{P_0})$ using $s$ on both of $u_{t-1}, p_1$, so $\pi$ extends to $L$-color $G$, which is false. This proves 1). Now we prove 2). Suppose 2) does not hold. Thus, by Claim \ref{IfNotTriThenIntersecBdCL}, $G^{P_1}$ is a triangle, so $x_0=x_1=u_t$. 

\vspace*{-8mm}
\begin{addmargin}[2em]{0em}
\begin{subclaim}\label{utut-1SameListL} $L(u_t)=L(u_{t-1})$. \end{subclaim}

\begin{claimproof} Suppose not. As both sets have size three, there is a $c\in L(u_t)\setminus L(u_{t-1})$. Choose a $b\in L(p_1)\setminus\{c\}$. As \ref{T4PartB} is violated,  there is a $\pi\in\mathcal{F}_b$ with $\pi(q_1)\neq c$. Thus, either $\pi(q_0)=c$ or $c\in L_{\pi}(u_t)$. In any case, since $G^{P_1}$ is a triangle, it follows that $\pi$ extends to $L$-color $G$, a contradiction. \end{claimproof}\end{addmargin}

Let $L(p_1)=\{b_0, b_1, b_2\}$. Now, since $L(p_1)=L(u_t)$ and $G^{P_1}$ is a triangle, but \ref{T4PartC} is violated, it follows that, for each $k=0,1,2$, there is a $\pi_k\in\mathcal{F}_{b_k}$ such that $S_{\pi_k}\neq L(u_t)\setminus\{b_k\}$. In particular, $L_{\pi_k}(u_t)$ is a nonempty subset of $L(u_t)\setminus\{b_k\}$. For each $k=0,1,2$, let $s_k:=\pi_k(q_0)$. Since $G^{P_1}$ is a triangle, we have $\Lambda_{G^{P_0}}(a, s_k, \bullet)\cap L_{\pi_k}(u_t)=\varnothing$. As $G^{P_0}$ is a broken wheel but not a triangle, we have $a\in L(u_1)$ and $\{s_0, s_1, s_2\}\subseteq L(u_1)$. Thus, we suppose without loss of generality that $s_0=s_1=s$ for some $s\in L(q_1)$. If $|L_{\pi_0}(u_t)\cup L_{\pi_1}(u_t)|\geq 2$, then one of $\pi_0, \pi_1$ extends to $L$-color $G$, so $L_{\pi_0}(u_t)=L_{\pi_1}(u_t)=\{r\}$ for some $r\in L(u_t)\setminus\{s\}$, and thus $\{\pi_0(q_1), b_0\}=\{\pi_1(q_1), b_1\}=\{b_0, b_1\}$. It follows that $r=b_2$ and $\pi_k(q_1)=b_{1-k}$ for each $k=0,1$, so $s\not\in L(u_t)$. But then, by Subclaim \ref{utut-1SameListL}, $s\not\in L(u_{t-1})$, so the $L$-coloring $(a, s, b_2)$ of $P_0$ extends to $L$-color $G^{P_0}$, a contradiction. \end{claimproof}

\begin{claim}\label{X0X1DistinctNoPGObsT4} $x_0\neq x_1$. \end{claim}

\begin{claimproof} Suppose not. By Claim \ref{IfX0X1SameObsK1NotDel} $x_0=x_1=u_{t-1}$ and $G^{P_0}$ is a broken wheel and $G^{P_0}-q_0$ is a path of even length. 

\vspace*{-8mm}
\begin{addmargin}[2em]{0em}
\begin{subclaim}\label{EachbUt-1AvoP1} For each $b\in L(u_{t-1})\cap L(p_1)$, there is an $L$-coloring $\psi$ of $V(P_0)$ with $\psi(q_0)=b$, where $\psi$ does not extend to $L$-color $G$. \end{subclaim}

\begin{claimproof} Since $G^{P_1}-q_1$ has length two and \ref{T4PartB} is violated, there is a $\pi\in\mathcal{F}_b$ with $\pi(q_0)\neq b$. As $\pi(q_1)\neq b$ as well, it follows that the $L$-coloring $(a, \pi(q_1), b)$ of $p_0q_0x_0$ does not extend to $L$-color $G^{P_0}$, or else $\pi$ extends to $L$-color $G$, as the same color is used on $u_{t-1}, p_1$. \end{claimproof}\end{addmargin}

By Theorem \ref{SumTo4For2PathColorEnds}, there is a $\phi\in\textnormal{End}(P_0, G^{P_0})$. Let $r:=\phi(p_0)$. By Subclaim \ref{EachbUt-1AvoP1},  $r\not\in L(p_1)$. Thus, by Claim \ref{IfNotTriThenIntersecBdCL}, $|L(u_{t-1})\cap L(p_1)|=2$. Let $L(p_1)=\{b_0, b_1, b_2\}$, where $b_0, b_1\in L(u_{t-1})$. Thus, for each $k=0,1$, there is an $L$-coloring $\psi_k$ of $V(P_0)$ which does not extend to $L$-color $G^{P_0}$, where $\psi_k(u_{t-1})=b_k$. Let $T:=\{\psi_0(q_0), \psi_1(q_0)\}$. As $G^{P_0}$ is not a triangle, it follows from \ref{PropCor2} of Proposition \ref{CorMainEitherBWheelAtM1ColCor} that $T=\{b_0, b_1\}$. That is, for each $k=0,1$, $\psi_k(q_0)=b_{1-k}$. Consider $\mathcal{F}_{b_0}$. As \ref{T4PartB} is violated, there is a $\sigma\in\mathcal{F}_{b_0}$ with $\sigma(q_1)\neq r$. As $L(p_1)=L(u_t)$, we have $r\not\in L(u_t)$. Now, if $\sigma(q_0)\neq r$, then, by our choice of $r$, we get that $\sigma$ extends to an $L$-coloring of $V(P\cup G^{P_0})$ using $r$ on $u_{t-1}$, and thus, as $r\not\in L(u_t)$, $\sigma$ extends to $L$-color $G$, which is false. Thus, $\sigma(q_0)=r$. Note that $\Lambda_{G^{P_0}}(a, r, \bullet)=T$ by \ref{BWheel2Lb} of Theorem \ref{BWheelMainRevListThm2}. As $\sigma(q_1)\neq b_0$, it follows that $\sigma$ extends to an $L$-coloring of $V(P\cup G^{P_0})$ using $b_0$ on $u_{t-1}$. As the same color is used on $u_{t-1}, p_1$, it follows that $\sigma$ extends to an $L$-coloring of $G$, which is false. This proves Claim \ref{X0X1DistinctNoPGObsT4}. \end{claimproof}

Since $x_0\neq x_1$, it follows from Claim \ref{AtMostTwoForEachBLp1} that $G^M$ is a wheel with central vertex $w$, where $C^M$ has odd length. Let $P^*:=x_0wx_1$, so that $G^M\setminus\{q_0, q_1\}=G^{P^*}$. 

\begin{claim}\label{TwoElSubNotConstLastCL} For any 2-element subset $T$ of $L(p_1)$, the set of colorings $\bigcup_{b\in T}\mathcal{F}_b$ is not constant on $p_1$. \end{claim}

\begin{claimproof} Let $T:=\{b_0, b_1\}$ and suppose there is a $c\in L(q_1)$ such that each element of $\mathcal{F}_{b_0}\cup\mathcal{F}_{b_1}$ uses $c$ on $q_1$. For each $k=0,1$, let $T_k:=\{\phi(q_0): \phi\in\mathcal{F}_{b_k}\}$. As $G$ is a counterexample, $|T_k|\geq 2$ for each $k=0,1$. Now, there is a $(P_0, G^{P_0})$-sufficient $L$-coloring $\sigma$ of $\{p_0, x_0\}$. This is immediate if $G^{P_0}$ is an edge. Otherwise, it follows from Theorem \ref{SumTo4For2PathColorEnds}. Let $r:=\sigma(x_0)$. Now, for each $k=0,1$, there is an $L$-coloring $\psi_k$ of $G^{P_1}$ using $c, b_k$ on $q_1, p_1$ respectively.  

\vspace*{-8mm}
\begin{addmargin}[2em]{0em}
\begin{subclaim}\label{S0S1IntersectSolelySigmax0} $G^{P_0}$ is not an edge, and $T_0\cap T_1=\{r\}$. Furthermore, $\psi_0(x_1)\neq\psi_1(x_1)$.

\end{subclaim}

\begin{claimproof} Suppose first that there is a $k\in\{0,1\}$ with $r\not\in T_k$. As no element of $\mathcal{F}_{b_k}$ extends to $L$-color $G$, it follows that $\Lambda_{G^{P^*}}(r, \bullet, \psi_k(x_1)\cap (L(w)\setminus\{c, r, \psi_k(x_1)\})=\varnothing$. But since $|V(G^{P^*})|$ is even, this contradicts \ref{PropCor3} of Proposition \ref{CorMainEitherBWheelAtM1ColCor}. Thus, $r\in T_0\cap T_1$. If $G^{P_0}$ is an edge, then $r=a$, which is false, as $a\not\in T_0\cup T_1$, so $G^{P_0}$ is not an edge. Now suppose toward a contradiction that $|T_0\cap T_1|\geq 2$. Every chord of the outer cycle of $G^{P_0}\cup G^M$ is incident to $q_0$, but $G^{P_0}\cup G^M$ is not a broken wheel. Since $a, c\not\in T_0\cap T_1$, it follows from Theorem \ref{EitherBWheelOrAtMostOneColThm} that there is an $L$-coloring $\tau$ of $G^{P_0}\cup G^M$ with $\tau(q_0)\in T_0\cap T_1$ and $\tau(q_1)=c$. Since $\Lambda_{G^{P_1}}(\tau(x_1), c, \bullet)\cap T\neq\varnothing$, at least one element of $\bigcup_{b\in T}\mathcal{F}_b$ extends to $L$-color $G$, which is false, so $T_0\cap T_1=\{r\}$. Finally, suppose $\psi_0(x_1)=\psi_1(x_1)=d$ for some $d\in L(x_1)$. Possibly $d=r$, but $G^{P^*}-w$ has even length, and, by \ref{PropCor3} of Proposition \ref{CorMainEitherBWheelAtM1ColCor}, $\Lambda_{G^{P^*}}(r, \bullet, d)\cap (L(w))\cap (L(w)\setminus\{c, d, r\})\neq\varnothing$. Since $|T_0\cup T_1|\geq 2$, it follows that $\sigma$ extends to an $L$-coloring $\tau$ of $G^{P_0}\cup G^M$ using $c,d$ on the respective vertices $q_1, x_1$ and using a color of $T_0\cup T_1$ on $q_0$. Thus, there is a $k\in\{0,1\}$ such that at least one element of $\mathcal{F}_k$ extends to $L$-color $G$, which is false.  \end{claimproof}\end{addmargin}

Now, it follows from Subclaim \ref{S0S1IntersectSolelySigmax0} that $r\in L(q_0)\setminus\{a\}$, and since $x_0q_1\not\in E(G)$, there is an $L$-coloring $\phi$ of $V(G^{P_0})\cup\{q_1\}$ with $\phi(q_0)=r$ and $\phi(q_1)=c$. Let $W:=L(w)\setminus\{r, \phi(q_0), c\}$. Since $r\in T_0\cap T_1$ and no element of $\bigcup_{b\in T}\mathcal{F}_b$ extends to $L$-color $G$, it follows that  there is no $L$-coloring of $G^{P^*}$ using $\phi(x_0)$ on $x_0$, a color of $W$ on $w$, and a color of $\{\psi_0(x_1), \psi_1(x_1)\}$ on $x_1$. Since $\psi_0(x_1)\neq\psi_1(x_1)$ and $|W|\geq 2$, this implies that $W=\{\psi_0(x_1), \psi_1(x_1)\}$ and furthermore, letting $u$ be the unique neighbor of $x_0$ on the path $G^{P^*}-w$, we have $L(u)=\{\phi(x_0)\}\cup W$. In particular, $r\not\in L(u)$ and $r\in L(w)\setminus W$. We now choose an arbitrary $\tau\in\bigcup_{b\in T}\mathcal{F}_b$ with $\tau(q_0)\neq r$. Say $\tau\in\mathcal{F}_{b_0}$ for the sake of definiteness. We produce a contradiction by showing that $\tau\cup\psi_0$ extends to $L$-color $G$. Choose an arbitrary $r'\in\Lambda_{G^{P_0}}(a, \tau(q_0), \bullet)$. We have either $r'=r$ or $r\in L(w)\setminus\{r', \tau(q_0), c\}$. In any case, since $G^{P^*}-w$ is not a triangle and since $r\not\in L(u)$ and $r\neq\psi_0(x_1)$, it follows that $\tau$ extends to an $L$-coloring of $V(P\cup G^{P_0}\cup G^M)$ using $\psi_0(x_1)$ on $x_1$, so $\tau\cup\psi_0$ extends to $L$-color $G$, a contradiction. This proves Claim \ref{TwoElSubNotConstLastCL}. \end{claimproof}

Now, it follows from \ref{LabCrownNonEmpt3} that there is a $(R_0, G^{R_0})$-sufficient $L$-coloring $\phi$ of $\{p_0, x_1\}$. Let $c:=\phi(x_1)$. Thus, for any $b\in L(p_1)$ and any $\psi\in\mathcal{F}_b$, either the union $\psi\cup\phi$ is not a proper $L$-coloring of $V(P)\cup\{x_1\}$ or the $L$-coloring $(c, \phi(q_1), b)$ of $x_1q_1p_1$ does not extend to $L$-color $G$. By Claim \ref{TwoElSubNotConstLastCL}, there is a $b\in L(p_1)\setminus\{c\}$ and a $\psi\in\mathcal{F}_b$ with $\psi(q_1)\neq c$. Since $q_0x_1\not\in E(G)$, the union $\phi\cup\psi$ is a proper $L$-coloring of its domain, so the $L$-coloring $(c, \psi(x_1), b)$ of $x_1q_1p_1$ does not extend to $L$-color $G^{P_1}$. It follows that $G^{P_1}$ is not a triangle. By Claim \ref{IfNotTriThenIntersecBdCL}, $|V(G^{P_1})|=4$ and $L(u_t)=L(p_1)$, so $x_1=u_{t-1}$ and since $(c, \psi(x_1), b)$ does not extend to $L$-color $G^{P_1}$, we have $c\in L(u_t)$, so $c\in L(p_1)$. Now, for any $\psi\in\mathcal{F}_c$, we have $\psi(q_1)\neq c$ and the union $\phi\cup\psi$ is a proper $L$-coloring of domain using the same color on $u_{t-1}, x_1$, so $\phi\cup\psi$ extends to $L$-color $G^{P_1}$ as well, which is false, as indicated above. This proves \ref{LabCrownNonEmpt4} and completes the proof of Theorem \ref{CombinedT1T4ThreePathFactListThm}. \end{proof}

\section{The proof of Theorem \ref{MainHolepunchPaperResulThm}}\label{MainResHolepunchSec}

With the intermediate results of the previous sections in hand, we now prove Theorem  \ref{MainHolepunchPaperResulThm}. We first note that, in the statement of Theorem \ref{MainHolepunchPaperResulThm}, we cannot drop the condition that $q_0, q_1$ have no common neighbor in $C\setminus P$. If we drop this condition, then we have the counterexample in Figure \ref{DropConditionHolepunchCounterFigM}.

\begin{center}
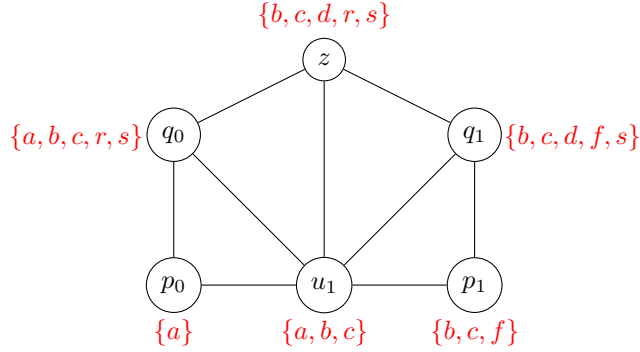
\begin{tikzpicture}
\node[shape=circle,draw=black] [label={[xshift=-0.0cm, yshift=-1.3cm]\textcolor{red}{$\{a\}$}}] (p0) at (0,0) {$p_0$};
\node[shape=circle,draw=black] [label={[xshift=-0.0cm, yshift=-1.3cm]\textcolor{red}{$\{a, b, c\}$}}] (u1) at (2, 0) {$u_1$};
\node[shape=circle,draw=black] [label={[xshift=-0.0cm, yshift=-1.3cm]\textcolor{red}{$\{b, c, f\}$}}] (p1) at (4, 0) {$p_1$};
\node[shape=circle,draw=black] [label={[xshift=-1.3cm, yshift=-0.7cm]\textcolor{red}{$\{a, b, c, r, s\}$}}] (q0) at (0,2) {$q_0$};
\node[shape=circle,draw=black] [label={[xshift=1.3cm, yshift=-0.7cm]\textcolor{red}{$\{b, c, d, f, s\}$}}] (q1) at (4,2) {$q_1$};
\node[shape=circle,draw=black] [label={[xshift=-0.0cm, yshift=0cm]\textcolor{red}{$\{b,c,d,r,s\}$}}] (z) at (2,3) {$z$};

 \draw[-] (p1) to (u1) to (p0) to (q0) to (z) to (q1) to (p1);
\draw[-] (z) to (u1);
\draw[] (q0) to (u1) to (q1);
\end{tikzpicture}\captionof{figure}{Theorem \ref{MainHolepunchPaperResulThm} is false if $q_0, q_1$ are allowed to have a common neighbor in $V(C\setminus\mathring{P})$}\label{DropConditionHolepunchCounterFigM}\end{center}

Letting $G$ be the graph above, it is straightforward to verify that there is no partial $L$-coloring $\phi$ of $G\setminus\{q_0, q_1\}$ whose domain contains $p_0, p_1, z$ such that both of $q_0, q_1$ have at least three leftover colors, where either $u_1$ is also colored or the inertness condition is satisfied. We now prove Theorem \ref{MainHolepunchPaperResulThm}, which we restate below. 

\begin{thmn}[\ref{MainHolepunchPaperResulThm}] Let $(G, C, P, L)$ be a rainbow, where $P$ has length four, the endpoints of $\mathring{P}$ have no common neighbor in $C\setminus P$, each internal vertex of $P$ has an $L$-list of size at least five, and at least one endpoint of $P$ has a list of size at least three. Then $\textnormal{Crown}_{L}(P, G)\neq\varnothing$.  \end{thmn} 

\begin{proof} Suppose not and let $G$ be a vertex-minimal counterexample to the theorem, where $P:=p_0q_0zq_1p_1$. By adding edges to $G$ if necessary, we suppose that every face of $G$, except $C$, is bounded by a triangle. This is permissible as we can add edges until this holds without creating a common neighbor to $q_0, q_1$ in $C\setminus P$. By removing colors from some lists if necessary, we suppose that $|L(p_0)|=1$ and $|L(p_1)|=3$ and furthermore, for each $v\in V(C\setminus P)$, $|L(v)|=3$, and, for each $v\in V(G\setminus C)\cup\{z\}$, $|L(v)|=5$. We now introduce the following notation.

\begin{enumerate}[label=\arabic*)]
\item We let $C\setminus\mathring{P}:=p_0u_1\ldots u_tp_1$ for some $t\geq 1$. For each $i\in\{0,1\}$, we let $x_i$ be the unique vertex of $N(q_i)\cap V(C\setminus\mathring{P})$ which is farthest from $p_i$ on the path $C\setminus\mathring{P}$. We let $P_0$ be the path $p_0q_0x_0$ and $P_1$ be the path $x_1q_1p_1$, where, for each $i\in\{0,1\}$, if $x_i=p_i$, then $P_i$ is an edge.
\item For any subgraph $H$ of $G$, we let $\mathcal{X}(H)$ be the set of $L$-colorings of $V(H)$ which do not extend to $L$-color $G$.
\end{enumerate}

We break the proof of Theorem \ref{MainHolepunchPaperResulThm} into several subsections of Section \ref{MainResHolepunchSec}. In particular, the main intermediate steps in the proof of Theorem \ref{MainHolepunchPaperResulThm} are the following: In Subsections \ref{CInduCycleSubS}, we show that every chord of $C$ is incident to precisely one of $q_0, q_1$. In Subsection \ref{CommNZXISubSFin}, we show that, for each $i\in\{0,1\}$, if there is a $w\in N(z)\cap N(q_i)$, then $w$ has no neighbors in $C\setminus\mathring{P}$, except possibly $x_i$. In Subsection \ref{CommNQ1ZQSubS}, we show that $q_0, q_1$ have no common neighbor other than $z$. In Subsection \ref{SubsecShowX0X1NoComm}, we show that $x_0, x_1$ have no common neighbor in $G\setminus C$. We repeatedly make use of the following fact: For each $i\in\{0,1\}$,  endpoint $y$ of $P_i$ and $c\in L(y)$, if either $P_i$ is an edge or the other endpoint of $P_i$ has a list of size at least three, then there is a $(P_i, G_i)$-sufficient $L$-coloring of $\{p_i, x_i\}$ using $c$ on $y$. This is immediate if $P_i$ is an edge, otherwise it follows from Theorem \ref{SumTo4For2PathColorEnds}. 

\subsection{Preliminary Restrictions}

Firstly, applying Theorem \ref{thomassen5ChooseThm} and Corollary \ref{CycleLen4CorToThom}, it immediate follows from the minimality of $|V(G)|$ that $G$ is short-inseparable, and furthermore, every chord of $C$ is incident to one of $\{q_0, z, q_1\}$. 

\begin{Claim}\label{PInducedNoPGSuff1}
Both of the following hold. 
\begin{enumerate}[label=\arabic*)]
\itemsep-0.1em
\item $P$ is an induced path with $x_0\neq x_1$, and any $L$-coloring of $\{p_0, z, p_1\}$ extends to an element of $\mathcal{X}(P)$; AND
\item Let $i\in\{0,1\}$ and $\phi$ be an $L$-coloring of $\{p_i, x_i, z, p_{1-i}\}$ whose restriction to $\{p_i, x_i\}$ is $(P_i, G^{P_i})$-sufficient. Then either $|L_{\phi}(q_i)|=2$ or $\phi$ extends to an element of $\mathcal{X}(G^{P_i}\cup P)$.
\end{enumerate} \end{Claim}

\begin{claimproof} We first show that $p_0q_1, p_1q_0\not\in E(G)$. Suppose there is a $k\in\{0,1\}$ with $p_kq_{1-k}\in E(G)$. Let $G=G'\cup G''$ be the natural $p_kq_{1-k}$-partition of $G$, where $z\in V(G')$ and $p_{1-k}\in V(G'')$. By Theorem \ref{SumTo4For2PathColorEnds}, there is a $(p_kq_{1-k}p_{1-k}, G'')$-sufficient $L$-coloring $\phi$ of $\{p_0, p_1\}$. Possibly $p_kz\in E(G)$ as well, but, in any case, we have $|L_{\psi}(z)|\geq 4$ and $|L_{\psi}(q_k)|\geq 4$, and $|L_{\psi}(q_{1-k})|\geq 3$, so $\psi$ extends to an $L$-coloring $\psi^*$ of $\{p_0, z, p_1\}$ where each of $q_0, q_1$ has an $L_{\psi^*}$-list of size at least three. As $G$ is short-inseparable, $\psi^*$ is $(P,G)$-sufficient, contradicting our assumption that $G$ is a counterexample. Thus, $p_0q_1, p_1q_0\not\in E(G)$, so it follows from our assumption on $G$ that $x_0\neq x_1$, and it also follows that any $L$-coloring of $\{p_0, z, p_1\}$ extends to an element of $\mathcal{X}(P)$.

Now suppose there is a $k\in\{0,1\}$ with $zp_k\in E(G)$. As $|L(z)|=5$, it follows from \ref{LabCrownNonEmpt4} of Theorem \ref{CombinedT1T4ThreePathFactListThm} that there is an $(p_0zq_{1-k}p_{1-k}, G-q_k)$-sufficient $L$-coloring $\psi$ of $\{p_0, z, p_1\}$, and $\psi$ is $(P,G)$-sufficient, which is false, as shown above, so $p_0, p_1\not\in N(z)$. Now suppose $q_0q_1\in E(G)$. Thus, $N(z)=\{q_0, q_1\}$. By \ref{LabCrownNonEmpt3} of of Theorem \ref{CombinedT1T4ThreePathFactListThm}, there is a $(p_0q_0q_1p_1, G-z)$-sufficient $L$-coloring $\psi$ of $\{p_0, x_1, p_1\}$. As $x_0\neq x_1$, we have $|L_{\psi}(q_0)|\geq 4$ and $|L_{\psi}(q_1)|\geq 3$. Thus, $\psi$ extends to an $L$-coloring $\psi^*$ of $\{p_0, x_1, p_1, z\}$ where each of $q_0, q_1$ has an $L_{\psi^*}$-list of size at least three, and $\psi^*$ is $(P,G)$-sufficient, contradicting our assumption that $G$ is a counterexample. To finish, it suffices to show that $p_0p_1\not\in E(G)$. Suppose $p_0p_1\in E(G)$. Thus, $C$ is an induced 5-cycle, and it is straightforward to check using Theorem \ref{BohmePaper5CycleCorList} that there is a $(P,G)$-sufficient $L$-coloring of $\{p_0, z, p_1\}$, which is false, as shown above. This proves 1). Now, 2) follows immediately from 1), together with the fact that $G$ is a counterexample. \end{claimproof}

\begin{claim}\label{IfGP1NotEdgeThen} 
 If $G^{P_1}$ is not an edge, then either there is an almost $(P_1, G^{P_1})$-universal color of $L(p_1)$ or $G^{P_1}$ is a triangle with $L(p_1)=L(v_1)$.
\end{claim}

\begin{claimproof} Suppose not. Thus, $G^{P_1}$ is not an edge or a triangle, and no color of $L(p_1)$ is almost $(P_1, G^{P_1})$-universal. Since $|L(p_1)|=3$, it follows that $G^{P_1}$ is a broken wheel with principal path $P_1$ and $L(p_1)=L(u_t)$, and $|V(G^{P_1})|\geq 4$. Thus, by \ref{BWheel3Lb} of Theorem \ref{BWheelMainRevListThm2}, $L(p_1)=L(u_t)=L(u_{t-1})$. In particular, $u_{t-1}\in N(q_1)$. Let $P^*:=p_0q_0zq_1u_{t-1}$. Note that $q_0, q_1$ have no common neighbor in $C^{P_*}\setminus P^*$, so, by minimality, there is a $\phi\in\textnormal{Crown}(P^*, G^*)$. In particular, $u_{t-1}\in\textnormal{dom}(\phi)$ and, since $|L(u_{t-1})=L(p_1)$, $\phi$ extends to an $L$-coloring $\psi$ of $\textnormal{dom}(\phi)\cup\{p_1\}$ using the same color on $u_{t-1}, p_1$. Thus, $\psi$ is $(P, G)$-sufficient, and $L_{\psi}(q_1)=L_{\phi}(q_1)$, so each of $q_0, q_1$ has an $L_{\psi}$-list of size at least three, which is false, as $G$ is a counterexample. \end{claimproof}

\subsection{All chords of $C$ are incident to one of $q_0, q_1$}\label{CInduCycleSubS}

To show that there is no chord of $C$ incident to $z$, we first temporarily introduce the following notation: If $z$ is incident to a chord of $C$, then we define vertices $v_0, v_1$ and paths $Q_0, Q_1, M, R$ as follows:
\begin{enumerate}[label=\arabic*)] 
\itemsep-0.1em
\item For each $i\in\{0,1\}$, $v_i$ is the unique vertex of $N(z)\cap V(C\setminus\{q_0, q_1\})$ which is closest to $p_i$ on the path $C\setminus\mathring{P}$, and $Q_i$ is the path $p_iq_izv_i$.
\item $M$ is the path $v_0zv_1$ and $R$ is the path $v_0zq_1p_1$. In particular, if $v_0=v_1$, then $M$ is an edge.
\end{enumerate}

Over the course of Claims \ref{zPreciselyOneMPiQiNotSize3}-\ref{CycCIndMTh}, we now rule out any chords of $C$ incident to $z$. 

\begin{Claim}\label{zPreciselyOneMPiQiNotSize3} If $z$ is incident to at least one chord of $C$, then 
\begin{enumerate}[label=\Alph*)]
\itemsep-0.1em
\item\label{AofInterCL} For any $\phi\in\textnormal{End}(Q_0, G^{Q_0})$, there is no element of $\textnormal{End}(R, G^R)$ using $\phi(v_0)$ on $v_0$; AND
\item\label{BofInterCL} $v_0=x_0$.
\end{enumerate} \end{Claim}

\begin{claimproof} Suppose $z$ is incident to at least one chord of $C$, but $v_0\neq x_0$. We first prove \ref{AofInterCL}. Suppose there exist such $\phi, \psi$. Then $\phi\cup\psi$ is a proper $L$-coloring of $\{p_0, v_0, p_1\}$. From our assumption on $P$, and the fact that $P$ is induced, it follows $|L_{\phi\cup\psi}(z)|\geq 4$ and that at most one of $q_0, q_1$ is adjacent to $v_0$, so $\phi\cup\psi$ extends to an $L$-coloring $\pi$ of $\{p_0, v_0, p_1, z\}$ such that $|L_{\pi}(q_0)|\geq 3$ and $|L_{\pi}(q_1)|\geq 3$. Furthermore, $\pi$ is $(P,G)$-sufficient, contradicting our assumption that $G$ is a a counterexample. This proves \ref{AofInterCL}. Now we prove \ref{BofInterCL}. By definition of $v_0$, no chord of the outer cycle of $G^{Q_0}$ is incident to $z$, so, by \ref{LabCrownNonEmpt3} of Theorem \ref{CombinedT1T4ThreePathFactListThm}, there is a $(Q_0, G^{Q_0})$-sufficient $L$-coloring $\phi$ of $\{p_0, v_0\}$. Let $b:=\phi(v_0)$.

\vspace*{-8mm}
\begin{addmargin}[2em]{0em} 
\begin{subclaim}\label{QM1SuffPsiV0p1z} There is a $(R, G^R)$-sufficient $L$-coloring $\psi$ of $\{v_0, p_1, z\}$, where $|L_{\psi}(q_1)|\geq 3$ and $\psi(v_0)=b$. \end{subclaim}

\begin{claimproof} Suppose first that $v_0\not\in N(q_1)$. In that case, we just need to find a $(R, G^R)$-sufficient $L$-coloring of $\{v_0, p_1, z\}$  using $b$ on $v_0$. As $|L(z)|=5$, there is such a coloring by \ref{LabCrownNonEmpt4} of Theorem \ref{CombinedT1T4ThreePathFactListThm} applied to $G^R$. Now suppose that $v_0\in N(q_1)$. Thus, $v_0=v_1=x_1$, so $G^R=G^{Q_1}$. In particular $G^{P_1}$ is not an edge, and there is a $(P_1, G^{P_1})$-sufficient $L$-coloring $\pi$ of $\{x_1, p_1\}$ using $b$ on $x_1$. As $P$ is induced, $|L_{\pi}(z)|\geq 4$, so $\pi$ extends to an $L$-coloring of $\{v_0, p_1, z\}$ which satisfies the subclaim.  \end{claimproof}\end{addmargin}

Let $\psi$ be as in Subclaim \ref{QM1SuffPsiV0p1z}. By assumption, $v_0\neq x_0$, so $|L_{\phi\cup\psi}(q_0)|\geq 3$. Furthermore, $|L_{\phi\cup\psi}(q_1)|\geq 3$ as well, and $\phi\cup\psi$ is $(P,G)$-sufficient, contradicting our assumption that $G$ is a counterexample. This proves Claim \ref{zPreciselyOneMPiQiNotSize3}. \end{claimproof}

The second intermediate result that we need in order to prove the main result of Subsection \ref{CInduCycleSubS} is the following. 

\begin{claim}\label{ifZIncAtleastthenV0X0V1X1} If $z$ is incident to at least one chord of $C$, then, for each $i=0,1$, we have $v_i=x_i$. \end{claim}

\begin{claimproof} Suppose not. By \ref{BofInterCL} of Claim \ref{zPreciselyOneMPiQiNotSize3}, $v_0=x_0$, so $v_1\neq x_1$. As $P$ is induced, $G^{P_0}$ is not an edge. Furthermore, there is a $\phi\in\textnormal{End}(P_0, G^{P_0})$. By \ref{AofInterCL} of Claim \ref{zPreciselyOneMPiQiNotSize3}, no element of $\textnormal{End}(R, G^R)$ uses $\phi(v_0)$ on $v_0$. Thus, it follows from \ref{LabCrownNonEmpt3} of Theorem \ref{CombinedT1T4ThreePathFactListThm} applied to $G^R$ that there is an even-length $(R, G^R)$-obstruction $Q^*$, and that $x_1\neq p_1$, so $G^{P_1}$ is not just an edge. Since $v_1\neq x_1$, $Q^*$ is not a triangle-type $(R, G^R)$-obstruction, so there is a $w\in V(G^R)$ adjacent to all the vertices of the cycle $v_1(C\setminus\mathring{P})x_1q_1z$, where the path $v_1(C\setminus\mathring{P})x_1$ has even length. 

\vspace*{-8mm}
\begin{addmargin}[2em]{0em} 
\begin{subclaim} $v_0=v_1$, i.e $M=G^M$ is an edge. \end{subclaim}

 \begin{claimproof} Suppose not. Let $\mathcal{A}$ be the set of $L$-colorings of $p_0q_0zv_1$ which do not extend to $L$-color $G^{P_0}\cup G^M$. By \ref{LabCrownNonEmpt4} of Theorem \ref{CombinedT1T4ThreePathFactListThm}, there is an $L$-coloring $\phi$ of $\{p_0, v_1\}$ which extends to at most two elements of $\mathcal{A}$. Now, let $c:=\phi(v_1)$ and let $\mathcal{A}'$ be the set of $L$-colorings of $P^{H_1}$ which do not extend to $L$-color $G^{Q_1}$. Since $v_0\neq v_1$, there is no triangle-type $(Q^1, G^{Q_1})$-obstruction, so, again by \ref{LabCrownNonEmpt4} of Theorem \ref{CombinedT1T4ThreePathFactListThm}, since $|L(p_1)|=3$, there is an $L$-coloring $\psi$ of $\{v_1, p_1\}$ using $c$ on $v_1$, where $\psi$ extends to at most one element of $\mathcal{A}'$. As $P$ is induced, we have $|L_{\phi\cup\psi}(z)|\geq 4$. Thus, there is a $d\in L_{\phi\cup\psi}(z)$ such that no element of $\mathcal{A}\cup\mathcal{A}'$ uses $d$ on $z$, and so $\phi\cup\psi$ extends to a $(P, G)$-sufficient $L$-coloring $\pi$ of $\{p_0, v_1, p_1, z\}$. Since $v_0\neq v_1$, and since $v_1q_1\not\in E(G)$ by assumption, each of $q_0, q_1$ has a $L_{\pi}$-list of size at least three, contradicting our assumption that $G$ is a counterexample. \end{claimproof}\end{addmargin}

Now, we have the situation illustrated in Figure \ref{V0SecCase0X1Figure}, where $G^{v_0wx_1}$ is a broken wheel and $|V(G^{v_0wx_1})|$ is even. Let $A$ be the set of $c\in L(z)$ such that there is a $(p_0q_0z, G^{Q_0})$-sufficient $L$-coloring of $\{p_0, z\}$ using $c$ on $z$. Since $|L(z)|=5$, it follows from Theorem \ref{SumTo4For2PathColorEnds} applied to $G^{Q_0}$ that $|A|\geq 3$. 

\vspace*{-8mm}
\begin{addmargin}[2em]{0em} 
\begin{subclaim}\label{AnyLZp1ExtZQ1P1BlSub} For any $L$-coloring $\phi$ of $\{z, p_1\}$ with $\phi(z)\in A$, there is an extension of $\phi$ to an $L$-coloring $\phi^*$ of $zq_1p_1$ such that, for any $d\in\Lambda_{G^{P_1}}(\bullet, \phi^*(q_1), \phi^*(p_1))$, we have $|L(w^*)\setminus\{\phi^*(z), \phi^*(q_1), d\}|=2$. \end{subclaim}

\begin{claimproof} By 1) of Claim \ref{PInducedNoPGSuff1}, $\phi$ extends to an $L$-coloring $\psi$ of $V(P)$ which does not extend to $L$-color $G$. We show that the restriction of $\psi$ to $zq_1p_1$ satisfies the subclaim. Suppose not. Thus, there is a $d\in\Lambda_{G^{P_1}}(\bullet, \psi(q_1), \psi(p_1))$ with $|L(w^*)\setminus\{\psi(z), \psi(q_1), d\}|>2$. As $v_0x_1\not\in E(G)$ and $\psi(z)\in A$, we get that $\psi$ extends to an $L$-coloring $\pi$ of $V(G^{P_0}\cup G^{P_1}\cup P)$ using $d$ on $x_1$. Since $\pi$ does not extend to $L$-color $G$, and $|V(G^{v_0wx_1})|$ is even, it follows from \ref{PropCor3} i) of Proposition \ref{CorMainEitherBWheelAtM1ColCor} applied to $G^{v_0wx_1}$ that $|L_{\pi}(w)|<2$, contradicting our choice of $d$. \end{claimproof}\end{addmargin}

\begin{center}\begin{tikzpicture}

\node[shape=circle,draw=black] (p0) at (-4,0) {$p_0$};
\node[shape=circle,draw=black] (u1) at (-3, 0) {$u_1$};
\node[shape=circle,draw=white] (u1+) at (-2, 0) {$\ldots$};
\node[shape=circle,draw=black] (v0) at (-1, 0) {$v_0$};
\node[shape=circle,draw=white] (mid) at (0, 0) {$\ldots$};
\node[shape=circle,draw=black] (v1) at (1, 0) {$x_1$};
\node[shape=circle,draw=white] (un+) at (2, 0) {$\ldots$};
\node[shape=circle,draw=black] (ut) at (3, 0) {$u_t$};
\node[shape=circle,draw=black] (p1) at (4, 0) {$p_1$};
\node[shape=circle,draw=black] (q0) at (-3,2) {$q_0$};
\node[shape=circle,draw=black] (q1) at (3,2) {$q_1$};
\node[shape=circle,draw=black] (z) at (0,4) {$z$};
\node[shape=circle,draw=white] (K0) at (-2.8, 1) {$G^{P_0}$};
\node[shape=circle,draw=white] (K1) at (2.8, 1) {$G^{P_1}$};
\node[shape=circle,draw=black] (w) at (0, 1.5) {$w$};
\node[shape=circle,draw=white] (H*) at (0, 0.4) {$G^{v_0wx_1}$};
 \draw[-] (p1) to (ut);
 \draw[-] (v0) to (q0) to (z) to (q1) to (v1);
\draw[-] (q1) to (p1);
\draw[-] (p0) to (u1) to (u1+) to (v0) to (mid) to (v1) to (un+) to (ut);
\draw[-] (p0) to (q0);
\draw[-] (v0) to (z) to (w) to (q1);
\draw[-, line width=1.8pt] (v0) to (w) to (v1);

\end{tikzpicture}\captionof{figure}{}\label{V0SecCase0X1Figure}\end{center}

By Subclaim \ref{AnyLZp1ExtZQ1P1BlSub}, we have $A\subseteq L(w)$ and $|L(w)|=5$, and furthermore, $A\cap L(x_1)=\varnothing$. As $|A|\geq 3$, there is a $d\in L(x_1)$ with $d\not\in L(w)$. Again by Subclaim \ref{AnyLZp1ExtZQ1P1BlSub}, there is a set $\mathcal{S}$ of three different $L$-colorings of $q_1p_1$, each using a different color on $p_1$, where no element of $\mathcal{S}$ uses $d$ on $q_1$ and, for each $\pi\in\mathcal{S}$, we have $d\not\in\Lambda_{G^{P_1}}(\bullet, \pi(q_1), \pi(p_1))$. Thus, $G^{P_1}$ is not a triangle, so $x_1p_1\not\in E(G)$, and, for each $\pi\in\mathcal{S}$, $(d, \pi(q_1), \pi(p_1))$ is a proper $L$-coloring of $x_1q_1p_1$ which does not extend to $L$-color $G^{P_1}$, contradicting \ref{PropCor2} of Proposition \ref{CorMainEitherBWheelAtM1ColCor}. This proves Claim \ref{ifZIncAtleastthenV0X0V1X1}. \end{claimproof}

Applying Claims \ref{zPreciselyOneMPiQiNotSize3} and \ref{ifZIncAtleastthenV0X0V1X1}, we now prove the main result of Subsection \ref{CInduCycleSubS}. 

\begin{claim}\label{CycCIndMTh} Every chord of $C$ has precisely one endpoint in $\{q_0, q_1\}$. \end{claim}

\begin{claimproof} Suppose not. Thus, there is at least one chord with $z$ as an endpoint. By Claim \ref{ifZIncAtleastthenV0X0V1X1}, $v_0=x_0$ and $v_1=x_1$. As $P$ is induced, neither $G^{P_0}$ nor $G^{P_1}$ is not an edge. Since $q_0, q_1$ have no common neighbor in $C\setminus\mathring{P}$ by our assumption on $G$, it follows that $v_0\neq v_1$, so $M$ is not just an edge, and we have the structure illustrated in Figure \ref{V0V1=X0X1Figure}. As in Claim \ref{ifZIncAtleastthenV0X0V1X1}, we let $A$ be the set of $c\in L(z)$ such that there is a $(p_0q_0z, G^{Q_0})$-sufficient $L$-coloring of $\{p_0, z\}$ using $c$ on $z$, and we recall that $|A|\geq 3$. Note that, if $G^M$ is not a triangle, then $G^{Q_0}\cup G^{Q_1}$ is induced in $G$. 

\vspace*{-8mm}
\begin{addmargin}[2em]{0em}
\begin{subclaim}\label{ForAnyainAandbinBSubCLphi} If $G^M$ is not a triangle then,
\begin{enumerate}[label=\arabic*)]
\itemsep-0.1em
\item For any $a\in A$, there is an $L$-coloring $\psi$ of $V(G^{Q_0}\cup G^{Q_1})$, where $\psi(z)=a$ and the $L$-coloring $(\psi(v_0), \psi(z), \psi(v_1))$ of $v_0zv_1$ does not extends to $L$-color $G^M$; AND
\item For any $b\in L(p_1)$, there is an $L$-coloring $\psi'$ of $V(G^{Q_0}\cup G^{Q_1})$, where $\psi'(p_1)=b$ and the $L$-coloring $(\psi'(v_0), \psi'(z), \psi'(v_1))$ of $v_0zv_1$ does not extend to $L$-color $G^M$.
\end{enumerate}
\end{subclaim}

\begin{claimproof} By definition, there is a $\pi\in\textnormal{End}(p_0q_0z, G^{Q_0})$ using $a$ on $z$. Since $|L(p_1)|=3$, there is a $\pi^*\in\textnormal{End}(zq_1p_1, G^{Q_1})$ using $a$ on $z$, so $\pi\cup\pi^*$ is a proper $L$-coloring of $\{p_0, z, p_1\}$. By 1) of Claim \ref{PInducedNoPGSuff1}, $\pi\cup\pi^*$ extends to an element of $\mathcal{X}(P)$, so it follows that there is a $\psi$ as in 1). Likewise, since $|A|\geq 3$, it follows from Theorem \ref{SumTo4For2PathColorEnds} that, for each $b\in L(p_1)$, there is a $\sigma\in\textnormal{End}(zq_1p_1, G^{Q_1})$ with $\sigma(p_1)=b$ and $\sigma(z)\in A$, so the same argument shows that there is a $\psi'$ as in 2). \end{claimproof}
\end{addmargin}

\begin{center}\begin{tikzpicture}
\node[shape=circle,draw=black] (p0) at (-4,0) {$p_0$};
\node[shape=circle,draw=black] (u1) at (-3, 0) {$u_1$};
\node[shape=circle,draw=white] (u1+) at (-2, 0) {$\ldots$};
\node[shape=circle,draw=black] (v0) at (-1, 0) {$v_0$};
\node[shape=circle,draw=white] (mid) at (0, 0) {$\ldots$};
\node[shape=circle,draw=black] (v1) at (1, 0) {$v_1$};
\node[shape=circle,draw=white] (un+) at (2, 0) {$\ldots$};
\node[shape=circle,draw=black] (ut) at (3, 0) {$u_t$};
\node[shape=circle,draw=black] (p1) at (4, 0) {$p_1$};
\node[shape=circle,draw=black] (q0) at (-3,2) {$q_0$};
\node[shape=circle,draw=black] (q1) at (3,2) {$q_1$};
\node[shape=circle,draw=black] (z) at (0,4) {$z$};
\node[shape=circle,draw=white] (K0) at (-2.8, 1) {$G^{P_0}$};
\node[shape=circle,draw=white] (K1) at (2.8, 1) {$G^{P_1}$};
\node[shape=circle,draw=white] (Hdown) at (0, 1.6) {$G^M$};
 \draw[-] (p1) to (ut);
 \draw[-] (v0) to (q0) to (z) to (q1) to (v1);
\draw[-] (q1) to (p1);
\draw[-] (p0) to (u1) to (u1+) to (v0) to (mid) to (v1) to (un+) to (ut);
\draw[-] (p0) to (q0);
\draw[-] (v0) to (z) to (v1);

\end{tikzpicture}\captionof{figure}{}\label{V0V1=X0X1Figure}\end{center}

\vspace*{-8mm}
\begin{addmargin}[2em]{0em} 
\begin{subclaim}\label{GMIsTrianglSubH} $G^M$ is a triangle, and $|L(v_0)\cap L(v_1)|>1$.  \end{subclaim}

\begin{claimproof} We first show that $G^M$ is a broken wheel with principal path $v_0zv_1$. Suppose not. As $|L(z)\setminus L(v_0)|\geq 2$, it follows from Theorem \ref{EitherBWheelOrAtMostOneColThm} applied to $G^M$ that there is a $c\in L(z)\setminus L(v_0)$ such that any $L$-coloring of $v_0zv_1$ using $c$ on $z$ extends to $L$-color $G^M$. On the other hand, $L(z)\setminus L(v_0)\subseteq A$, so we contradict Subclaim \ref{ForAnyainAandbinBSubCLphi}. Thus, $G^M$ is a indeed a broken wheel with principal path $v_0zv_1$. Let $u^*$ be the unique neighbor of $v_0$ on the path $G^M-z$.  Now suppose $G^M$ is not a triangle and consider the following cases.

\textbf{Case 1:} $L(z)\not\subseteq L(v_0)\cup L(u_*)$

In this case, by 1) of Claim \ref{PInducedNoPGSuff1}, there is a  $\phi\in\mathcal{X}(P)$, where $\phi(z)\not\in L(v_0)\cup L(u^*)$. As $v_0v_1\not\in E(G)$ and $\phi(z)\not\in L(v_0)$, we get that $\phi$ extends to an $L$-coloring $\phi'$ of $V(G^{P_0}\cup G^{P_1}\cup P)$. But since $\phi'(z)\not\in L(u^*)$, it follows that $\phi'$ extends to $L$-color $G$, a contradiction. 

\textbf{Case 2:} $L(z)\subseteq L(v_0)\cup L(u_*)$

In this case, $|L(v_0)\cap L(u_*)|\leq 1$. By Subclaim \ref{ForAnyainAandbinBSubCLphi}, since $|A|\geq 3$, there is a set $\mathcal{S}$ of three different $L$-colorings of $v_0zv_1$, each using a different color on $z$, none of which extend to $L$-color $G^M$. Thus, there is a color $c$ such that $L(v_0)\cap L(u_*)=\{c\}$, and each element of $\mathcal{S}$ uses $c$ on $v_0$. But at least one element of $\mathcal{S}$ uses a color on $z$ which does not lie in $L(u_*)\setminus\{c\}$, so at least one element of $\mathcal{S}$ extends to $L$-color $G^M$, which is false.

We conclude that $G^M$ is a triangle. Now we show that $|L(v_0)\cap L(v_1)|>1$. Suppose not. Thus, there is an $X\subseteq L(v_1)\setminus L(v_0)$ with $|X|=2$. Since $|A|\geq 3$, there is a $c\in A\setminus X$, so there is a set $\mathcal{T}$ of three different $L$-colorings of $V(G^{Q_0}\cup P)$, each of which uses $c$ on $z$ and a different color on $p_1$, where, for each $\pi\in\mathcal{T}$, we have $\Lambda_{G^{P_1}}(\bullet, \pi(q_1), \pi(p_1))=\{\pi(v_0)\}$. Thus, there is a color $d$ with $L(v_0)\cap L(v_1)=\{d\}$, where $\Lambda_{G^{P_1}}(\bullet, \pi(q_1), \pi(p_1))=\{d\}$ for each $\pi\in\mathcal{T}$. It follows that $G^{P_1}$ is a broken wheel with principal path $P_1$ and, since $|L(p_1)|=3$, we contradict \ref{BWheel1Lb} \ref{BWheel1C} of Theorem \ref{BWheelMainRevListThm2}. This proves Subclaim \ref{GMIsTrianglSubH}. \end{claimproof}\end{addmargin}

\vspace*{-8mm}
\begin{addmargin}[2em]{0em} 
\begin{subclaim}\label{Lp1Lv1SubGP1Tr} No color of $L(p_1)$ is almost $(P_1, G^{P_1})$-universal. \end{subclaim}

\begin{claimproof} Suppose there is such a $b\in L(p_1)$. By Subclaim \ref{GMIsTrianglSubH}, there is a $c\in L(z)\setminus (L(v_0)\cup L(v_1))$. As $c\not\in L(v_0)$, it follows from 1) of Claim \ref{PInducedNoPGSuff1}, together with Theorem \ref{thomassen5ChooseThm}, that there is an $L$-coloring $\pi$ of $V(P\cup G^{P_0})$ which does not extend to $L$-color $G$, where $\pi(z)=c$ and $\pi(p_1)=b$. On the other hand, $c\not\in L(v_1)$, and, since $G^M$ is a triangle and $|\Lambda_{G^{P_1}}(\bullet, \pi(q_1), \pi(p_1))|\geq 2$, it follows that $\pi$ extends to $L$-color $G$, a contradiction. \end{claimproof}\end{addmargin}

Now, it follows from Claim \ref{IfGP1NotEdgeThen} that $G^{P_1}$ is a triangle with $L(p_1)=L(v_1)$. We now fix a $\tau\in\textnormal{End}(P_0, G^{p_0})$ and, since $|L_{\tau}(z)|\geq 4$, we fix a $d\in L_{\tau}(z)$ with $|L_{\tau}(q_0)\setminus\{d\}|\geq 3$. We now show that $d\in L(v_1)$. Suppose $d\not\in L(v_1)$. Since $|L_{\tau}(v_1)|\geq 2$, there is a $d'\in L(p_1)$ with $|L_{\tau}(v_1)\setminus\{d'\}|\geq 2$. By 2) of Claim \ref{PInducedNoPGSuff1}, $\tau$ extends to a $\tau'\in\mathcal{X}(G^{P_0}\cup P)$, where $\tau'(z)=d$ and $\tau'(p_1)=d'$. Since $G^{P_1}$ and $G^M$ are both triangles and $d\not\in L(v_1)$, $\tau'$ extends to $L$-color $G$,  a contradiction. Thus, we indeed have $d\in L(v_1)$. Since $L(v_1)=L(p_1)$, we have $d\in L(p_1)$. Since $|L_{\tau}(v_1)\setminus\{d\}|\geq 1$, it follows that $\tau$ extends to an $L$-coloring $\tau^*$ of $\{p_0, v_0, v_1, p_1, z\}$ using $d$ on both of $z, p_1$, so each of $q_0, q_1$ has an $L_{\tau^*}$-list of size at least three. Since each of $G^{P_1}$ and $G^{M_1}$ is a triangle, it follows from our choice of $\tau$ that $\tau*$ is $(P,G)$-sufficient, contradicting our assumption that $G$ is a counterexample. This proves Claim \ref{CycCIndMTh}. \end{claimproof}

\subsection{Dealing with common neighbors of $z, x_i$}\label{CommNZXISubSFin}

In this subsection, we show that, for each $i\in\{0,1\}$ and any $w\in N(q_i)\cap N(z)$, we have $N(w)\cap V(C)\subseteq\{q_i, z, x_i\}$. We first prove the following intermediate result. 

\begin{claim}\label{SwitchBaseColorUseTerm} For any $i\in\{0,1\}$ and $w\in N(q_i)\cap N(z)$, both of the following hold:
\begin{enumerate}[label=\Alph*)]
\item\label{SwitchA} Either $N(w)\cap V(C)\subseteq\{z, q_i, x_i\}$ or $q_{1-i}\in N(w)$; AND
\item\label{SwitchB} $N(w)\cap V(C)\not\subseteq\{z, q_0, q_1, p_i\}$ and $w$ is adjacent to at most one of $p_0, p_1$. 
\end{enumerate}
 \end{claim}

\begin{claimproof} We first show \ref{SwitchA}. Suppose $q_{1-i}\not\in N(w)$. By assumption, $w$ has a neighbor other than $x_i$ on the path $x_{1-i}(C\setminus\mathring{P})x_i$. Among all such neighbors of $w$, let $v$ be the one which is closest to $x_i$ on this path. If $v=p_{1-i}$, then, since $w\not\in N(q_{1-i})$, it follows from our triangulation conditions that $z\in N(p_{1-i})$, contradicting Claim \ref{CycCIndMTh}. Thus, $v\neq p_{1-i}$, so let $R:=p_{1-i}q_{1-i}zq_iv$ and $Q:=vwq_ip_i$. Note that $v\neq p_i$, or else $v=p_i=x_i$, contradicting our choice of $v$. In particular, $|L(v)|=3$. Now, by minimality, there is a set of $|L(p_{1-i})$ elements of $\textnormal{Crown}(R, G^R)$, each using a different color on $v$. Furthermore, it follows from \ref{LabCrownNonEmpt4} of Theorem \ref{CombinedT1T4ThreePathFactListThm} that there is a set of $|L(p_i)|$ different $\mathcal{G}^Q$-base colorings of $\{v, p_i\}$, each using a different color on $v$, so there is a $\phi\in\textnormal{Crown}(R, G^R)$ and a $\mathcal{G}^Q$-base-coloring $\psi$ of $\{v, q_i\}$ using the same color on $v$. Let $\mathcal{F}$ be the set of extensions of $\psi$ to $L$-colorings of $V(Q)$ which do not extend to $L$-color $G^Q$, and consider $\psi\cup\phi$. Each of $q_0, q_1$ has an $L_{\psi\cup\phi}$-list of size at least three, so, since $G$ is a counterexample, $\psi\cup\phi$ extends to an $L$-coloring $\tau$ of $\textnormal{dom}(\psi\cup\phi)\cup\{q_0, q_1\}$ which does not extend to $L$-color $G$. By Claim \ref{IfGP1NotEdgeThen}, $p_i\not\in N(v)$ , so, since $|L_{\phi}(w)|\geq 3$, we have $|L_{\phi\cup\psi}(w)|\geq 3$. As $wq_{1-i}\not\in E(G)$, we have $|L_{\tau}(w)|\geq 2$. But since $\phi$ is a $\mathcal{G}^Q$-base coloring, we have $|\mathcal{F}|=2$ and each element of $\mathcal{F}$ uses $\tau(q_i)$ on $q_i$, so, by Definition \ref{BaseColoringDefn}, $G^{vwx_i}$ is a broken wheel, where $G^{vwx_i}-w$ has even length, contradicting our choice of $v$. This proves \ref{SwitchA}.

Now we prove \ref{SwitchB}. Suppose $N(w)\cap V(C)\subseteq\{z, q_0, q_1, p_i\}$. Thus, $x_i=p_i$ and $N(q_i)=\{z, p_i\}$, and, letting  $P':=(P-q_i)+zwp_i$, it follows that, for any $\phi\in\textnormal{Crown}(P', G-q_i)$, each of $q_0, q_1, w$ has an $L_{\phi}$-list of size at least three, so it is straightforward to check that $\textnormal{Crown}(P', G-q_i)\subseteq\textnormal{Crown}(P, G)$. Thus, there is no such $\phi$, as $\textnormal{Crown}(P,G)=\varnothing$. Now, by minimality, it follows that $w, q_{1-i}$ have a common neighbor in $C^{P'}\setminus P'=C\setminus P$, contradicting our assumption that $N(w)\cap V(C)\subseteq\{z, q_0, q_1, p_i\}$. To finish, we show that $w$ is adjacent to at most one of $p_0, p_1$. Suppose not. By A), $w$ is adjacent to all the vertices of $P$. Let $G':=G\setminus\mathring{P}$. Then $G'$ has $p_0wp_1$ on its outer cycle and, by Theorem \ref{SumTo4For2PathColorEnds}, there is a $\phi\in\textnormal{End}(p_0wp_1, G')$. Now, there is a $d\in L(z)$ with $|L_{\phi}(z)\setminus\{d\}|\geq 3$, and $\phi$ extends to a $\phi'\in\mathcal{X}(P)$ with $\phi'(z)=d$, which is false, as $|L_{\phi'}(w)|\geq 1$. \end{claimproof}

We now introduce the following definition, which we retain for the remainder of the proof of Theorem \ref{MainHolepunchPaperResulThm}. We let $P_{\pentagon}$ denote the path $x_0q_0zq_1p_1$. As $z$ is incident to no chords of $C$ and $P$ is an induced path, $C^{P_{\pentagon}}$ is induced. 

\begin{claim}\label{IfwAdj10zq1ThenAdjCAllInt} All three of the following hold.
\begin{enumerate}[label=\arabic*)]
\itemsep-0.1em
\item\label{MainCLFac} For each $i\in\{0,1\}$ and any $w\in N(q_i)\cap N(z)$, we have $N(w)\cap V(C)\subseteq\{q_i, z, x_i\}$; AND
\item\label{SeconLabFact} $d(x_0, x_1)>1$, and, in particular, any $L$-coloring of $V(P)$ extends to $L$-color all of $V(P\cup G^{P_0}\cup G^{P_1})$; AND
\item\label{ThirdLabFact}  For each $i\in\{0,1\}$, $N(z)\cap N(q_i)\cap N(p_i)=\varnothing$.
\end{enumerate}
 \end{claim}

\begin{claimproof} We first show that 1) implies both 2) and 3), then we prove 1). Note that 1), together with \ref{SwitchB} of Claim \ref{SwitchBaseColorUseTerm}, immediately implies 3), so we just show that 1) implies 2). Suppose that 1) holds. Suppose first that $d(x_0, x_1)=1$. Thus, $C^{P_{\pentagon}}$ is an induced 5-cycle, and, since 1) holds, it follows from Theorem \ref{BohmePaper5CycleCorList} that any $L$-coloring of $V(C^{P_{\pentagon}})$ extends to $L$-color $G^{P_{\pentagon}}$. As $|L(z)|=5$, there is an $L$-coloring $\phi$ of $\{p_0, x_0, z\}$ with $|L_{\phi}(q_0)|\geq 3$, where the restriction of $\phi$ to $\{p_0, x_0\}$ is $(P^0, G^{P_0})$-sufficient. By 2) of Claim \ref{PInducedNoPGSuff1}, $\phi$ extends to a set of three $L$-colorings $\{\sigma^0, \sigma^1, \sigma^2\}$ of $V(P\cup G^{P_0})$, all using different colors on $p_1$, where none of $\sigma^0, \sigma^1, \sigma^2$ extends to $L$-color $G$. Letting $b:=\phi(x_0)$, at least one $\sigma^k$ uses a color other than $b$ on $p_1$, so $G^{P_1}$ is not an edge, and, in particular, $\Lambda_{G^{P_1}}(\bullet, \sigma^k(q_1), \sigma^k(p_1))=\{b\}$ for each $k\in\{0,1,2\}$, so $G^{P_1}$ is a broken wheel with principal path $P_1$ and we contradict \ref{BWheel1Lb} \ref{BWheel1C} of Theorem \ref{BWheelMainRevListThm2}. We conclude that 1) implies 2) as well. 

Now we prove 1). Suppose toward a contradiction that 1) does not hold. As $x_0\neq x_1$ and $P$ is an induced path, it follows from Claim \ref{SwitchBaseColorUseTerm} that there is a $w\in V(G\setminus C)$ which is adjacent to $q_0, z, q_1$, where $w$ has at least one neighbor in $C\setminus\mathring{P}$. In particular, $N(z)=\{q_0, w, q_1\}$. For each $k=0,1$, let $v_k$ be the unique neighbor of $w$ on the path $x_0(C\setminus\mathring{P})x_1$ which is closest to $x_k$. Let $R_k:=p_{1-k}q_{1-k}wv_k$ and $Q_k:=v_kwq_kp_k$, so that $G^{R_k}\cup G^{Q_k}=G-w$. By \ref{SwitchB} of Claim \ref{SwitchBaseColorUseTerm}, for each $k=0,1$, we have $v_k\neq p_{1-k}$, so the rainbow $\mathcal{G}^{R_k}$ is well-defined. 

\vspace*{-8mm}
\begin{addmargin}[2em]{0em}
\begin{subclaim}\label{ForEachKGluePiP1P0} For each $k\in\{0,1\}$, there is an $L$-coloring $\pi$ of $\{p_{1-k}, v_k, x_k, p_k\}$ whose restriction to $\{p_{1-k}, v_k\}$ is a $\mathcal{G}^{R_k}$-base coloring and whose restriction to $\{v_k, x_k, p_k\}$ is $(Q_k, G^{Q_k})$-sufficient. \end{subclaim}

\begin{claimproof} By \ref{LabCrownNonEmpt4} of Theorem \ref{CombinedT1T4ThreePathFactListThm} that there is a family of $|L(p_{1-k})|$ different $\mathcal{G}^{R_k}$-base colorings, each using a different color on $v_k$. If $Q_k$ is a triangle, then $v_k=x_k=p_k$ and any $L$-coloring of  $\{p_k\}$ is $(Q_k, G^{Q_k})$-sufficient, and then the subclaim follows immediately, so suppose that $Q_k$ is not a triangle, i.e $v_k\neq p_k$. Thus, $Q_k$ is a 3-path and $|L(v_k)|=3$. and, by \ref{LabCrownNonEmpt} of Theorem \ref{CombinedT1T4ThreePathFactListThm} applied to $\mathcal{G}^{Q_k}$, there are $|L(p_k)|$ different $(Q_k, G^{Q_k})$-sufficient $L$-colorings of $\{v_k, x_k p_k\}$, each using a different color on $v_k$, so the subclaim follows. \end{claimproof}\end{addmargin}

\vspace*{-8mm}
\begin{addmargin}[2em]{0em}
\begin{subclaim}\label{EachRKEvenTilted} For each $k\in\{0,1\}$, the following hold:
\begin{enumerate}[label=\alph*)]
\itemsep-0.1em
\item $v_{1-k}=x_{1-k}$ and at least one terminal edge of $R_k$ is even-tilted in $\mathcal{G}^{R_k}$.
\item Either $|V(G^{P_k})|>1$ or $wv_k$ is even-tilted in $\mathcal{G}^{R_k}$
\end{enumerate}
\end{subclaim}

\begin{claimproof} Let $\pi$ be an $L$-coloring of $\{p_{1-k}, v_k, x_k, p_k\}$ as in Subclaim \ref{ForEachKGluePiP1P0} and $\mathcal{F}$ be the set of $L$-colorings of $V(R_k)$ which use $\pi(p_{1-k}), \pi(v_k)$ on $p_{1-k}, v_k$ respectively and do not extend to $L$-color $G^{R_k}$. We first show that $k$ satisfies a). Suppose not. Since a) is violated, no chord of $R_k$ incident to a vertex of $\mathring{R}_k$.  We have $|L_{\pi}(q_k)|\geq 3$, and, as $q_{1-k}\not\in N(v_k)$, we have $|L_{\pi}(q_{1-k})|\geq 4$. Let $S:=\{c\in L(z): |L_{\pi}(q_k)\setminus\{c\}|\geq 3\}$. Thus, $|S|\geq 2$ and, for each $c\in S$, there is an extension of $\pi$ to an $L$-coloring $\pi^c$ of $V(P)\cup\{v_k, x_k\}$, where $\pi^c$ does not extend to $L$-color $G$ and $\pi^c(z)=c$. . As a) is violated, it follows from Definition \ref{BaseColoringDefn} that $\{\psi(q_{1-k}), \psi(w)\}$ is constant as $\psi$ runs over $\mathcal{F}$, where possibly $|\mathcal{F}|\leq 1$, so there is a set $T$ of size two, where each element of $\mathcal{F}$ uses the colors of $T$ on $q_{1-k}w$. Since $w\not\in N(p_{1-k})$, it follows that, for each $c\in S$, we have $\pi^c(q_{1-k})\in T$ and $L(w)\setminus\{\pi(v_k), c, \pi^c(q_k)\}=T$, so we get $|S|=2$ and furthermore, $|L_{\pi}(q_k)|=3$ and $S\subseteq L_{\pi}(q_k)$, contradicting the definition of $S$. Thus, $k$ satisfies a).

Now we show that $k$ satisfies b). Suppose not. Thus, $R_k:=p_{1-k}q_{1-k}wp_k$ and $\textnormal{dom}(\pi)=\{p_0, p_1\}$, but $H$ is not even-tilted in $\mathcal{G}^{R_k}$. By Claim \ref{IfwAdj10zq1ThenAdjCAllInt}, $p_{1-k}\not\in N(w)$, so there is a $d\in L(z)$ with $|L_{\pi}(w)\setminus\{d\}|\geq 4$. As $G$ is a counterexample, $\pi$ extends to a $\pi'\in\mathcal{X}(P)$ with $\pi'(z)=d$. But then, $|L_{\pi'}(w)|\geq 2$, so $|\mathcal{F}|=2$ and $\mathcal{F}$ is constant on $q_{1-k}$, so only \ref{T4PartB} of Definition \ref{BaseColoringDefn} is satisfied and $H$ is even-tilted in $\mathcal{G}^{R_k}$, a contradiction.  \end{claimproof}\end{addmargin}

Let $H:=G^{x_0wx_1}$. As $w$ is adjacent to each of $x_0, x_1$, we have the situation illustrated in Figure \ref{WMidAjdBothX0X1}. Possibly one ore both of $G^{P_0}, G^{P_1}$ is an edge. 

\begin{center}\begin{tikzpicture}

\node[shape=circle,draw=black] (p0) at (-4,0) {$p_0$};
\node[shape=circle,draw=black] (u1) at (-3, 0) {$u_1$};
\node[shape=circle,draw=white] (u1+) at (-2, 0) {$\ldots$};
\node[shape=circle,draw=black] (v0) at (-1, 0) {$x_0$};
\node[shape=circle,draw=white] (mid) at (0, 0) {$\ldots$};
\node[shape=circle,draw=black] (v1) at (1, 0) {$x_1$};
\node[shape=circle,draw=white] (un+) at (2, 0) {$\ldots$};
\node[shape=circle,draw=black] (ut) at (3, 0) {$u_t$};
\node[shape=circle,draw=black] (p1) at (4, 0) {$p_1$};
\node[shape=circle,draw=black] (q0) at (-3,2) {$q_0$};
\node[shape=circle,draw=black] (q1) at (3,2) {$q_1$};
\node[shape=circle,draw=black] (z) at (0,4) {$z$};
\node[shape=circle,draw=white] (K0) at (-2.8, 1) {$G^{P_0}$};
\node[shape=circle,draw=white] (K1) at (2.8, 1) {$G^{P_1}$};
\node[shape=circle,draw=black] (w) at (0, 1.5) {$w$};
\node[shape=circle,draw=white] (H) at (0, 0.5) {$H$};
 \draw[-] (p1) to (ut);
 \draw[-] (v0) to (q0) to (z) to (q1) to (v1);
\draw[-] (q1) to (p1);
\draw[-] (p0) to (u1) to (u1+) to (v0) to (mid) to (v1) to (un+) to (ut);
\draw[-] (p0) to (q0);
\draw[-] (v0) to (q0) to (w) to (q1);
\draw[-] (w) to (z);
\draw[-, line width=1.8pt] (v0) to (w) to (v1);

\end{tikzpicture}\captionof{figure}{}\label{WMidAjdBothX0X1}\end{center}

\vspace*{-8mm}
\begin{addmargin}[2em]{0em}
\begin{subclaim}\label{ReGP1Case}  For each $a\in L(p_1)$, there is a $\sigma^a\in\mathcal{X}(G^{P_0}\cup P)$ with $|L(w)\setminus\{\sigma^a(x_0), \sigma^(z)\}|\geq 2$, and, for each $b\in L(x_1)$, there is a $\tau^b\in\mathcal{X}(P\cup G^{P_1})$ with $b\in\{\tau^b(x_1), \tau^b(q_1)\}$ and $|L(w)\setminus\{\tau^b(x_1), \tau^b(z)\}|\geq 4$. \end{subclaim}

\begin{claimproof} Let $a\in L(p_1)$ and $b\in L(x_1)$. Consider the following cases.

\textbf{Case 1:} $L(z)=L(w)$

Since $|L(p_1)|=|L(x_1)|=3$, there is a $(P_1, G_1)$-sufficient $L$-coloring $\phi$ of $\{x_1, p_1\}$ with $\phi(x_1)=b$. Likewise, since since $|V(H)|$ is even, $d(x_0, x_1)>1$, so there is an $L$-coloring $\psi$ of $\{p_0, x_0, p_1\}$ with $\psi(p_1)=a$  and whose restriction to $\{p_0, x_0\}$ is $(P_0, G_0)$-sufficient. Let $A:=\{c\in L(z): |L_{\psi}(q_0)\setminus\{c\}|\geq 3\}$ and $B:=\{c\in L(z): |L_{\phi}(q_1)\setminus\{c\}|\geq 3\}$. Note that $|A|\geq 2$ and $|B|\geq 2$. If $\psi(x_0)\in L(w)$, then, since $L(w)=L(z)$, it also lies in $A$, and then $\phi'$ extends to a $\sigma^a$ satisfying the above by 2) of Claim \ref{PInducedNoPGSuff1}. If $\psi(x_0)\not\in L(w)$, then we just choose an arbitrary color of $A$, and again, $\psi$ extends to a $\sigma^a$ as above. The same argument applied to $\phi$ shows that $\phi$ extends to a $\tau^b\in\mathcal{X}(P\cup G^{P_1})$ with $|L(w)\setminus\{\tau^b(x_1), \tau^b(z)\}|\geq 2$, and since $\tau^b(x_1)=b$, we are done. 

\textbf{Case 2:} $L(z)\neq L(w)$

In this case, there is a $d\in L(z)\setminus L(w)$. We construct $\sigma^a$ by choosing a $\pi\in\mathcal{X}(P)$ using $a,d$ on $p_1, z$ respectively, and, as $d(x_0, x_1)>1$, we can extend $\pi$ to an element of $\mathcal{X}(P\cup G^{P_0})$. Now we construct a $\tau^b$ satisfying the subclaim. Suppose there is no such $\tau^b$. As $G$ is a counterexample, there is a family $\mathcal{S}\subseteq\mathcal{X}(P)$, where each element of $\mathcal{S}$ uses $d$ on $z$ and a different color on $p_1$, so it follows that, for each $\pi\in\mathcal{S}$, we have $\pi(q_1)\neq b$ and $b\not\in\Lambda_{G^{P_1}}(\bullet, \pi(q_1), \pi(p_1))$. As $|L(p_1)|=3$, $G^{P_1}$ is not a triangle and we contradict \ref{PropCor2} of Proposition \ref{CorMainEitherBWheelAtM1ColCor}. \end{claimproof}\end{addmargin}

For each $a\in L(p_1)$ and $b\in L(x_1)$, let $\sigma^a$ and $\tau^b$ be as in Subclaim \ref{ReGP1Case}. 

\vspace*{-8mm}
\begin{addmargin}[2em]{0em}
\begin{subclaim}\label{NoColX1x0x1Unuv4}  $H$ is a broken wheel with $|V(H)|=4$, and no color of $L(x_1)$ is $(x_0wx_1, H)$-universal. \end{subclaim}

\begin{claimproof} We first show that $H$ is a broken wheel with an even number of vertices. This is immediate from Subclaim \ref{EachRKEvenTilted} if either $G^{P_1}$ is an edge or  $p_1q_1$ is not even-tilted in $\mathcal{G}^{R_0}$, so we suppose that $G^{P_1}$ is  a broken wheel with an even number of vertices. By Claim \ref{IfGP1NotEdgeThen}, there is an $a\in L(p_1)$ which is almost $(P_1, G^{P_1})$-universal. Let $T':=\Lambda_{G^{P_1}}(\bullet, \sigma^a(q_1), a)$, so $|T'|\geq 2$. Since $\sigma^a$ does not extend to $L$-color $G$, there is no $L$-coloring of $H$ which uses $\sigma^a(x_0)$ on $x_0$, a color of $L_{\sigma^a}(w)$ on $w$, and a color of $T'$ on $x_1$. As $|L_{\sigma^a}(w)|\geq 2$, so it follows from \ref{PropCor2} of Proposition \ref{CorMainEitherBWheelAtM1ColCor} that $|V(H)|$ is even, as desired. Furthermore, if there is $b\in L(x_1)$ which is $(x_0wx_1, H)$-universal, then $\tau^b$ extends to $L$-color $G$, which is false, so no color of $L(x_1)$ which is $(x_0wx_1, H)$-universal. To finish, we show that $|V(H)|=4$. Suppose not. Since $|V(H)|$ is even, we have $|V(H)|\geq 6$. Let $yy'y''x_1$ be the unique subpath of $H-w$ with length three and $x_1$ as an endpoint. By \ref{PropCor1} of Proposition \ref{CorMainEitherBWheelAtM1ColCor}, $L(x_1)=L(y'')=L(y')$. Let $G^*:=(G\setminus\{y', y''\})+yx_1$, where $G^*$ has outer cycle $C^*:=(C\setminus\{y', y''\})+yx_1$. Since $H-w$ has length at least four, $q_0, q_1$ have no common neighbor in $C\setminus P$. Thus, by minimality, there is a $\phi\in\textnormal{Crown}_L(P, G^*)$. Since $L(x_1)=L(y')$, it follows from Observation \ref{MinCounterReUseObs} that any $L$-coloring of $G^*$ extends to $L$-color $G$, so $\phi\in\textnormal{Crown}_L(P,G)$ as well, which is false, as $G$ is a counterexample. \end{claimproof}\end{addmargin}

If $G^{P_0}$ is an edge, then, since $|L(p_0)|=1$ and $H$ is not a triangle, it follows that there is an $(x_0wx_1, H)$-universal color of $L(x_1)$, contradicting Subclaim \ref{NoColX1x0x1Unuv4}. Thus, $G^{P_0}$ is not an edge. Let $x^*$ be the lone vertex of $H\setminus\{x_0, w, x_1\}$. Subclaim \ref{NoColX1x0x1Unuv4} also implies that $L(x_1)=L(x^*)$ and there is no $c\in L(x_1)$ with $L(x_0)\cap (L(x^*)\setminus\{c\})=\varnothing$, so $|L(x_0)\cap L(x_1)|\geq 2$. Note that, for each $b\in L(x_0)\cap L(x_1)$, $\tau^b$ does extend to $L$-color $V(G^{P_0}\cup G^{P_1}\cup P)=V(G)\setminus\{w, x^*\}$, as $H$ is not a triangle. Thus, $\tau^b(q_0)\neq b$ and $b\not\in\Lambda_{G^{P_0}}(\tau^b(p_0), \tau^b(q_0), \bullet)$, or else it follows from the definition of $\tau^b$ that it extends to $L$-color $G$, which is false. As $|L(p_0)|=1$, it follows that $G^{P_0}$ is not a triangle and, in particular, for each $b\in L(x_0)cap L(x_1)$, $(\tau^b(p_0), \tau^b(q_0), b)$ is a proper $L$-coloring of $p_0q_0x_0$ which does not extend to $L$-color $G^{P_0}$. Thus, by \ref{PropCor2} of Proposition \ref{CorMainEitherBWheelAtM1ColCor}, $|L(x_0)\cap L(x_1)|=2$ and, letting $c$ be the lone color of $L(x_0)\setminus L(x_1)$, we have $\Lambda_{G^{P_0}}(\tau^b(p_0), \tau^b(q_0), \bullet)=\{c\}$ for each $b\in L(x_0)\cap L(x_1)$. But then, since $L(x_1)=L(x^*)$, we have $c\not\in L(x^*)$, so each $\tau^b$ extends to an $L$-coloring of $G$ using $c$ on $x_0$, which is false. This proves Claim \ref{IfwAdj10zq1ThenAdjCAllInt}. \end{claimproof}

\subsection{Ruling out a common neighbor of $q_0, z, q_1$}\label{CommNQ1ZQSubS}

To prove that there is no common neighbor of $q_0, z, q_1$, we first prove the following intermediate result.

\begin{claim}\label{Rule57OutK0K1EdgeMCase} If there is a $w\in V(G\setminus C)$ adjacent to all three of $q_0, z, q_1$, then neither $G^{P_0}$ nor $G^{P_1}$ is an edge.  \end{claim}

\begin{claimproof} Suppose there is such a $w$ but there is a $j\in\{0,1\}$ such that $G^{P_j}$ is  an edge. Thus, $G^{P_{\pentagon}}=G\setminus (G^{P_{1-j}}\setminus P_{1-j})$. Let $S$ be the set of vertices $x^*$ of $G^{P_{\pentagon}}\setminus C^{P_{\pentagon}}$ such that there is a 2-chord of $C^{P_{\pentagon}}$ with both endpoints in $V(C^{P_{\pentagon}})\setminus\{q_0, z, q_1\}$ which separates $x^*$ from $z$. Let $\mathcal{S}$ be the set of partial $L$-colorings $\phi$ of $(V(C^{P_{\pentagon}})\setminus\{q_0, z, q_1\})\setminus S$ such that $\textnormal{dom}(\phi)$ contains $p_j, x_{1-j}$ and any extension of $\phi$ to an $L$-coloring of $G\setminus S$ extends to $L$-color all of $G$. Now, it follows from our triangulation conditions that, for each edge $e$ of the path $C^{P_{\pentagon}}-\{q_0, z, q_1\}$, there is at least one 2-chord of $C^{P_{\pentagon}}$ whose endpoints lie in $C^{P_{\pentagon}}-\{q_0, z, q_1\}$, where this 2-chord separates $e$ from $z$ (possibly the endpoints of this 2-chord are simply the endpoints of $e$). Now, it follows from repeated application of Theorem \ref{EitherBWheelOrAtMostOneColThm} that there is a family of $|L(p_j)|$ different elements of $\mathcal{S}$, each using a different color on $x_{1-j}$. This is true even if $G^{P_{1-j}}$ is also an edge. In any case, one more application of Theorem \ref{EitherBWheelOrAtMostOneColThm} (in the case where $G^{P_{1-j}}$ is not an edge) shows that there is a $\phi\in\mathcal{S}$ and a $(P_{1-j}, G^{P_{1-j}})$-sufficient $L$-coloring $\psi$ of $\{x_{1-j}, p_{1-j}\}$, where $\phi(x_{1-j})=\psi(x_{1-j})$. 

\vspace*{-8mm}
\begin{addmargin}[2em]{0em}
\begin{subclaim}\label{SubCClCfWrGB4} There is a $c\in L_{\phi\cup\psi}(z)$ such that $|L_{\phi\cup\psi}(q_{1-j})\setminus\{c\}|\geq 3$, where any extension of $\phi$ to an $L$-coloring of $\textnormal{dom}(\phi)\cup\{q_0, z, q_1\}$ using $c$ on $z$ extends to $L$-color all of $H$.\end{subclaim}

\begin{claimproof} Let $H^*:=G^{P_{\pentagon}}\setminus (S\cup\textnormal{dom}(\phi))$ and let $D^*$ be the outer cycle of $H^*$. Now, by \ref{MainCLFac} of Claim \ref{IfwAdj10zq1ThenAdjCAllInt}, $w$ does not lie on the outer cycle of $H^*$, i.e $w\in V(H^*\setminus D^*)$. Note that $S\cup (\textnormal{dom}(\phi)\cap V(C))$ consists of precisely the path $p_j(C\setminus\mathring{P})x_{1-j}$. Furthermore, since $C^{P_{\pentagon}}$ is induced, it follows from the definition of $S$ that each vertex of $G^{P_{\pentagon}}\setminus C^{P_{\pentagon}}$ has  at most two neighbors in $S$. Thus, each vertex of $D^*\setminus\{q_0, z, q_1\}$ has an $L_{\phi}$-list of size at least three. Now, since $|L_{\phi\cup\psi}(z)|\geq 3$ and no chords of $C$ are incident to $z$, there exist distinct $c, d\in L_{\phi\cup\psi}(z)$ where $|L_{\phi\cup\psi}(q_1)\setminus\{c\}|\geq 3$ and $|L_{\phi\cup\psi}(q_1)\setminus\{d\}|\geq 3$. Suppose now that Subclaim \ref{SubCClCfWrGB4} does not hold. By definition of $\mathcal{S}$, any $L_{\phi}$-coloring of $H^*$ extends to $L$-color all of $G^{P_{\pentagon}}$, so there exist two $L_{\phi}$-colorings of the path $q_0zq_1$ which use different colors on $z$, where neither of these two $L$-colorings extends to $L_{\phi}$-color $H^*$. Since $q_0q_1\not\in E(G)$ and $N(z)=\{q_0, w, q_1\}$, it follows from Theorem \ref{EitherBWheelOrAtMostOneColThm} applied to $H_*$ that the outer cycle of $H^*$ contains the path $q_0wq_1$, which is false, as $w\not\in V(D^*)$. \end{claimproof}\end{addmargin}

Let $c$ be as in Subclaim \ref{SubCClCfWrGB4}. Since $q_{1-j}$ is incident to no chords of $C$, we have $|L_{\phi\cup\psi}(q_{1-j})\setminus\{c\}|\geq 3$. It follows from our choice of $\psi$ that the extension of $\phi\cup\psi$ to an $L$-coloring of $\textnormal{dom}(\phi\cup\psi)\cup\{z\}$ obtained by coloring $z$ with $c$ lies in $\textnormal{Crown}(P, G)$, contradicting our assumption that $G$ is a counterexample. \end{claimproof}

\begin{Claim}\label{NoCommToBothQ0Q1ExcZ} There is no vertex of $G$ adajcent to both of $q_0, q_1$, except for $z$. \end{Claim}

\begin{claimproof} Suppose there is a $w\in N(q_0)\cap N(q_1)$ with $w\neq z$. As $P$ is induced, we have $z\in N(w)$ by our triangulation conditions. By \ref{MainCLFac} of Claim \ref{IfwAdj10zq1ThenAdjCAllInt}, $N(w)\cap V(C)=\{q_0, z, q_1\}$, so any chord of the outer cycle of $G-z$ is incident to precisely one of $q_0, q_1$, and the outer cycle of $G^{P_{\pentagon}}-z$ is induced. As $|L(z)|=5$, it follows from \ref{SeconLabFact} of Claim \ref{IfwAdj10zq1ThenAdjCAllInt} that any $L$-coloring of $V(P-z)$ extends to $L$-color all of $V(G^{P_0}\cup G^{P_1})$, and since the outer cycle of $G^{P_{\pentagon}}-z$ is induced, it follows from Lemma \ref{PartialPathColoringExtCL0} that any $L$-coloring of $\{q_0, x_0, x_1, q_1\}$ extends to $L$-color all of $G^{P_{\pentagon}}-z$. Thus, any $L$-coloring of $V(P-z)$ extends to $L$-color all of $V(G-z)$. 

\vspace*{-8mm}
\begin{addmargin}[2em]{0em}
\begin{subclaim}\label{OneSideOrOtherK0K1For1to2} There is a $\phi\in\mathcal{X}(P)$ with either $|\Lambda_{G^{P_0}}(\phi(p_0), \phi(q_0), \bullet)|\geq 2$ or $|\Lambda_{G^{P_1}}(\bullet, \phi(q_1), \phi(p_1))|\geq 2$.
\end{subclaim}

\begin{claimproof} Fix an arbitrary $a\in L(p_0)$ and $b\in L(p_1)$. There is a family $\phi_1, \cdots, \phi_5$ of five elements of $\mathcal{X}(P)$, each of which uses a different color on $z$ and uses $a,b$ on $p_0, p_1$ respectively. For each $1\leq k\leq 5$, let $\phi_k'$ be the restriction of $\phi_k$ to $V(P-z)$ and let $\mathcal{C}(\phi_{k'})$ be the set of $c\in L(w)$ such that $\phi_{k'}$ extends to an $L$-coloring of $G-z$ using $c$ on $w$. As any $L$-coloring of $V(P-z)$ extends to $L$-color $V(G-z)$, it follows that, for each $k$, we have $\mathcal{C}(\phi_{k'})\neq\varnothing$, so $\mathcal{C}(\phi_{k'})=\{\phi_k(z)\}$. Thus, $\phi_1', \cdots, \phi_5'$ are all distinct, so there is an $i\in\{0,1\}$ with $|\{\phi_k'(q_i): k=1, \cdots, 5\}|\geq 3$. If $G^{P_i}$ is not a broken wheel, then the subclaim now follows from Theorem \ref{EitherBWheelOrAtMostOneColThm}, and otherwise it follows from \ref{BWheel2Lb} of Theorem \ref{BWheelMainRevListThm2}.  \end{claimproof}\end{addmargin}

Let $\phi$ be as in Subclaim \ref{OneSideOrOtherK0K1For1to2}. Now, there is a $j\in\{0,1\}$ such that restriction of $\phi$ to $q_jp_j$ extends to two different $L$-colorings of $G^{P_j}$ using different colors on $x_j$. Suppose without loss of generality that $j=1$.  No generality is lost in this assumption as $\phi$ fixes the colors being used on $p_0, p_1$. Now, $\phi$ extends to a $\psi$ of $V(G^{P_0}\cup P)$. Let $\tilde{P}:=x_0q_0wq_1x_1$. It follows from \ref{MainCLFac} of Claim \ref{IfwAdj10zq1ThenAdjCAllInt} that $C^{\tilde{P}}$ is induced. We show now that $G^{\tilde{P}}\setminus C^{\tilde{P}}$ contains an edge $v_0v_1$, where each $v_k$ is the central vertex of a wheel as illustrated in Figure \ref{Neitherk0k1EdgeFiguMainSec}, where $\tilde{P}$ is in bold black.

\begin{center}\begin{tikzpicture}
\node[shape=circle,draw=black] [label={[xshift=-0.0cm, yshift=-1.3cm]\textcolor{red}{$\{\psi(p_0)\}$}}] (p0) at (-4,0) {$p_0$};
\node[shape=circle,draw=white] (u1+) at (-3, 0) {$\ldots$};
\node[shape=circle,draw=black] [label={[xshift=-0.0cm, yshift=-1.3cm]\textcolor{red}{$\{\psi(x_0)\}$}}] (x0) at (-2, 0) {$x_0$};
\node[shape=circle,draw=white] (mid-) at (-1, 0) {$\ldots$};
\node[shape=circle,draw=black] (v) at (0, 0) {$v$};
\node[shape=circle,draw=white] (mid) at (1, 0) {$\ldots$};
\node[shape=circle,draw=black] (x1) at (2, 0) {$x_1$};
\node[shape=circle,draw=white] (un+) at (3, 0) {$\ldots$};
\node[shape=circle,draw=black] [label={[xshift=-0.0cm, yshift=-1.3cm]\textcolor{red}{$\{\psi(p_1)\}$}}] (p1) at (4, 0) {$p_1$};
\node[shape=circle,draw=black] [label={[xshift=-1.1cm, yshift=-0.6cm]\textcolor{red}{$\{\psi(q_0)\}$}}] (q0) at (-3,2) {$q_0$};
\node[shape=circle,draw=black] [label={[xshift=1.1cm, yshift=-0.6cm]\textcolor{red}{$\{\psi(q_1)\}$}}] (q1) at (3,2) {$q_1$};
\node[shape=circle,draw=black] [label={[xshift=-0.0cm, yshift=0.1cm]\textcolor{red}{$\{\psi(z)\}$}}] (z) at (0,4) {$z$};
\node[shape=circle,draw=black] (w) at (0, 2.5) {$w$};
\node[shape=circle,draw=white] (K0) at (-3, 1) {$G^{P_0}$};
\node[shape=circle,draw=black] (v0) at (-1, 1.2) {$v_0$};
\node[shape=circle,draw=black] (v1) at (1, 1.2) {$v_1$};
\node[shape=circle,draw=white] (K1) at (3, 1) {$G^{P_1}$};
\node[shape=circle,draw=white] (NewH) at (1.0, 0.5) {$H$};
 \draw[-] (p1) to (un+);
 \draw[-, line width=1.8pt] (x0) to (q0) to (w) to (q1) to (x1);
\draw[-] (q1) to (p1);
\draw[-] (p0) to (u1+) to (x0) to (mid-) to (v) to (mid) to (x1) to (un+);
\draw[-] (p0) to (q0);
\draw[-] (q0) to (z) to (q1);
\draw[-] (z) to (w) to (v0) to (x0);
\draw[-] (v) to (v0) to (v1) to (v);
\draw[-] (v0) to (q0);
\draw[-] (w) to (v1) to (q1);
\draw[-] (x1) to (v1);
 \draw[-, line width=1.8pt, color=red] (x0) to (v0) to (v1) to (v);
\end{tikzpicture}\captionof{figure}{}\label{Neitherk0k1EdgeFiguMainSec}\end{center}

Let $S:=\Lambda_{G^{P_1}}(\bullet, \phi(q_1), \phi(p_1))$ and let $T:=L_{\psi}(w)$. By Lemma \ref{PartialPathColoringExtCL0} aplied to $G^{\tilde{P}}+z$, we have $|T|=2$. Let $\tilde{L}$ be a list-assignment for $V(G^{\tilde{P}})$ in which $x_0, q_0, q_1$ are precolored by $\psi$, and furthermore, $\tilde{L}(x_1)=\{\psi(x_1)\}\cup S$ and $\tilde{L}(z)=L(w)\setminus\{\psi(z)\}$, and otherwise $\tilde{L}=L$. Now, $|\tilde{L}(x_1)|=3$ and $G^{\tilde{P}}$ is not $\tilde{L}$-colorable. By Lemma \ref{PartialPathColoringExtCL0} applied to $G^{\tilde{P}}$, there is a $v_0\in V(G^{\tilde{P}}\setminus C^{\tilde{P}})$ adjacent to at least three vertices of $\tilde{P}-x_1$. Note that $N(v_0)\cap V(\tilde{P}-x_1)=\{x_0, q_0, w\}$ by our triangulation conditions. Precoloring $v_0$ with an appropriate color, another application of Lemma \ref{PartialPathColoringExtCL0} shows that there is a vertex $v_1$ of $G^{\tilde{P}}\setminus C^{\tilde{P}}$ adjacent to each of $v_0, w, q_1$. Let $v$ be the farthest neighbor from $v_1$ on the path $x_0(C\setminus\mathring{P})x_1$ and let $H:=G^{vv_1x_1}$, as illustrated in Figure \ref{Neitherk0k1EdgeFiguMainSec}. 

\vspace*{-8mm}
\begin{addmargin}[2em]{0em}
\begin{subclaim}\label{SubForV0VStar} $v_0v\in E(G)$ and $H$ is a broken wheel, where $|V(H)|$ is even. \end{subclaim}

\begin{claimproof} Let $H^*:=G^{x_0v_0v_1v}$ and . The 3-path $x_0v_0x_1x_1$ is indicated in bold red in Figure \ref{Neitherk0k1EdgeFiguMainSec}. By \ref{LabCrownNonEmpt3} of Theorem \ref{CombinedT1T4ThreePathFactListThm}, there is an $(x_0v_0v_1v, H^*)$-sufficient $L$-coloring $\sigma$ of $\{x_0, v\}$ with $\sigma(x_0)=\psi(x_0)$. Since $|L_{\psi\cup\sigma}(v_1)|\geq 3$ and $|S|\geq 2$, there is a $c\in L_{\psi\cup\sigma}(v_1)$ such that the $L$-coloring $(\sigma(v), c)$ of $vv_1$ extends to an $L$-coloring of $G^{vv_1x_1}$ which uses a color of $S$ on $v_1$, so $\psi\cup\sigma$ extends to an $L$-coloring $\psi'$ of $G^{P_0}\cup G^{P_1}\cup P\cup H$ with $\psi'(v_1)=c$. Since $\psi'$ does not extend to $L$-color $G$, it follows from our choice of $\sigma$ that it does not extend to $L$-color the edge $v_0w$, so $L_{\psi'}(v_0)$ and $|L_{\psi'}(w)$ are the same singleton. In particular, since $|L_{\psi'}(v_0)|=1$, we have $v_0v\in E(G)$. Now we show the second part of the subclaim. Since $|L_{\psi}(v_0)|\geq 3$, there is a $d\in L_{\psi}(v_0)\setminus T$, so $\psi$ extends to an $L$-coloring $\tau$ of $\textnormal{dom}(\psi)\cup V(H)$ using $d$ on $v_0$. Since $d\not\in T$ and $\tau$ does not extend to $L$-color $G$, it follows that there is no $L$-coloring of $H$ using $\tau(v)$ on $v$, a color of $L_{\tau}(v_1)$ on $v_1$, and a color of $S$ on $x_1$. Since $|L_{\tau}(v_1)|\geq 2$ and $|S|\geq 2$, it follows from \ref{PropCor2} of Proposition \ref{CorMainEitherBWheelAtM1ColCor} that $H$ is a broken wheel with an even number of vertices. \end{claimproof}\end{addmargin}

Let $u_{x_1}$ be the unique neighbor of $x_1$ on the path $x_0(C\setminus\mathring{P})x_1$. As $|V(H)|$ is even, $u_{x_1}\not\in N(v_0)$. Consider the 3-path $Q:=x_0v_0v_1u_{x_1}$. By \ref{LabCrownNonEmpt4} of Theorem \ref{CombinedT1T4ThreePathFactListThm}, there is a $\mathcal{G}^Q$-base-coloring $\pi$ of $\{x_0, u_{x_1}\}$ using $\psi(x_0)$ on $x_0$. Letting $\mathcal{F}$ be the set of extensions of $\pi$ to $L$-colorings of $V(Q)$ which do not extend to $L$-color $G^Q$, we have $|\mathcal{F}|\leq 2$. As $|S|\geq 2$, it follows that $\pi$ extends to an $L$-coloring $\pi'$ of $\{x_0, u_{x_1}, x_1\}$, where $\pi'(x_1)\in S$. Let $\mathcal{F}'$ be the set of $L$-colorings $\tau$ of $Q$ using $\psi(x_0), \pi(u_{x_1})$ on $x_0, u_{x_1}$ respectively, where $\tau$ uses a color of $L_{\psi\cup\pi'}(v_0)$ on $v_0$ and a color of $L_{\psi\cup\pi'}(v_1)$ on $v_1$. Since $\pi'(x_1)\in S$ but $\psi\cup\pi'$ does not extend to $L$-color $G$, it follows that, for any $\tau\in\mathcal{F}'\setminus\mathcal{F}$ we have $\{\tau(v_0), \tau(v_1)\}=T$, so $|\mathcal{F}'\setminus\mathcal{F}|\leq 2$. We also have $|\mathcal{F}'\cap\mathcal{F}|\leq 2$, so $|\mathcal{F}'|\leq 4$. As $u_{x_1}\not\in N(v_0)$, we have $|L_{\psi\cup\pi'}(v_0)|\geq 3$. We also have $|L_{\psi\cup\pi'}(v_1)|\geq 2$, so $|\mathcal{F}'|=4$, and we get $|L_{\psi\cup\pi'}(v_0)|=3$ and $|L_{\psi\cup\pi'}(v_1)|=2$, where $L_{\psi\cup\pi'}(v_1)\subseteq L_{\psi\cup\pi'}(v_0)$. Furthermore, $|\mathcal{F}'\setminus\mathcal{F}|=2$ and $\mathcal{F}'\cap\mathcal{F}=\mathcal{F}$, so $L_{\psi\cup\pi'}(v_0)=L_{\psi\cup\pi'}(v_1)=T$, and every element of $\mathcal{F}'\cap\mathcal{F}=\mathcal{F}$ uses the lone color of $L_{\psi\cup\pi'}(v_0)\setminus T$ on $v_0$. By definition of $\pi$, $H-x_1$ is a broken wheel with an even number of vertices, contradicting Subclaim \ref{SubForV0VStar}. This proves Claim \ref{NoCommToBothQ0Q1ExcZ}. \end{claimproof}

\subsection{Ruling out a common neighbor to $x_0, x_1$ in $G\setminus C$}\label{SubsecShowX0X1NoComm}

To prove that there is no such common neighbor of $x_0, x_1$, we first prove the following.

\begin{claim}\label{ForEachCTwoLRestrict} Both of the following hold.
\begin{enumerate}[label=\arabic*)]
\itemsep-0.1em
\item  For each $a\in L(p_1)$ and $b\in L(z)$, there is an $L$-colorings $\psi$ of $V(G^{P_0}\cup G^{P_1}\cup P)$ whose restriction to $V(P)$ lies in $\mathcal{X}(P)$, where $\psi(p_1)=a$ and $\psi(z)=b$; AND
\item There exist $\psi, \psi'\in\mathcal{X}(G^{P_0}\cup G^{P_1}\cup P)$ which do not restrict to the same $L$-coloring of $\{x_0, x_1\}$. 
\end{enumerate}
 \end{claim}

\begin{claimproof} 1) just follows from \ref{SeconLabFact} of Claim \ref{IfwAdj10zq1ThenAdjCAllInt}, together with the fact that $G$ is a counterexample. Now we prove 2). Suppose 2) does not hold. As $|L(p_1)|=3$, it follows from 1) that $G^{P_1}$ is not an edge, and, in particular, there is a fixed color $c\in L(x_1)$ such that, for each $a\in L(p_1)$, there is a $\phi\in\mathcal{X}(P)$ with $\phi(p_1)=a$ and $\Lambda_{G^{P_1}}(\bullet, \phi(q_1), \phi(p_1)))=\{c\}$, so $G^{P_1}$ is a broken wheel with principal path $P_1$ and we contradict \ref{BWheel1Lb} \ref{BWheel1C} of Theorem \ref{BWheelMainRevListThm2}. \end{claimproof}

\begin{claim}\label{X0X1NoCommonNbrG-C} $x_0, x_1$ have no common neighbor in $G\setminus C$. \end{claim}

\begin{claimproof} Let $Y:=N(x_0)\cap N(x_1)\cap V(G\setminus C)$. Note that $|Y|\leq 2$. Suppose toward a contradiction that $Y\neq\varnothing$. Thus, there is a $v\in Y$ such, letting $D$ be the 6-cycle $vx_0q_0zq_1x_1$, there are no vertices of $Y$ in $\textnormal{Int}(D)\setminus D$. Let $H:=G^{x_0vx_1}$ and, for any $L$-coloring $\psi$ of $V(G^{P_0}\cup G^{P_1}\cup P)$, we let $S_{\psi}:=\Lambda_H(\psi(x_0), \bullet, \psi(x_1))\cap L_{\psi}(v)$. We let $\mathcal{X}^{\textnormal{res}}$ denote the set of $L$-colorings of $V(G^{P_0}\cup G^{P_1}\cup P)$ whose restriction to $V(P)$ lies in $\mathcal{X}(P)$.

\vspace*{-8mm}
\begin{addmargin}[2em]{0em} 
\begin{subclaim} $D$ is an induced cycle. \end{subclaim}

\begin{claimproof} Suppose not. By \ref{SeconLabFact} of Claim \ref{IfwAdj10zq1ThenAdjCAllInt}, $P_{\pentagon}$ is an induced path in $G$, so all chords of $D$ lie in $\textnormal{Int}(D)$ and are incident to $v$.  By Claim \ref{NoCommToBothQ0Q1ExcZ}, $v$ is adjacent to at most one of $q_0, q_1$. As $G$ has no induced 4-cycles, there is precisely one chord of $D$, and this chord is $vq_k$ for some $k\in\{0,1\}$. Now, by 2) of Claim \ref{ForEachCTwoLRestrict}, there are $\psi, \psi'\in\mathcal{X}(G^{P_0}\cup G^{P_1}\cup P)$ which restrict different $L$-colorings of $\{x_0, x_1\}$. Note now that every $L$-coloring of $V(D)$ extends to $L$-color $V(\textnormal{Int}(D))$, or else Theorem \ref{BohmePaper5CycleCorList} applied to the 5-cycle $(D\setminus\{x_k\})+vq_k$ implies that $q_0, q_1$ have a common neighbor other than $z$, contradicting Claim \ref{NoCommToBothQ0Q1ExcZ}. Thus, neither $\psi$ nor $\psi'$ extends to $L$-color $V(\textnormal{Ext}(D))$, so $S_{\psi}=S_{\psi'}=\varnothing$. Since $v$ only has three neighbors in $V(G^{P_0}\cup G^{P_1}\cup P)$, it follows from \ref{PropCor3} of Proposition \ref{CorMainEitherBWheelAtM1ColCor} applied to $H$ that $|L_{\psi}(v)|=|L_{\psi'}(v)|=2$, and, in particular, $\psi\neq\psi$ and $\psi\neq\psi$, so, again by \ref{PropCor3} of Propostion \ref{CorMainEitherBWheelAtM1ColCor}, $\psi$ and $\psi'$ restrict to the same $L$-coloring of $\{x_0, x_1\}$, a contradiction.  \end{claimproof}\end{addmargin}

As $D$ is induced, we have $N(v)\cap V(D)=\{x_0, x_1\}$, so, for each $\psi\in\mathcal{X}^{\textnormal{res}}$, we have $S_{\psi}=\Lambda_H(\psi(x_0), \bullet, \psi(x_1))$, and $|L_{\psi}(v)|\geq 3$. Thus, it follows from \ref{PropCor3} of Proposition \ref{CorMainEitherBWheelAtM1ColCor} applied to $H$ that $S_{\psi}\neq\varnothing$ and furthermore, if $\psi$ uses the same color on $x_0, x_1$, then $|S_{\psi}|\geq 2$. The above implies that each $\psi\in\mathcal{X}^{\textnormal{res}}$ extends to an $L$-coloring of $V(\textnormal{Ext}(D))$ which does not extend to $L$-color $G$. Since $N(q_0)\cap N(q_1)=\{z\}$, and since $N(x_0)\cap N(x_1)\cap V(\textnormal{Int}(D))=\{v\}$  by our choice of $v$, it follows from Theorem \ref{BohmePaper5CycleCorList} that $\textnormal{Int}(D)\setminus D$ consists of an edge $y_0y_1$, where $N(y_k)=\{y_{1-k}, x_k, q_k, v\}$ for each $k\in\{0,1\}$. Furthermore, for each $\psi\in\mathcal{X}^{\textnormal{res}}$, we have:
\begin{equation}\label{StarEqIdtaustaus}\tag{$\star$} |L_{\psi}(y_0)|=|L_{\psi}(y_1)|=2,\ \textnormal{and}\ S_{\psi}\subseteq L_{\psi}(y_0)\cap L_{\psi}(y_1)\end{equation}
Note that (\ref{StarEqIdtaustaus})  implies that $1\leq |S_{\psi}|\leq 2$ for each $\psi\in\mathcal{X}^{\textnormal{res}}$, and that $L(z)=L(y^0)=L(y^1)$. 

\vspace*{-8mm}
\begin{addmargin}[2em]{0em} 
\begin{subclaim}\label{EveryPsiExist} For every $\psi\in\mathcal{X}^{\textnormal{res}}$, we have $\psi(x_0)=\psi(x_1)$. \end{subclaim}

\begin{claimproof} It suffices to show that there exists a $\psi^*\in\mathcal{X}^{\textnormal{res}}$ with $\psi^*(x_0)=\psi^*(x_1)$. To see this, suppose there is such a $\psi^*$, and let $\psi$ be an arbitrary element of $\mathcal{X}^{\textnormal{res}}$. Note that $|S_{\psi^*}|=2$, and, since $|L_{\psi^*}(v)|\geq 4$, \ref{PropCor3} of Proposition \ref{CorMainEitherBWheelAtM1ColCor} applied to $H$ implies that either $\psi(x_0)=\psi(x_1)$ or $S_{\psi}=L_{\psi}(v)$. In the latter case, $|S_{\psi}|\geq 3$, which is false, so $\psi(x_0)=\psi(x_1)$. Thus, we now show that there exists a $\psi^*\in\mathcal{X}^{\textnormal{res}}$ using the same color on $x_0, x_1$. Now, for each $r\in L(z)$, there is an element of $\mathcal{X}^{\textnormal{res}}$ using $r$ on $z$, so there exists a $c\in L(x_1)$ and distinct $\psi, \psi'\in\mathcal{X}^{\textnormal{res}}$ with $\psi(z)\neq\psi'(z)$ and $\psi(x_1)=\psi'(x_1)=c$. If either of $\psi, \psi'$ uses $c$ on $x_0$, then we are done, so suppose $c\not\in\{\psi(x_0), \psi'(x_1)\}$. Thus, (\ref{StarEqIdtaustaus})  implies that $\psi(q_0)=\psi'(q_0)=c$, and since $|L(p_0)|=1$, we have $\Lambda_{G^{P_0}}(\psi(p_0), \psi(q_0), \bullet)=\Lambda_{G^{P_0}}(\psi'(p_0), \psi'(q_0), \bullet)$. Since neither of the the restrictions of $\psi, \psi'$ to $V(P)$ extends to $L$-color $G$, we may suppose that $\psi(x_0)=\psi'(x_0)=d$ for some color $d\neq c$, and so $S_{\psi}=S_{\psi'}$, so (\ref{StarEqIdtaustaus}) also implies that $\psi(q_1)=\psi'(q_1)=d$ as well. It also implies that $S_{\psi}=S_{\psi'}=\{f\}$ for some color $f$, where $L(y_0)=L(y_1)=\{c,d,f, \psi(z), \psi'(z)\}$, and that $c\in L(z)$. Thus, there is a $\pi\in\mathcal{X}^{\textnormal{res}}$ with $\pi(z)=c$, and, again by (\ref{StarEqIdtaustaus}), $\pi(x_1)\neq c$, so $\pi$ does not restrict to the same $L$-coloring of $\{x_0, x_1\}$ as $\psi$. Since $S_{\psi}=\{f\}$, we have $L_{\psi}(v)\setminus S_{\psi}|\geq 2$, so, applying \ref{PropCor3} of Proposition \ref{CorMainEitherBWheelAtM1ColCor} again, either $\pi(x_0)=\pi(x_1)$ or $S_{\pi}=L_{\pi}(v)$. In the latter case, $|S_{\pi}|\geq 3$, which is false, so $\pi(x_0)=\pi(x_1)$ and we are done. \end{claimproof}\end{addmargin}

By Subclaim \ref{EveryPsiExist}, for each $\psi\in\mathcal{X}^{\textnormal{res}}$, we have $|S_{\psi}|=2$, and, by (\ref{StarEqIdtaustaus}), we get $\psi(z)\in L(y_1)\setminus (S_{\psi}\cup\{\psi(x_1)\})$. It follows that, for each $c\in L(x_1)$, the elements of $\{\psi\in\mathcal{X}^{\textnormal{res}}: \psi(x_1)=c\}$ use at most two of the colors of $z$. Thus, for each $c\in L(x_1)$, there is a $\psi\in\mathcal{X}^{\textnormal{res}}$ with $\psi(x_0)=\psi(x_1)=c$. For each such $\psi$, we have $\Lambda_{G^{P_0}}(\psi(p_0), \psi(q_0), \bullet)=\{c\}$, so $G^{P_0}$ is a broken wheel and we contradict \ref{BWheel1Lb} \ref{BWheel1C} of Theorem \ref{BWheelMainRevListThm2}. This proves Claim \ref{X0X1NoCommonNbrG-C}. \end{claimproof}

\subsection{Completing the Proof of Theorem \ref{MainHolepunchPaperResulThm}}

For each $i\in\{0,1\}$, we let $\textnormal{Ob}(x_i)$ be the set of $y\in V(G\setminus C)$ such that $N(y)\cap V(P_{\pentagon})=\{x_i, q_i, z\}$.

\begin{Claim}\label{EachObViNonempty} For each $i=0,1$, both of the following hold.
\begin{enumerate}[label=\arabic*)]
\itemsep-0.1em
\item There is a $\phi\in\mathcal{X}(G^{P_0}\cup G^{P_1}\cup P)$, where each vertex of $\textnormal{Ob}(x_i)$ has an $L_{\phi}$-list of size at least three; AND
\item  $\textnormal{Ob}(x_i)|=1$.
\end{enumerate} \end{Claim}

\begin{claimproof} Note that, for each $i\in\{0,1\}$, we have $|\textnormal{Ob}(x_i)|\leq 1$, as $G$ is $K_{2,3}$-free.

\vspace*{-8mm}
\begin{addmargin}[2em]{0em}
\begin{subclaim}\label{ObV0ObV1AllObstructions} For any $y\in V(G^{P_{\pentagon}}\setminus C^{P_{\pentagon}})$ with at least three neighbors in $P_{\pentagon}$, $w\in\textnormal{Ob}(x_0)\cup\textnormal{Ob}(x_1)$. \end{subclaim}

\begin{claimproof} Suppose that $y\not\in\textnormal{Ob}(x_0)\cup\textnormal{Ob}(x_1)$. By Claim \ref{NoCommToBothQ0Q1ExcZ}, there is a $k\in\{0,1\}$ with $y\not\in N(q_k)$. As $y\not\in\textnormal{Ob}(x_{1-k})$, either $N(y)\cap V(P_{\pentagon})=\{z, q_{1-k}, x_k\}$ or $N(y)\cap V(P_{\pentagon})=\{q_k, x_0, x_1\}$. In the former case, as $x_0\neq x_1$, we contradict \ref{MainCLFac} of Claim \ref{IfwAdj10zq1ThenAdjCAllInt}. In the latter case, we contradict Claim \ref{X0X1NoCommonNbrG-C}. \end{claimproof}\end{addmargin}

We now show that each $i\in\{0,1\}$ satisfies 1). Fix an $i\in\{0,1\}$ and suppose $i$ violates 1). Thus, it follows from 1) of Claim \ref{ForEachCTwoLRestrict} that $\textnormal{Ob}(x_i)\neq\varnothing$, so $\textnormal{Ob}(x_i)=\{y\}$ for some vertex $y$, and, by \ref{ThirdLabFact} of Claim \ref{IfwAdj10zq1ThenAdjCAllInt}, $G^{P_i}$ is not an edge. Now it also follows from 1) of Claim \ref{ForEachCTwoLRestrict} that $L(z)=L(y)$. We now fix a $\sigma\in\textnormal{End}(P_0, G^{P_0})$. We note now that $\sigma$ extends to an $L$-coloring $\sigma'$ of $\{p_0, x_0, z, p_1\}$ with $|L_{\sigma}(q_0)|\geq 3$ and $|L_{\sigma'}(y)|\geq 3$: This is immediate if $\sigma(x_0)\not\in L(y)$, as $|L_{\sigma}(z)|\geq 4$. Since $x_0x_1\not\in E(G)$, it follows from 2) of Claim \ref{PInducedNoPGSuff1} that $\sigma'$ extends to a $\phi\in\mathcal{X}(G^{P_0}\cup G^{P_1}\cup P)$, and $|L_{\phi}(y)|\geq 3$, contradicting our assumption on $i$. Thus, each $i\in\{0,1\}$ satisfies 1). Now suppose Claim \ref{EachObViNonempty} does not hold. Thus, there is a $k\in\{0,1\}$ with $\textnormal{Ob}(x_k)=\varnothing$. By Subclaim \ref{ObV0ObV1AllObstructions}, every vertex of $G^{P_{\pentagon}}\setminus C^{P_{\pentagon}}$ with at least three neighbors on $P_{\pentagon}$ lies in $\textnormal{Ob}(x_{1-k})$. By 1), there is a $\psi\in\mathcal{X}(G^{P_0}\cup G^{P_1}\cup P)$, where every vertex of $\textnormal{Ob}(x_{1-k})$ has an $L_{\psi}$-list of size at least three. In particular, $\psi$ does not extend to $L$-color $G^{P_{\pentagon}}$. Since $C^{P_{\pentagon}}$ is induced, it follows from Lemma \ref{PartialPathColoringExtCL0}, there is a $y\in V(G^{P_{\pentagon}}\setminus C^{P_{\pentagon}})$ with $|L_{\psi}(y)|<3$. Thus, $y\in\textnormal{Ob}(x_{1-k})$, a contradiction. \end{claimproof}

Applying Claim \ref{EachObViNonempty}, let $\textnormal{Ob}(x_i)=\{y_i\}$ for each $i=0,1$. This is illustrated in Figure \ref{Y0Y1ObstructionVertexFigure}, where $P_{\pentagon}$ is in bold.  

\begin{center}\begin{tikzpicture}
\node[shape=circle,draw=black] (p0) at (-4,0) {$p_0$};
\node[shape=circle,draw=white] (u1+) at (-2.5, 0) {$\ldots$};
\node[shape=circle,draw=black] (v0) at (-1, 0) {$x_0$};
\node[shape=circle,draw=white] (mid) at (0, 0) {$\ldots$};
\node[shape=circle,draw=black] (v1) at (1, 0) {$x_1$};
\node[shape=circle,draw=white] (un+) at (3, 0) {$\ldots$};
\node[shape=circle,draw=black] (p1) at (4, 0) {$p_1$};
\node[shape=circle,draw=black] (q0) at (-3,2) {$q_0$};
\node[shape=circle,draw=black] (q1) at (3,2) {$q_1$};
\node[shape=circle,draw=black] (z) at (0,4) {$z$};
\node[shape=circle,draw=black] (y0) at (-1.5,2) {$y_0$};
\node[shape=circle,draw=black] (y1) at (1.5,2) {$y_1$};
\node[shape=circle,draw=white] (K0) at (-2.8, 1) {$G^{P_0}$};
\node[shape=circle,draw=white] (K1) at (2.8, 1) {$G^{P_1}$};
 \draw[-] (p1) to (ut);
 \draw[-, line width=1.8pt] (v0) to (q0) to (z) to (q1) to (v1);
\draw[-] (q1) to (p1);
\draw[-] (p0) to (u1+) to (v0) to (mid) to (v1) to (un+) to (p1);
\draw[-] (p0) to (q0);
\draw[-] (q0) to (y0);
\draw[-] (q1) to (y1);
\draw[-] (z) to (y1) to (v1);
\draw[-] (z) to (y0) to (v0);
\end{tikzpicture}\captionof{figure}{}\label{Y0Y1ObstructionVertexFigure}\end{center}

As $G$ is $K_{2,3}$-free, neither $G^{P_0}$ nor $G^{P_1}$ is an edge. To finish the proof of Theorem  \ref{MainHolepunchPaperResulThm}, we first show that $y_0y_1\in E(G)$. By Claim \ref{EachObViNonempty}, there is a $\phi\in\mathcal{X}(G^{P_0}\cup G^{P_1}\cup P)$, where $|L_{\phi}(y_0)|\geq 3$. Now suppose $y_0y_1\not\in E(G)$. By \ref{MainCLFac} of Claim \ref{IfwAdj10zq1ThenAdjCAllInt}, $y_1$ has no neighbors in $x_0(C\setminus P)x_1$, except $x_1$. As $C^{P_{\pentagon}}$ is induced and $y_0y_1\not\in E(G)$, the outer cycle of $G^{P_{\pentagon}}-q_1$ is induced. Now, $\phi$ extends to an $L$-coloring $\phi^*$ of $\textnormal{dom}(\phi)\cup\{y_1\}$, and $|L_{\phi^*}(y_0)|\geq 3$. By  Lemma \ref{PartialPathColoringExtCL0}, there is a $w\in V(G^{P_{\pentagon}}-q_1)$ with at least three neighbors among $x_0, z, y_1, x_1$, where $w\neq y_0$. As $G$ is short-inseperable, we have $w\not\in N(z)$, so $\{x_0, x_1, y_1\}\subseteq N(w)$, contradicting Claim \ref{X0X1NoCommonNbrG-C}. Thus, we indeed have $y_0y_1\in E(G)$, and $\textnormal{deg}(z)=4$. Consider the 3-path $X:=x_0y_0y_1x_1$. By 2) of Claim \ref{ForEachCTwoLRestrict}, there are two $L$-colorings of $V(X)$, neither of which extends to $L$-color $G^X$, where these two $L$-colorings have non-equal restrictions to $\{x_0, x_1\}$. Thus, by \ref{LabCrownNonEmpt2} \ref{T2PartB} of Theorem \ref{CombinedT1T4ThreePathFactListThm}, $G^X$ is a wheel, which again contradicts Claim \ref{X0X1NoCommonNbrG-C}. This proves Theorem \ref{MainHolepunchPaperResulThm}.  \end{proof}

\section{Two 5-Path Analogues to Theorem \ref{MainHolepunchPaperResulThm}}\label{ConseqFinSec}

The remainder of this paper consists of the proof of Theorems \ref{ModifiedRes5ChordCaseDegen} and \ref{RainbowNonEqualEndpointColorThm} below. 

\begin{theorem}\label{ModifiedRes5ChordCaseDegen} Let $\mathcal{G}:=(G, C, P, L)$ be a rainbow, where $|E(P)|=5$ and each vertex of $\mathring{P}$ has a list of size at least five and at least one neighbor in $C\setminus\mathring{P}$. Suppose further that at least one endpoint of $P$ has a list of size at least three. Then, for each endpoint $y$ of the middle edge of $P$, there is a $\phi\in\textnormal{Crown}(P, G)$ with $y\in\textnormal{dom}(\phi)$. \end{theorem}

\begin{proof} Suppose the theorem does not hold, and let $(G, C, P, L)$ be a vertex-minimal counterexample, where $P:=p_0q_0v_0v_1q_1p_1$. By removing colors from the lists of some vertices if necessary, we suppose that each vertex of $C\setminus P$ has an $L$-list of size precisely three and we suppose without loss of generality that $|L(p_0)|=1$ and $|L(p_1)|=3$, although we note that we cannot remove colors from the lists of $\mathring{P}$. Note that our conditions on $P$ imply that no there is no chord of $C$ with both endpoints in $\mathring{P}$. By minimality, together with Corollary \ref{CycleLen4CorToThom} and Theorem \ref{thomassen5ChooseThm}, $G$ is short-inseparable and every chord of $C$ has an endpoint in $C\setminus\mathring{P}$. 

\begin{claim}\label{TwoCrossEdgesRuleOutConfig1} For each $k\in\{0,1\}$, $q_{1-k}, v_{1-k}\not\in N(p_k)$. \end{claim}

\begin{claimproof} Let $k\in\{0,1\}$. Suppose first that $p_kq_{1-k}\in E(G)$. Thus, $N(v_0)\cap V(C\setminus\mathring{P})=N(v_1)\cap V(C\setminus\mathring{P})=\{p_k\}$. Let $G':=G\setminus\{q_k, v_k, v_{1-k}\}$. By Theorem \ref{SumTo4For2PathColorEnds}, there is a $(p_kq_{1-k}p_{1-k}, G')$-sufficient $L$-coloring $\psi$ of $\{p_0, p_1\}$. As $|L_{\phi}(v_{1-k})|\geq 4$, it follows that $\psi$ extends to an $L$-coloring $\phi$ of $\{p_0, v_0, v_1, p_1\}$ with $|L_{\phi}(q_{1-k})|\geq 3$, so $\phi\in\textnormal{Crown}(P,G)$, contradicting our assumption that $G$ is a counterexample. Thus, $q_{1-k}\not\in N(p_k)$. Now suppose that $v_{1-k}\in N(p_k)$. Thus, $N(v_k)\cap V(C\setminus\mathring{P})=\{p_k\}$ and $G\setminus\{q_k, v_k\}=G^R$. By \ref{LabCrownNonEmpt4} of Theorem \ref{CombinedT1T4ThreePathFactListThm}, since $|L(v_{1-k})|\geq 5$, there is an $(R, G^R)$-sufficient $L$-coloring $\phi$ of $\{p_k, v_{1-k}, p_{1-k}\}$. Furthermore, $|L_{\phi}(q_{1-k})|\geq 3$, since $q_{1-k}\not\in N(p_k)$. Choosing an arbitrary color of $L_{\phi}(q_k)$, we extend $\phi$ to an element of $\textnormal{Crown}(P,G)$ with both of $y_0, y_1$ in its domain, contradicting our assumption that $\mathcal{G}$ is a counterexample. \end{claimproof}

Since $G$ is a counterexample, there is a $k\in\{0,1\}$ such that no element of $\textnormal{Crown}(P,G)$ has $v_k$ in its domain. Now, let $y$ be the unique neighbor of $v_{1-k}$ which is closest to $p_{1-k}$ on the path $C\setminus\mathring{P}$. 

\begin{claim}\label{ComClaimUQKVK}  $q_k, v_k, v_{1-k}$ have a common neighbor in $C\setminus P$. Furthermore, letting  $u$ be this unique common neighbor to $q_k, v_k, v_{1-k}$, we have $u\not\in N(v_{1-k})$ and $v_{1-k}\in N(p_{1-k})$.  \end{claim}

\begin{claimproof} Suppose  $q_k, v_k, v_{1-k}$ have no common neighbor in $C\setminus P$. Let $R:=p_kq_kv_kv_{1-k}y$ and $Q:=yv_{1-k}q_{1-k}p_{1-k}$. By Claim \ref{TwoCrossEdgesRuleOutConfig1}, $y\neq p_k$. Suppose first that $Q$ is a triangle. In that case, $y=p_{1-k}$, and Theorem \ref{MainHolepunchPaperResulThm} implies that there is a $\phi\in\textnormal{Crown}(R, G^R)\neq\varnothing$. Furthermore, $N(q_{1-k})=\{v_{1-k}, p_{1-k}\}$ and $v_k\in\textnormal{dom}(\phi)$, so $\phi\in\textnormal{Crown}(P,G)$, contradicting our assumption on $k$. Thus, $Q$ is not a triangle, so $|L(y)|=3$. Since either $|L(p_k)|=1$ or $|L(p_k)|=3$, it follows from Theorem \ref{MainHolepunchPaperResulThm} that there is a family of $|L(p_k)|$ different elements of $\textnormal{Crown}(R, G^R)$, each using a different color on $y$. Thus, it follows from \ref{LabCrownNonEmpt} of Theorem \ref{CombinedT1T4ThreePathFactListThm} applied to $\mathcal{G}^Q$ that there is a $\psi\in\textnormal{Crown}(R, G^R)$ and a $\phi\in\textnormal{Crown}(Q, G^Q)$ with $\psi(y)=\phi(y)$. Note that $v_k\in\textnormal{dom}(\psi)$ by definition and $|L_{\psi\cup\phi}(v_{1-k})|=|L_{\psi}(y)|\geq 3$ by our choice of $y$. Likewise, since $v_{1-k}$ is left uncolored, each of $q_0, q_1$ has an $L_{\psi\cup\phi}$-list of size at least three, and $\psi\cup\phi$ is $(P,G)$-sufficient, so $\psi\cup\phi\in\textnormal{Crown}(P,G)$, contradicting our assumption on $k$. Thus, there is indeed a $u\in V(C\setminus P)$ adjacent to $q_k, v_k, v_{1-k}$. Now suppose $u\in N(v_{1-k})$. By two applications of Theorem \ref{SumTo4For2PathColorEnds} (to $G^{p_kq_ku}$ and $G^{uq_{1-k}p_{1-k}}$ respectively), there is a $(P,G)$-sufficient $L$-coloring $\phi$ of $\{p_k, u, p_{1-k}\}$. Since $|L_{\phi}(v_k)|\geq 4$, $\phi$ extends to an $L$-coloring $\phi'$ of $\textnormal{dom}(\phi)\cup\{v_k\}$ with $|L_{\phi'}(q_k)|\geq 3$, and each of $v_{1-k}, q_{1-k}$ also has an $L_{\phi'}$-list of size at least three, contradicting our assumption on $k$. Thus, $u\not\in N(q_{1-k})$. Now suppose $v_{1-k}\in N(p_{1-k})$. As above, two applications of Theorem \ref{SumTo4For2PathColorEnds} (to $G^{p_kq_ku}$ and $G^{uv_{1-k}p_{1-k}}$ respectively) show that  there is a $(P,G)$-sufficient $L$-coloring $\psi$ of $\{p_k, u, p_{1-k}\}$, and since $|L_{\psi}(v_k)|\geq 4$, $\psi$ extends to an $L$-coloring $\psi'$ of $\textnormal{dom}(\psi)\cup\{v_0, v_1\}$ with $|L_{\psi'}(q_k)|\geq 3$. Then $\psi'\in\textnormal{Crown}(P,G)$, contradicting our assumption on $k$. \end{claimproof}

Let $u$ be as in Claim \ref{ComClaimUQKVK} and let $Q_*:=uv_{1-k}q_{1-k}p_{1-k}$. As$|L(u)|=3$, it follows from Theorem \ref{SumTo4For2PathColorEnds} applied to $G^{p_kq_ku}$ and \ref{LabCrownNonEmpt4} of Theorem \ref{CombinedT1T4ThreePathFactListThm} applied to $G^Q$ that there is a $\phi\in\textnormal{End}(p_kq_ku, G^{p_kq_ku})$ and a $\mathcal{G}^Q$-base coloring $\psi$ with $\phi(u)=\psi(u)$. As above, $|L_{\phi\cup\psi}(v_k)|\geq 4$, so there is a $c\in L_{\phi\cup\psi}(v_k)$ with $|L_{\phi\cup\psi}(q_k)\setminus\{c\}|\geq 3$. As $p_{1-k}\not\in N(v_{1-k})$, we have $|L_{\phi\cup\psi}(v_{1-k})\setminus\{c\}|\geq 3$. Thus, by our choice of $\psi$, it follows that $\phi\cup\psi$ extends to a $(P,G)$-sufficient $L$-coloring $\tau$ of $\{p_0, u, p_1, v_0, v_1\}$ with $\tau(v_k)=c$. Now, since $u\not\in N(q_{1-k})$,  we have $|L_{\tau}(q_{1-k})|\geq 3$, so $\tau\in\textnormal{Crown}(P,G)$, contradicting our assumption on $k$. This proves Theorem \ref{ModifiedRes5ChordCaseDegen}. \end{proof}

If both endpoints of $P$  have lists of size at least three, then we can prove a stronger version of Theorem \ref{ModifiedRes5ChordCaseDegen}.

\begin{defn} \emph{Let $\mathcal{G}:=(G, C, P, L)$ be a rainbow, where $|E(P)|=5$. A \emph{$\mathcal{G}$-obstruction} is a vertex $x\in V(C)$ such that either $x$ is adjacent to all four vertices of $\mathring{P}$ or there is a subpath $pqv$ of $P$, where $p$ is an endpoint of $P$, such that $x$ is adjacent to $p$ and to both vertices of $\mathring{P}\setminus\{q, v\}$.} \end{defn}

\begin{theorem}\label{RainbowNonEqualEndpointColorThm} Let $\mathcal{G}:=(G, C, P, L)$ be a rainbow, where $P=p_0q_0v_0v_1q_1p_1$ and each vertex of $\mathring{P}$ has a list of size at least five and at least one neighbor in $C\setminus\mathring{P}$. Then either
\begin{enumerate}[label=\arabic*)]
\itemsep-0.1em
\item There is a $\mathcal{G}$-obstruction; OR
\item For for each $j\in\{0,1\}$, there is an $i\in\{0,1\}$ and an $a\in L(p_i)$ such that, for any $b\in L(p_{1-i})$, there is an element of $\textnormal{Crown}(P, G)$ which has $v_j$ in its domain and uses $a,b$ on $p_i, p_{1-i}$ respectively. 
\end{enumerate}
 \end{theorem}

\begin{proof} Suppose the theorem does not hold, and let $\mathcal{G}=(G, C, P, L)$ be a vertex-minimal counterexample in which $|E(G)|$ is maximized among all vertex-minimal counterexamples. By removing colors from the lists of some vertices if necessary, we suppose that each vertex of $C\setminus\mathring{P}$ has a list of size precisely three and each vertex of $G\setminus C$ has a list of size precisely five, although we note that we cannot remove colors from the lists of $\mathring{P}$. Our conditions on $P$ imply that no there is no chord of $C$ with both endpoints in $\mathring{P}$ and, since there is no $\mathcal{G}$-obstruction, $q_0, q_1$ have no common neighbor in $C$. Furthermore, the vertex-minimality of $G$, together with Corollary \ref{CycleLen4CorToThom} and Theorem \ref{thomassen5ChooseThm}, implies that $G$ is short-inseparable and that every chord of $C$ has an endpoint in $C\setminus\mathring{P}$. For each $i=0,1$, we define the following: Let $y_i, y_i'$ be the neighbors of $v_i$ which are respectively farthest and closest from $p_i$ on the path $C\setminus\mathring{P}$ and furthermore, $Q_i:=p_iq_iv_iy_i$. Let $x_i$ be the neighbor of $q_i$ which is farthest from $p_i$ on the path $C\setminus\mathring{P}$. For each $j=0,1$, let $\mathcal{D}_j$  denote the set of $L$-colorings $\sigma$ of $\{p_0, p_1\}$ such that $\sigma$ does not extend to an element of $\textnormal{Crown}(P,G)$ whose domain contains $v_j$. Let $\mathcal{D}$ denote the set of $L$-colorings $\sigma$ of $\{p_0, p_1\}$ such that $\sigma$ does not extend to an element of $\textnormal{Crown}(P,G)$ whose domain contains $\{v_0, v_1\}$. In particular, $\mathcal{D}_0\cup\mathcal{D}_1\subseteq\mathcal{D}$. Since there is no $\mathcal{G}$-obstruction and each of $v_0, v_1$ has a neighbor in $C\setminus\mathring{P}$, we have $p_0p_1\not\in E(G)$. Since $p_0p_1\not\in E(G)$, it follows that, for each $i\in\{0,1\}$ and $a\in L(p_i)$, there is a $\sigma\in\mathcal{D}$ using $a$ on $p_i$. Note that, for each $i\in\{0,1\}$ and $a\in L(y_i)$, there is a $(Q_i, G^{Q_i})$-sufficient $L$-coloring $\phi^i_a$ of $\{p_i, y_i, v_i\}$ with $|L_{\phi_i^a}(q_i)|\geq 3$ and $\phi_i^a(y_i)=a$. This is immediate if $y_i=p_i$ and otherwise, it follows from \ref{LabCrownNonEmpt4} of Theorem \ref{CombinedT1T4ThreePathFactListThm} (if $q_iy_i\not\in E(G)$ and Theorem \ref{SumTo4For2PathColorEnds} (if $q_iy_i\in E(G)$). 

\begin{claim}\label{Y0Y1MeetMiddle} $y_0=y_1$ \end{claim}

\begin{claimproof} Suppose not. Thus, letting $R:=y_0v_0v_1y_1$, $R$ is a 3-path. If $d(y_0, y_1)=1$, then, since $G$ is short-inseparable, $V(G^R)=V(R)$, and we can add at least one of the chords of $C^R$ to $G$ without creating a $\mathcal{G}$-obstruction, contradicting edge-maximality. Thus, $d(y_0, y_1)>1$. 

\vspace*{-8mm}
\begin{addmargin}[2em]{0em}
\begin{subclaim}\label{NoLoneVertSuff} For each $i\in\{0,1\}$, no $L$-coloring of $\{y_i\}$ is $(R, G^R)$-sufficient. \end{subclaim}

\begin{claimproof} Suppose not. Say without loss of generality that there is an $(R, G^R)$-sufficient $L$-coloring $\psi$ of $\{y_0\}$ . Let $a:=\psi(y_0)$. Now, there is a $\sigma\in\mathcal{D}$ with $\sigma(p_0)=\phi_0^a(p_0)$. By Theorem \ref{SumTo4For2PathColorEnds}, there is a $\phi'\in\textnormal{End}(p_1q_1v_1, G^{Q_1})$ with $\phi'(p_1)=\sigma(p_1)$ and $\phi'(v_1)\neq\phi_0^a(v_0)$, so $\phi_0^a\cup\phi'$ is a proper $L$-coloring of $\{p_0, y_0, v_0, v_1, p_1\}$ which is not $(P,G)$-sufficient. Thus, $\phi_0^a\cup\phi'$ extends to an $L$-coloring $\tau$ of $V(P)\cup\{y_0\}$ which does not extend to $L$-color $G$. Since $d(y_0, y_1)>1$, it follows from our choice of of $\phi_0^a, \phi'$ that $\tau$ extends to $L$-color $V(G^{Q_0}\cup G^{Q_1})$. But then, by our choice of $a$, $\tau$ extends to $L$-color $G$, a contradiction.  \end{claimproof}\end{addmargin}

Since $|L(y_0)|=|L(y_1)|=3$, It follows from Subclaim \ref{NoLoneVertSuff}, together with \ref{LabCrownNonEmpt2} \ref{T2PartB} of Theorem \ref{CombinedT1T4ThreePathFactListThm} that $G^R$ is a wheel. Let $w$ be the central vertex of the wheel and let $H:=G^{y_0wy_1}$. It also follows from Subclaim \ref{NoLoneVertSuff} that, for each $i\in\{0,1\}$, no color of $L(y_i)$ is $(y_0wy_1, H)$-universal. 

\vspace*{-8mm}
\begin{addmargin}[2em]{0em}
\begin{subclaim} $6\leq |V(G^R)|\leq 7$ and all the vertices of $H-w$ have the same 3-list. \end{subclaim}

\begin{claimproof} If $6\leq |V(G^R)|\leq 7$, then $|E(H-w)|\leq 3$ and it follows from Subclaim \ref{NoLoneVertSuff} that all the vertices of $H-w$ have the same 3-list, so suppose $|V(G^R)|>7$. Choose an arbitrary $i\in\{0,1\}$ and let $y_iz_1z_2z_3$ be the unique 3-path of $H-w$ with $y_i$ as an endpoint. By \ref{PropCor1} of Proposition \ref{CorMainEitherBWheelAtM1ColCor} applied to $H$, $L(y_i)=L(z_1)=L(z_2)$. Let $G^{\dagger}$ be a graph obtained from $G$ by deleting $z_1, z_2$ and replacing them with the edge $y_iz_3$, so that $G^{\dagger}$ has outer cycle $C^{\dagger}:=(C\setminus\{z_1, z_2\})+y_iz_3$. Possibly there is one (but not both) of the indices $i\in\{0,1\}$ such that $y_i=p_i$, but, in any case, letting $\mathcal{G}^{\dagger}$ be the rainbow $(G^{\dagger}, C^{\dagger}, P, L)$, there is no $\mathcal{G}^{\dagger}$-obstruction, as the induced cycle obtained by contradicting $C^R$ has length at least five. Thus, by minimality, there is an $L$-coloring $\phi$ of $V(G)\setminus\{z_1, z_2\}$ with $\phi(y_i)\neq\phi(z_3)$, where $\phi$ does not extend to $L$-color $G$, contradicting Observation \ref{MinCounterReUseObs}. \end{claimproof}\end{addmargin}

Let $T$ be the common list iof the vertices of $H-w$. By Subclaim \ref{NoLoneVertSuff}, $T\subseteq L(w)$, so $|L(v_0)\setminus (L(w)\setminus T)|\geq 3$. By Theorem \ref{SumTo4For2PathColorEnds}, there is a $\phi\in\textnormal{End}(p_0q_0v_0, G^{Q_0})$ with $\phi(v_0)\not\in L(w)\setminus T$. Let $a:=\phi(p_0)$ and $\sigma\in\mathcal{D}$ with $\sigma(p_0)=a$. Since $p_0v_1\not\in E(G)$, we have $|L_{\phi}(v_1)|\geq 4$. By Theorem \ref{SumTo4For2PathColorEnds}, there is a $\psi\in\textnormal{End}(p_1q_1v_1, G^{Q_1})$ with $\psi(v_1)\neq\phi(v_0)$ and $\psi(p_1)=\sigma(p_1)$. Now, $\phi\cup\psi$ is a proper $L$-coloring of $\{p_0, v_0, v_1, p_1\}$ which is not $(P,G)$-sufficient, so it extends to an $L$-coloring $\tau$ of $V(P)$ which does not extend to $L$-color $G$. But since $d(y_0, y_1)>1$, $\tau$ does extend to $L$-color $V(G^{Q_0}\cup G^{Q_1})$. Thus, $\tau$ extends to an $L$-coloring of $G-w$. Since all the vertices of $H-w$ are using colors of $T$, there is also a color left for $w$, so $\tau$ extends to $L$-color $G$, a contradiction. This proves Claim \ref{Y0Y1MeetMiddle}. \end{claimproof}

Let $y_0=y_1=y$.  For each $i=0,1$, let $K_i:=G^{y_i'v_iy}$ and $P_i:=p_iq_ix_i$. 

\begin{claim} For each $i=0,1$, $y_i'=x_i$.\end{claim}

\begin{claimproof} Suppose there is an $i\in\{0,1\}$ for which this does not hold, say $i=0$ without loss of generality. In particular, $y_0'\neq p_0$ and $q_1\not\in N(y_0')$. Let $R'$ be the 3-path $x_0q_0v_0y_0'$. If $d(x_0, y_0')=1$, thenthere is at least one chord of $C^{R'}$ that we can add to $G$ without creating a $\mathcal{G}$-obstruction, contradicting edge-maximality. Thus, $d(x_0, y_0')>1$.

\vspace*{-8mm}
\begin{addmargin}[2em]{0em}
\begin{subclaim}\label{NoLColEitherSideSuff} No $L$-coloring of $\{y_0'\}$ is $(R', G^{R'})$-sufficient. \end{subclaim}

\begin{claimproof} Suppose there is such an $L$-coloring $\psi$ of $\{y_0'\}$. Let $a:=\psi(x^*)$. Note that there is a $(y_0'v_0y_0, K_0)$-sufficient $L$-coloring $\rho$ of $\{y_0', y_0\}$ with $\rho(y_0')=a$. Let $b:=\rho(y_0')$. There is a $\sigma\in\mathcal{D}$ with $\sigma(p_1)=\phi^1_b(p_1)$. Since $d(p_0, y_0')>1$, the union $\sigma\cup\rho\cup\phi^1_b$ is a proper $L$-coloring of its domain and extends to an $L$-coloring $\tau$ of $\{y_0', y_0\}\cup (V(P)\setminus\{q_0, q_1\})$. As $y_0'\neq x_0$, each of $q_0, q_1$ has an $L_{\tau}$ list of size at least three. By our choice of $a$, and the fact that $d(x_0, y_0')>1$, $\tau$ is $(P,G)$-sufficient, contradicting our choice of $\sigma$.  \end{claimproof}\end{addmargin}

\vspace*{-8mm}
\begin{addmargin}[2em]{0em}
\begin{subclaim}\label{NoLColEitherSideSuff2} No $L$-coloring of $\{x_0\}$ is $(R', G^{R'})$-sufficient. \end{subclaim}

\begin{claimproof} Suppose there is an $(R', G^{R'})$-sufficient $L$-coloring $\psi$ of $\{x^*\}$. Let $a:=\psi(x^*)$. In this case, there is a $(P_0, G^{P_0})$-sufficient $L$-coloring $\rho$ of $\{p_0, x_0\}$ with $\rho(x_0)=a$. Now, there is a $j\in\{0,1\}$ and a $\sigma\in\mathcal{D}_j$ with $\sigma(p_0)=\rho(p_0)$. Let $b:=\sigma(p_1)$. Consider the following cases:

\textbf{Case 1:} $q_1\not\in N(y)$

In this case, by \ref{LabCrownNonEmpt4} of Theorem \ref{CombinedT1T4ThreePathFactListThm}, there is a $\mathcal{G}^Q$-base coloring $\psi$ of $\{y, p_1\}$ with $\psi(p_1)=b$. As $x_0\neq y_0'$, $\rho\cup\psi$ is a proper $L$-coloring of $\{p_0, x_0, y, p_1\}$ with $|L_{\rho\cup\psi}(v_0)|\geq 4$ and $|L_{\rho\cup\psi}(q_0)|\geq 3$. It follows from our choice of $\psi$ that $\rho\cup\psi$ extends to a $(P,G)$-sufficient $L$-coloring $\tau$ of $\{y\}\cup (V(P)\setminus\{q_0, q_1\})$ with $|L_{\tau}(q_0)|\geq 3$. Since $q_1\not\in N(y)$, $|L_{\tau}(q_1)|\geq 3$ as well, contradicting our choice of $\sigma$. 

\textbf{Case 2:} $q_1\in N(y)$

In this case, applying Theorem \ref{SumTo4For2PathColorEnds} to $G^{Q_1}-v_1$, there is a $(Q_1, G^{Q_1})$-sufficient $L$-coloring $\psi'$ of $\{p_1, y\}$ with $\psi'(p_1)=b$. Now, $\rho\cup\psi'$ is a $(P,G)$-sufficient $L$-coloring of $\{p_0, x_0, y, p_1\}$, where each of $q_0, q_1$ has an $L_{\rho\cup\psi'}$-list of size at least three and each of $v_0, v_1$ has an $L_{\rho\cup\psi'}$-list of size at least four. Thus, coloring $v_j$ appropriately, $\rho\cup\psi'$ extends to an element of $\textnormal{Crown}(P,G)$ with domain $\{p_0, x_0, y, v_j, p_1\}$, contradicting our choice of $j$. \end{claimproof}\end{addmargin}

It follows from Subclaim \ref{NoLColEitherSideSuff}, together with \ref{LabCrownNonEmpt2} \ref{T2PartA} of Theorem \ref{CombinedT1T4ThreePathFactListThm}, that $G^{R'}$ is a wheel. Let $w$ be the central vertex of this wheel and $u$ be the unique neighbor of $y_0'$ on the path $G^{R'}\setminus\{q_0, v_0, w\}$.

\vspace*{-8mm}
\begin{addmargin}[2em]{0em}
\begin{subclaim}\label{WheelOn6VertR'} $6\leq|V(G^{R'})|\leq 7$ and all the vertices of $G^{R'}\setminus\{q_0, v_0, w\}$ have the same 3-list.  \end{subclaim}

\begin{claimproof} If $|V(G^{R'})|\leq 7$, then it follows from Subclaims \ref{NoLColEitherSideSuff}-\ref{NoLColEitherSideSuff2} that all the vertices of $G^{R'}\setminus\{q_0, v_0, w\}$ have the same 3-list, so it suffices to prove that $6\leq |V(G^{R'})|\leq 7$. Suppose not. Thus, $|V(G^{R'})|>7$. Let $y_0'uu'u''$ be the terminal length-three subpath of $G^{R'}\setminus\{q_0, v_0, w\}$ with $y_0'$ as an endpoint. By \ref{PropCor1} of Proposition \ref{CorMainEitherBWheelAtM1ColCor}, $L(y_0')=L(u)=L(u')$. Let $G^{\dagger}$ be a graph obtained from $G$ by deleting $u, u'$ and replacing them with the edge $y_0'u''$, so that $G^{\dagger}$ has outer cycle $C^{\dagger}:=(C\setminus\{u, u'\})+y_0'u''$. Let $\mathcal{G}^{\dagger}:=(G^{\dagger}, C^{\dagger}, P, L)$. The cycle obtained from $C^{R'}$ by  replacing $u, u'$ with $y_0'u''$ has length at least five, so there is no $G^{\dagger}$-obstruction, even if $x_0=p_0$ and $y_0'=y$. Thus, as in Claim \ref{Y0Y1MeetMiddle}, we contradict the minimality of $\mathcal{G}$ by applying Observation \ref{MinCounterReUseObs}. \end{claimproof}\end{addmargin}

Let $T$ be a set equal to the list of each vertex of $G^{R'}\setminus\{q_0, v_0, w\}$. By Subclaim \ref{NoLColEitherSideSuff}, $T\subseteq L(w)$. Fix an $a\in L(y)$ and consider $\phi^a_1$. There is a $\sigma\in\mathcal{D}$ with $\sigma(p_1)=\phi^a_1(p_1)$. As $|L(v_0)\setminus (L(w)\setminus T)|\geq 3$, the union $\sigma\cup\phi^a_1$ extends to an $L$-coloring $\tau$ of $\{p_0, v_0, v_1, y, p_1\}$ with $\tau(v_0)\not\in L(w)\setminus T$. Each of $q_0, q_1$ has an $L_{\tau}$-list of size at least three, so $\tau$ is not $(P,G)$-sufficient. Thus, $\tau$ extends to an $L$-coloring $\tau'$ of $V(P)\cup\{y\}$ which does not extend to $L$-color $G$. On the other hand, since $d(x_0, y_0')>1$, $\tau'$ extends to an $L$-coloring of $G-w$ in which all the vertices of $G^{R'}\setminus\{q_0, v_0, w\}$ use colors of $T$. By our choice of $\tau(v_0)$, there is a color left for $w$, so $\tau'$ extends to $L$-color $G$, a contradiction. \end{claimproof}

\begin{claim}\label{ForEachY0SmallUniv} For each $i\in\{0,1\}$ either $K_i$ is a triangle with $L(x_i)=L(y)$,or $K_i$ is an edge. \end{claim}

\begin{claimproof} Say $i=0$ without loss of generality, and $x_0\neq y$. In particular, $d(p_0, y)>1$. Let $Q_0':=p_0q_0v_0x_0$. Possibly $Q_0'$ is a triangle. We first show that no color of $L(y)$ is $(x_0v_0y, K_0)$-universal. Suppose there is such an $a\in L(y)$. There is a $\sigma\in\mathcal{D}$ with $\sigma(p_1)=\phi_1^a(p_1)$. Now, $|L_{\phi_1^a}(v_0)|\geq 3$, so it follows from Theorem \ref{SumTo4For2PathColorEnds} that there is a $\psi\in\textnormal{End}(p_0q_0v_0, G^{Q_0'})$ with $\psi(p_0)=\sigma(p_0)$ and $\psi(v_0)\in L_{\phi_1^a}(v_0)$. Since $x_0\neq y$ and $p_0y\not\in E(G)$, $\psi\cup\phi_1^a$ is a proper $L$-coloring of $\{p_0, v_0, v_1, y, p_1\}$ and $|L_{\psi\cup\phi_1^a}(q_0)|\geq 3$. We have $|L_{\psi\cup\phi_1^a}(q_1)|\geq 3$ as well, and our choice of $a, \psi, \phi^a$ imply that $\psi\cup\phi_1^a$ is $(P,G)$-sufficient, so $\psi\cup\phi_1^a\in\textnormal{Crown}(P,G)$, contradicting our choice of $\sigma$. We conclude that no color of $L(y)$ is $(x_0v_0y, K_0)$-universal, so $K_0$ is a broken wheel with principal path $x_0v_0y$. Furthermore, letting $u$ be the unique neighbor of $y$ on the path $K_0-v_0$, we have $L(u)=L(y)$. To finish, it suffices to show that $|V(K_0)|=3$. Suppose $|V(K_0)|\geq 4$. Now, there is  $c\in L(v_0)$ with $c\not\in L(y)\cup L(u)$. By Theorem \ref{SumTo4For2PathColorEnds}, there is a $\phi\in\textnormal{End}(p_0q_0v_0, G^{Q_0'})$ with $\phi(v_0)=c$. Let $\sigma\in\mathcal{D}$ with $\sigma(p_0)=\phi(p_0)$, and let $b:=\sigma(p_1)$. By Theorem \ref{SumTo4For2PathColorEnds}, there is a $\psi\in\textnormal{End}(p_1q_1v_1, G^{Q_1})$ with $\psi(p_1)=b$ and $\psi(v_1)\neq c$, so $\phi\cup\psi$ is a proper $L$-coloring of $\{p_0, v_0, v_1, p_1\}$ which is not $(P,G)$-sufficient. Thus, $\phi\cup\psi$ extends to an $L$-coloring $\tau$ of $V(P)$ which does not extend to $L$-color $G$. Since $c\not\in L(y)$ and $d(x_0, y)>1$, it follows from our choice of $\phi, \psi$ that $\tau$ extends to an $L$-coloring of $V(G^{Q_0'}\cup G^{Q_1})$. Since $c\not\in L(u)$, it follows that $\tau$ extends to $L$-color $G$, a contradiction. This proves Claim \ref{ForEachY0SmallUniv}. \end{claimproof}

\begin{claim}\label{PrecOneWithi'Sub} There is precisely one $i\in\{0,1\}$ with $x_i\neq y$.  \end{claim}

\begin{claimproof} Since $q_0, q_1$ have no common neighbor in $C$, there is at least one $i\in\{0,1\}$ with $x_i\neq y$. Suppose that both of $x_0, x_1$ are distinct from $y$. Fix an arbitrary $\sigma\in\mathcal{D}$. Since $|L(v_k)\setminus L(y)|\geq 2$ for each $k=0,1$, $\sigma$ extends to an $L$-coloring $\psi$ of $\{p_0, v_0, v_1, p_1\}$ with $\psi(v_0), \psi(v_1)\not\in L(y)$, and $\psi$ is not $(P,G)$-sufficient, so it extends to an $L$-coloring $\tau$ of $V(P)$ which does not extend to $L$-color $G$. By Claim \ref{ForEachY0SmallUniv}, each of $K_0, K_1$ is a triangle and $L(x_0)=L(y)=L(x_1)$, so $\tau(v_0)\not\in L(x_0)$ and $\tau(v_1)\not\in L(x_1)$. Possibly there is an $i\in\{0,1\}$ with $x_i=p_i$, but, in any case, $\tau$ extends to $L$-color $G-y$ and there is now a color left for $y$, a contradiction. \end{claimproof}

Applying Claims \ref{ForEachY0SmallUniv}-\ref{PrecOneWithi'Sub}, we suppose without loss of generality that $x_0=y$ and $x_1\neq y$, so $K_1$ is a triangle with $L(x_1)=L(y)$, and we have the diagram in Figure \ref{FourVertMathringPCommNB}. As there is no $\mathcal{G}$-obstruction, neither $G^{P_0}$ nor $G^{P_1}$ is an edge.

\begin{center}\begin{tikzpicture}
\node[shape=circle,draw=black] (p1) at (0,0) {$p_0$};
\node[shape=circle,draw=black] (p2) at (0,2) {$q_0$};
\node[shape=circle,draw=black] (p2+) at (2,2) {$v_0$};
\node[shape=circle,draw=black] (p2++) at (4,2) {$v_1$};
\node[shape=circle,draw=black] (u1) at (2,0) {$y$};
\node[shape=circle,draw=black] (p4) at (6,0) {$p_1$};
\node[shape=circle,draw=black] (p3) at (6,2) {$q_1$};
\node[shape=circle,draw=black] (u2) at (4,0) {$x_1$};
\node[shape=circle,fill=white] (dot1) at (1,0) {$\cdots$};
\node[shape=circle,fill=white] (dot2) at (5,0) {$\cdots$};

 \draw[-] (p1) to (dot1) to (u1) to (u2) to (dot2) to (p4) to (p3) to (p2++) to (p2+) to (p2) to (p1);
 \draw[-] (p2) to (u1) to (p2+);
 \draw[-] (u1) to (p2++) to (u2);
 \draw[-] (u2) to (p3);
\end{tikzpicture}\captionof{figure}{}\label{FourVertMathringPCommNB}\end{center}

\begin{claim}\label{ForEachISideTriang} For each $i\in\{0,1\}$, $G^{P_i}$ is a triangle with $L(p_i)=L(x_i)$. \end{claim}

\begin{claimproof} We first show the following. 

\vspace*{-8mm}
\begin{addmargin}[2em]{0em}
\begin{subclaim}\label{NoEndPKUnivFirstCase} Let $k\in\{0,1\}$. Then no color of $L(p_k)$ is $(P_k, G^{P_k})$-universal. \end{subclaim}

\begin{claimproof} Suppose there is an $a\in L(p_k)$ which is $(P_k, G^{P_k})$-universal. Let $\sigma\in\mathcal{D}$ with $\sigma(p_k)=a$. Thus, $\sigma$ extends to an $L$-coloring $\tau$ of $\{p_0, v_0, v_1, p_1\}$ with $\tau(v_0)\not\in L(y)$ and $\tau(v_1)\not\in L(y)\cap L(x_1)$. Our choice of $a, \tau(v_0), \tau(v_1)$ implies that $\tau$ is $(P,G)$-sufficient, contradicting our choice of $\sigma$.  \end{claimproof}\end{addmargin}

It follows from Subclaim \ref{NoEndPKUnivFirstCase} that each of $G^{P_0}$ and $G^{P_1}$ is a broken wheel. 

\vspace*{-8mm}
\begin{addmargin}[2em]{0em}
\begin{subclaim}\label{EitherAOrBUniversality} Let $k\in\{0,1\}$ and $p_kz_1z_2$ be the terminal 2-path of $C-q_k$ containing $p_k$. Then $3\leq |V(G^{P_k})|\leq 4$ and $L(p_k)=L(z_1)$. Lastly, if $|V(G^{P_k})|=4$, then either $|V(G^{P_{1-k}})|=3$ or $L(p_k)\neq L(z_2)$. \end{subclaim}

\begin{claimproof} We have $L(p_k)=L(z_1)$ by Subclaim \ref{NoEndPKUnivFirstCase}. If $G^{P_k}$ is a triangle, then we are done, so suppose $|V(G^{P_k})|\geq 4$. If $L(p_k)\neq L(z_2)$, then, since no color of $L(p_k)$ is $(P_k, G^{P_k})$-universal, it follows from \ref{PropCor1} of Proposition \ref{CorMainEitherBWheelAtM1ColCor} that $|V(G^{P_k})|=4$, so we are done in that case. Now suppose that $L(p_k)=L(z_2)$. Let $M:=z_2q_kv_kv_{1-k}$, so $G^M=G\setminus\{p_k, z_1\}$. For any $\rho\in\textnormal{Crown}(M, G^M)$, $\rho(z_2)\in L(p_k)$, and the $L$-coloring of $\textnormal{dom}(\rho)\cup\{p_k\}$ obtained by coloring $p_k$ with $\rho(z_2)$ is an element of $\textnormal{Crown}(P, G)$. Thus, by the minimality of $G$, there is a $\mathcal{G}^M$ obstruction, so $x_k=z_2$ and $G^{P_{1-k}}$ is a triangle. In particular, $|V(G^{P_k})|=4$ and we are done. \end{claimproof}\end{addmargin}

Now suppose there is a $k\in\{0,1\}$ violating Claim \ref{ForEachISideTriang}. Thus, $|V(G^{P_k})|=4$. Let $p_kz_1z_2$ be as in Subclaim \ref{EitherAOrBUniversality}, where $z_2=x_k$. Consider the following cases.

\textbf{Case 1:} $L(p_k)\neq L(z_2)$

In this case, by \ref{BWheel3Lb} of Theorem \ref{BWheelMainRevListThm2}, there is an $a\in L(p_k)$ which is almost $(P^k, G^{P_k})$-universal. Let $\sigma\in\mathcal{D}$ with $\sigma(p_k)=a$. Since $K_1$ is a triangle, and $L(x_0)=L(x_1)$, $\sigma$ extends to an $L$-coloring $\tau$ of $\{p_0, v_0, v_1, p_1\}$ with $\tau(v_1)\not\in L(x_0)\cup L(x_1)$ and $\tau(v_0)\not\in L(x_0)$. But then $\tau$ is $(P,G)$-sufficient, contradicting our choice of $\sigma$. 

\textbf{Case 2:} $L(p_k)=L(z_2)$

In this case, $G^{P_{1-k}}$ is a triangle, and $C\setminus\mathring{P}$ is a path of length four in which all vertices have the same 3-list $T$. If $k=1$, then any $L$-coloring of $\{p_0, p_1\}$ extends to an element $\tau$ of $\textnormal{Crown}(P,G)$ where $\textnormal{dom}(\tau)=V(C)\setminus\{q_0, q_1, z_1\}$ and $\tau$ uses the same color on $p_1, z_2$. This contradicts our assumption that $\mathcal{G}$ is a counterexample, so $k=0$. If there is an $a\in T\setminus L(v_1)$, then, letting $\sigma\in\mathcal{D}$ with $\sigma(p_0)=a$, we again extend $\sigma$ to an element of $\textnormal{Crown}(P,G)$ with domain $V(C)\setminus\{q_0, q_1, z_1\}$ by coloring $z_2$ with $a$, contradicting our choice of $\sigma$. Thus, $T\subseteq L(v_1)$. Fix a $\sigma\in\mathcal{D}$. Let $\sigma(p_0)=a$ and $\sigma(p_1)=b$. If $a=b$, then, since $|L(v_1)\setminus T|\geq 2$, we can extend $\sigma$ to a $(P,G)$-sufficient $L$-coloring $\tau$ of $\{p_0, z_2, v_1, v_2, p_1\}$ with $\tau(z_2)=a$ and $\tau(v_1)\in L(v_1)\setminus T$, and $\tau\in\textnormal{Crown}(P,G)$, contradicting our choice of $\sigma$. Thus $a\neq b$, and we extend $\sigma$ to an $L$-coloring $\tau$ of $V(C)\setminus\{q_0, q_1, z_1\}$ with $\tau(z_2)=a$ and $\tau(v_1)=b$. Again, $\tau\in\textnormal{Crown}(P,G)$, contradicting our choice of $\sigma$. \end{claimproof}

As $L(y)=L(x_1)$, it follows from Claim \ref{ForEachISideTriang} that $|V(G)|=8$ and $C\setminus\mathring{P}=p_0x_0x_1p_1$, where all the vertices of $C\setminus\mathring{P}$ have a common $3$-list $T$. 

\begin{claim}\label{TContLV0LV1Last} $T\subseteq L(v_0)\cap L(v_1)$. \end{claim}

\begin{claimproof}  Suppose $T\not\subseteq L(v_1)$ and let $a\in T$ with $a\not\in L(v_1)$. Now, there is a $j\in\{0,1\}$ and a $\sigma\in\mathcal{D}_j$ with $\sigma(p_0)\neq a$. On the other hand, $\sigma$ extends to an $L$-coloring $\sigma'$ of $\{p_0, x_0, x_1, p_1\}$ with $\sigma'(x_0)=a$. Since $\sigma\in\mathcal{D}_j$, we have $|L_{\sigma'}(v_1)|=3$, so $a\in L(v_1)$, a contradiction. Thus, $T\subseteq L(v_1)$. Now suppose that $T\not\subseteq L(v_0)$ and let $b\in T\setminus L(v_0)$. Now, there is a $\sigma\in\mathcal{D}$ with $\sigma(p_1)=b$. Since $b\in L(v_1)\setminus L(v_0)$, it follows that $\sigma$ extends to an $L$-coloring $\tau$ of $V(C)\setminus\{q_0, q_1\}$ with $\tau(v_1)=b$ and $|L_{\tau}(q_0)|\geq 3$. We have $|L_{\tau}(q_1)|\geq 3$ as well, contradicting our choice of $\sigma$. \end{claimproof}

By Claim \ref{TContLV0LV1Last}, any $L$-coloring $\sigma$ of $\{p_0, p_1\}$ with $\sigma(p_0)\neq\sigma(p_1)$ extends to an $L$-coloring $\tau$ of $\{p_0, x_0, v_0, v_1, p_1\}$ with $\tau(v_0)=\sigma(p_0)$ and $\tau(x_0)=\sigma(p_1)$ and $\tau(v_1)\not\in T$, so $\tau$ is $(P,G)$-sufficient and $\tau\in\textnormal{Crown}(P,G)$. Thus, for each $\sigma\in\mathcal{D}_0\cup\mathcal{D}_1$, $\sigma(p_0)=\sigma(p_1)$. Let $j\in\{0,1\}$ with $\mathcal{D}_j\neq\varnothing$. Fix a$\sigma\in\mathcal{D}_j$ and let $a:=\sigma(p_0)$. Then $\sigma$ extends to an $L$-coloring $\tau$ of $V(C)\setminus\{q_0, q_1, v_{1-j}\}$ with $\tau(v_j)=a$. Each of $q_0, q_1, v_{1-j}$ has an $L_{\tau}$-list of size at least three, so $\sigma\not\in\mathcal{D}_j$, a contradiction.  \end{proof}

 \section*{Acknowledgements}

The author would like to thank Bruce Richter for his detailed proofreading (both for mathematical correctness and for typos) and formatting suggestions

\end{document}